\documentclass[12pt]{amsart}

\usepackage{graphicx}
\usepackage{newtxtext}
\usepackage{newtxmath}
\usepackage{natbib}
\usepackage{hyperref}

\usepackage{amsfonts}
\usepackage{mathtools}
\usepackage[mathscr]{eucal}
\usepackage[]{hyperref}
\usepackage{setspace,url}
\usepackage{tikz}
\usepackage{graphicx,epsfig,subfig}

\hypersetup{
    colorlinks = true,
    urlcolor   = blue,
    citecolor  = black,
}

\newtheorem{lemma}{Lemma}[section]

\newtheorem{proposition}{Proposition}[section]
\newtheorem{remark}{Remark}[section]

\numberwithin{equation}{section}

\renewcommand{\H}{\mathcal{H}}

\newcommand{\reals}{\mathbb{R}}

\renewcommand{\Re}{\text{Re }}

\newcommand{\prtl}{\ensuremath{\partial}}

\newcommand{\veps}{\ensuremath{\varepsilon}}
\newcommand{\calH}{\ensuremath{\mathcal{H}}} 
\newcommand{\calO}{\ensuremath{\mathcal{O}}} 

\newcommand{\R}{\mathbb{R}}

\renewcommand{\sp}{\omega_{2k_0}}
\newcommand{\sq}{\omega_{0}}

\newcommand{\bu}{{\bf u}}

\newcommand{\sgn}{{\rm sgn}}

\newcommand{\Hil}{{\rm H}}

{\begin{trivlist} \item[]{\textbf{Proof} }}%
{\hspace*{\fill}$\rule{.3\baselineskip}{.35\baselineskip}$\end{trivlist}}

\newcommand{\RomanNumeralCaps}[1]
\linenumbers

\begin{document}

\title{ A Hamiltonian Dysthe equation for  deep-water gravity waves with constant vorticity}

\author{Philippe Guyenne}
\author{Adilbek Kairzhan}
\author{Catherine Sulem}
\address[Guyenne]{Department of Mathematical Sciences, University of Delaware, Newark, DE 19716, USA; guyenne@udel.edu }
\address[Kairzhan]{Department of Mathematics, University of Toronto, Toronto, Ontario M5S 2E4, Canada; 
kairzhan@math.toronto.edu }
\address[Sulem]{Department of Mathematics, University of Toronto, Toronto, Ontario M5S 2E4, Canada ; 
sulem@math.utoronto.ca}

\maketitle

\begin{abstract}
This paper is a study of the water wave problem in a two-dimensional domain of infinite depth in the presence of nonzero constant vorticity. A goal is to describe the effects of uniform shear flow on the modulation of weakly nonlinear quasi-monochromatic surface gravity waves. Starting from the Hamiltonian formulation of this problem and using techniques from Hamiltonian transformation theory, we derive a Hamiltonian Dysthe equation for the time evolution of the wave envelope. Consistent with previous studies, we observe that the uniform shear flow tends to enhance or weaken the modulational instability of Stokes waves depending on its direction and strength. Our method also provides a non-perturbative procedure to reconstruct the surface elevation from the wave envelope, based on the Birkhoff normal form transformation to eliminate all non-resonant triads. This model is tested against direct numerical simulations of the full Euler equations and against a related Dysthe equation recently derived by \cite{CCK18} in the context of constant vorticity. Very good agreement is found for a range of values of the vorticity.
\end{abstract}



\section{Introduction}

The water wave problem refers to the motion of a free surface over a body of water of finite of infinite depth. 
Its classical formulation usually assumes that the fluid is inviscid and irrotational. It is well known since the seminal work of \cite{Z68}
that, in this setting, the water wave equations can be written as a Hamiltonian system with a standard Darboux symplectic structure 
whose Hamiltonian is the total energy. The canonical conjugate variables are given by $(\eta,\xi)$,  
where $\eta(x,t)$ is the surface elevation and $\xi(x,t)$ denotes the boundary values of the velocity potential on the free surface.
With introduction of the Dirichlet--Neumann operator $G(\eta) $ that maps the Dirichlet boundary condition for a harmonic function 
in the fluid domain to its Neumann boundary condition, this Hamiltonian takes an explicit lower-dimensional form 
\[
H(\eta,\xi) = \frac{1}{2} \int_{\R} \Big[ \xi G(\eta)\xi + g \eta^2 \Big] dx \,,
\]
in terms of surface variables alone \citep{CS93}.
Moreover, the operator $G(\eta)$ is analytic with respect to $\eta$ for Lipschitz functions $\eta$ 
and this provides an expansion of the Hamiltonian near the stationary solution $w = (\eta,\xi) = 0$ 
which corresponds to a fluid at rest \citep{CM85}. The solution $w=0$ is an elliptic stationary point in dynamical systems terms.
In this Hamiltonian framework, perturbation calculations can be performed following general rules from Hamiltonian transformation theory,  
including canonical transformations and reduction to normal forms, to derive asymptotic models for weakly nonlinear water waves
while preserving the Hamiltonian character of the original system (see \cite{CGS21b} for a review). 

Recently, a number of theoretical investigations have been devoted to the water wave problem with nonzero vorticity due to 
its relevance to oceanography and coastal engineering where the influence of currents on waves may play an important role \citep{C01,SBSBV19}. 
Unlike the irrotational case, a full-dimensional computation is in general required to solve for the vorticity field.
This has prompted efforts to propose simplified models for rotational waves in the long-wave shallow-water regime \citep{CL14,RG15}.
Of special interest is the case of nonzero constant vorticity which corresponds to vertically shear flow with a linear profile.
The direction of the underlying current can be that of wave propagation (co-propagating) or opposite (counter-propagating).
Similar to the irrotational water wave problem, this particular case allows for a lower-dimensional reformulation of the governing equations.
As a consequence, it has received much attention in both mathematical and numerical studies,
regarding e.g. the existence and stability of steadily progressing wave solutions \citep{BP22,DH19,HW20,SMKD17,TP88,V96},
the flow structure beneath waves \citep{RMN17}, or the focusing of transient waves by an adverse current \citep{C09,MC15}.

Furthermore, as an extension of Zakharov's idea, a Hamiltonian formulation for nonlinear water waves 
with constant vorticity has been derived by \cite{W07} (see also \cite{CIP08}).
The associated symplectic structure in terms of $(\eta,\xi)$ is not canonical but a change of variables reduces the system to a canonical one.   
Based on this formulation, recent work has been conducted involving long-wave modeling in the Korteweg--de Vries regime \citep{W08}, 
rigorous mathematical analysis on the existence of quasi-periodic traveling wave solutions \citep{BFM21},
and direct numerical simulation of unsteady waves on deep or shallow water \citep{G17,G18}.
In particular, the reduction to surface variables makes it possible to solve the full equations
via efficient and accurate numerical solvers such as the boundary integral method, the conformal mapping technique
or the high-order spectral method.

This article is devoted to the effects of constant vorticity in the setting of weakly nonlinear surface gravity waves 
for which modulation theory is a classical  tool. 
Under consideration is the asymptotic scaling regime where approximate solutions are constructed as 
slow modulations of monochromatic waves in two space dimensions. 
For this problem, \cite{TKM12} derived a cubic nonlinear Schr\"odinger (NLS) equation governing the envelope of surface gravity waves 
on finite depth using the method of multiple scales. Their analysis was extended to gravity-capillary waves by \cite{HKAC18}
and to hydroelastic waves by \cite{GWM19}.
Similarly, \cite{CCK18} carried out the perturbation calculations up to an order higher and obtained a Dysthe equation
for gravity-capillary waves with constant vorticity on infinite depth. We also point out the earlier work of \cite{B98} who proposed an NLS equation for gravity waves in the presence of linear shear confined to a finite-depth layer near the free surface.

The classical theory for modulational analysis which describes the long-time evolution and stability of fast oscillatory solutions 
of nonlinear dispersive partial differential equations (PDEs) is based on the so-called modulational Ansatz for solutions 
in the form of a weakly nonlinear narrowband wave train. In the context of gravity water waves, 
the leading nontrivial terms give rise to the NLS equation for the evolution of the slowly varying wave envelope 
(see \cite{Z68} for the original derivation in the irrotational case). 
\cite{D79} later proposed a higher-order approximation for waves on deep water, 
which has since been extended to many other settings. 
The Dysthe equation and its variants are widely used in the oceanographic community because of their efficiency and ability
to describe waves of moderately large steepness. It has been observed that numerical solutions of the Dysthe model 
provide a better agreement with laboratory experiments than NLS and are able e.g. 
to emulate the asymmetry of propagating wave packets, a feature not captured by NLS \citep{GKSX21,LM85}.

Unlike the NLS equation which is a canonical Hamiltonian PDE, earlier versions of the  Dysthe equation  
do not preserve the Hamiltonian character of the primitive equations.  
\cite{GT11} used a Zakharov's four-wave interaction model obtained by \cite{K94} 
to derive a Hamiltonian version of Dysthe's equation for three-dimensional gravity waves on finite depth.  
\cite{CGS21a} and \cite{GKS22} considered the modulational regime for the two- and three-dimensional problem of gravity waves 
on deep water respectively, and derived the corresponding Dysthe equations directly from the Euler equations for potential flow 
through a sequence of canonical transformations.
By construction, the resulting Dysthe equations preserve the Hamiltonian character of the original problem.
The main objective of the present paper is to extend this approach to the modulation of weakly nonlinear wave trains 
in the presence of constant vorticity and derive a Dysthe equation that conforms with 
the Hamiltonian nature of the water wave system in this setting. 
We focus on the two-dimensional problem of wave propagation over infinite depth.

From a modeling viewpoint, it is desirable that such important structural properties as energy conservation 
are inherited by the approximation. Aside from interest in the Dysthe equation as an asymptotic model for water wave applications, 
this question is particularly relevant considering the successful and widespread use of the Hamiltonian formalism 
in physical sciences, including the field of fluid mechanics and free-surface flows \citep{BO82,K94}.
The associated mathematical tools as applied to this physical problem are thus also of interest.
Throughout our derivation, care is taken to perform both the expansion of the Hamiltonian functional 
and the canonical change of symplectic structure in a systematic manner,
starting from the basic formulation introduced by \cite{W07}.
Our calculations are facilitated by the fact that the Dirichlet--Neumann operator admits an explicit Taylor series expansion \citep{CS93}.
This property has been extensively used in previous studies on irrotational water waves \citep{L13} and can also be exploited here.

In modulation theory, reliance on a modulational Ansatz implies that an additional step is required to reconstruct
physical quantities such as the surface elevation from the solution of the envelope equation.
This step is determined as part of the asymptotic analysis and, in classical approaches like the method of multiple scales,
this reconstruction is typically carried out perturbatively via a Stokes-type expansion.
It is an important computation that also influences the model's performance, 
as shown e.g. in comparisons with experimental data on extreme waves \citep{ZGO15}.
In the present context, the reconstruction procedure is associated with
the Birkhoff normal form transformation to eliminate all non-resonant triads.
It is a non-perturbative computation but requires solving an auxiliary system of PDEs for the surface variables.
In this way, higher-order harmonic contributions to the wave spectrum are automatically generated from
the carrier wave component through nonlinear interactions according to these PDEs.
We provide a detailed derivation of this auxiliary system for nonzero constant vorticity,
noting that the resulting equations are significantly more complicated than in the irrotational case.
By definition, these equations also take the form of a Hamiltonian evolutionary system with a canonical symplectic structure.
As a consequence, the entire solution process fits within a Hamiltonian framework.
Based on this approximation, we conduct a linear stability analysis and examine the dependence on vorticity.
We then test analytical and numerical predictions from our model against direct numerical simulations of the full system,
by inspecting the time evolution of perturbed Stokes waves and their possible instability.
We also compare our results to numerical solutions of the related Dysthe model by \cite{CCK18}
and discuss the performance of the different reconstruction methods.
Previous observations on the focusing (resp. defocusing) effects of negative (resp. positive) vorticity
associated with a counter-propagating (resp. co-propagating) current are recovered.

The starting point of our approach is the water wave system in its Hamiltonian canonical form 
in terms of the surface elevation and a modified velocity potential. In Section \ref{sect2},
the Taylor expansion of the Hamiltonian is presented  and then expressed in terms of the complex symplectic coordinates 
that diagonalize its quadratic terms. Unlike the irrotational case, the linear dispersion relation is not an even function of wavenumber. 
In Section 3, we introduce elements of transformation theory and Birkhoff normal form transformations. 
A third-order Birkhoff normal form transformation that eliminates all non-resonant cubic terms from the Hamiltonian is explicitly calculated. 
It is defined as an auxiliary Hamiltonian flow from the original variables to transformed variables. 
In Section 4, we propose the new Hamiltonian truncated at fourth order
and, in Section 5, we use the modulational Ansatz together with a homogenization technique to derive the resonant quartic contributions.  
In Section 6, the Hamiltonian Dysthe equation is obtained for the wave envelope. 
Section 7 is devoted to a validation of our approximation through numerical simulation and comparison with other models.
A theoretical prediction on modulational stability of Stokes waves in the presence of constant vorticity is first established.  
We then explain the procedure to reconstruct the surface variables by inverting the third-order Birkhoff normal form transformation,
and we describe the numerical methods to solve the equations involved in the various models.
Finally, we present a numerical investigation where predictions from our Hamiltonian Dysthe equation are compared to those from
the envelope model by \cite{CCK18} and to direct simulations of the full water wave system.

\section{The water wave system} \label{sect2}

\subsection{Governing equations}
We consider the evolution of a free surface $\{ y = \eta(x,t) \}$ on top of a two-dimensional fluid of infinite depth
\[
\mathcal{S}(\eta) = \left\{ x \in \R, -\infty < y < \eta(x,t) \right\} \,,
\]
under the influence of gravity.   Assuming the  flow is  incompressible and  inviscid, the velocity field, denoted by $\bu(x,y,t) = (u(x,y,t), v(x,y,t))^\top$,
satisfies the Euler equations
\begin{equation}
\label{euler-eqns}
\prtl_t \bu + (\bu \cdot \nabla) \bu + \frac{1}{\rho} \nabla{P} - {\bf g} = {\bf 0} \,, 
\end{equation}
\[
\nabla \cdot \bu = 0 \,,
\]
where $\rho$ is the (constant) fluid density, $P(x,y,t)$ is the pressure and ${\bf g} = (0, -g)^\top$ is the acceleration due to gravity.
Hereafter, the symbol $\nabla$ denotes the spatial gradient $(\partial_x,\partial_y)^\top$ when applied to functions
or the variational gradient when applied to functionals.
The vorticity defined as 
\[
\gamma = \prtl_x v - \prtl_y u \,,
\]
satisfies
\[
\partial_t \gamma +(\bu \cdot  \nabla ) \gamma=0 \,,
\]
which expresses the  well-known fact  that, in two dimensions, the vorticity is conserved along particle trajectories.
In particular, if the initial vorticity is  constant throughout the fluid domain,  it remains constant. 
The present study is devoted to flows with the property of having nonzero constant vorticity. 

The boundary conditions at the free surface are composed of the dynamical condition
\[
P = P_0 \,,
\]
where $P_0$ is the atmospheric constant, and the kinematic condition
\begin{equation*} \label{kinematic}
v= \partial_t \eta + u \partial_x \eta \,.
\end{equation*}
The divergence-free condition implies the existence of a stream function $\psi$ such that
\[
u = \partial_y\psi \,, \quad v= -\partial_x \psi \,,
\]
satisfying
\[
-\Delta \psi = \gamma \,.
\]
By construction, 
$\widetilde \psi = \psi + \gamma y^2/2$
is a harmonic function.
Introducing the generalized potential $\varphi$ defined as the harmonic conjugate of $\widetilde \psi$, we have
\begin{align*}
&\partial_x \varphi = \partial_y \widetilde\psi = u + \gamma y \,, \\
&\partial_y \varphi = -\partial_x \widetilde\psi = v \,.
\end{align*}
The presence of constant vorticity induces a horizontal background shear current that has a linear profile in the vertical direction. 
We associate $\gamma > 0$ with a co-propagating current because it contributes positively to the horizontal fluid velocity
$u = \partial_x \varphi - \gamma y$ for $y < 0$ (along most of the water column),
while $\gamma < 0$ is associated with a counter-propagating current.

In the variables $\varphi$ and $\widetilde \psi$, Eq. \eqref{euler-eqns} takes the form
\[
\nabla\Big[ \partial_t \varphi + \frac{1}{2} \Big( (\partial_x \varphi)^2 + (\partial_y \varphi)^2 \Big) +\gamma \widetilde \psi  
- \gamma y \partial_x \varphi + P + g y \Big] = \bf{0} \,,
\]
and after integration, it reduces to 
\[
\partial_t \varphi + \frac{1}{2} \Big( (\partial_x \varphi)^2 + (\partial_y \varphi)^2 \Big) + \gamma \widetilde \psi  
- \gamma \eta \partial_x \varphi + g \eta = 0 \,,
\]
at the free surface $y=\eta$, with $P_0 = 0$ without loss of generality. 
The Euler system can then be written as
\begin{align}
&\Delta \varphi = 0 \quad {\textrm{in} } \;\;\; \mathcal{S} ~, \label{e1}\\
& \partial_t \eta - \partial_y \varphi + (\partial_x \varphi) (\partial_x \eta) - \gamma \eta \partial_x \eta = 0 \,, 
 \quad {\rm{on}} \ y=\eta(x,t) \,,   \label{e2}\\
& \partial_t \varphi + \frac{1}{2} \Big( (\partial_x \varphi)^2 + (\partial_y \varphi)^2 \Big) + \gamma \widetilde \psi  
- \gamma \eta \partial_x \varphi + g\eta = 0 \,,  \quad {\rm{on}} \ y=\eta(x,t) \,, \label{e3}
 \end{align}
with the condition that $\varphi, \widetilde \psi \to 0 $ uniformly in $x$ as $y\to -\infty$.

\subsection{Hamiltonian formulation}

It is well known since the seminal paper of \cite{Z68} that the irrotational ($\gamma=0$) water wave system has a Hamiltonian formulation in the variables $(\eta, \xi) $ where $ \xi(x,t) = \varphi(x,\eta(x,t),t)$ is the trace of the velocity potential on the free surface.

\cite{W07} (see also \cite{CIP08}) observed that, in the presence of constant vorticity, the water wave system in $(\eta, \xi)$ can  still be expressed as a Hamiltonian system but in non-canonical form, namely 
\begin{equation}
\label{ww-nearly-hamiltonian-equation}
\partial_t \begin{pmatrix}
\eta \\ \xi
\end{pmatrix} =  
\begin{pmatrix}
0 & 1 \\ -1 & \gamma \partial_x^{-1}
\end{pmatrix} \begin{pmatrix}
\partial_\eta H \\ \partial_\xi H
\end{pmatrix} \,.
\end{equation}
The Hamiltonian $H(\eta, \xi)$ is the total energy 
\begin{equation*}
\label{hamiltonian-initial}
H (\eta, \xi) = \frac{1}{2} \int_\reals \left[ \xi G(\eta) \xi - \gamma \eta^2 \partial_x \xi + \frac{\gamma^2}{3} \eta^3 + g \eta^2 \right] dx \,,
\end{equation*}
where $G(\eta)$ is the Dirichlet--Neumann operator (DNO) in the fluid domain, which associates to the Dirichlet data $\xi$ 
on $y=\eta(x)$ the normal derivative of the harmonic function $\varphi$ with an additional normalized factor, namely
$$
G(\eta) : \xi \longmapsto \sqrt{1+(\partial_x \eta)^2} \, \partial_{n} \varphi \big|_{y=\eta} \,.
$$
Other invariants of motion are the volume (or mass)
$$V= \int_\reals \eta dx \,,$$ 
and the momentum (or impulse)
\begin{equation} \label{moment}
I = \int_\reals \Big( \eta \partial_x \xi - \frac{1}{2} \gamma \eta^2 \Big) dx \,.
\end{equation}
\cite{W07} found that, under the change of variables 
\begin{equation} \label{xi-zeta-relation}
(\eta, \xi) \to \Big( \eta, \zeta = \xi-\frac{\gamma}{2} \partial_x^{-1} \eta \Big) \,,
\end{equation}
the above system \eqref{ww-nearly-hamiltonian-equation} can be transformed into canonical form 
\begin{equation}
\label{ww-hamiltonian-equation}
\partial_t \begin{pmatrix}
\eta \\ \zeta
\end{pmatrix} =  
J \, \nabla \H(\eta,\zeta)  =
\begin{pmatrix}
0 & 1 \\ -1 & 0
\end{pmatrix} \begin{pmatrix}
\partial_\eta \H \\ \partial_\zeta \H
\end{pmatrix} \,,
\end{equation}
where the Hamiltonian $\H(\eta, \zeta) = H(\eta, \xi)$ in the new variables is 
\begin{equation}
\label{hamiltonian}
\H (\eta, \zeta) = \frac{1}{2} \int_\reals \left[ \Big( \zeta + \frac{\gamma}{2} \partial_x^{-1} \eta \Big) G(\eta) \Big( \zeta + \frac{\gamma}{2} \partial_x^{-1} \eta \Big) - \gamma \eta^2  \partial_x\zeta - \frac{\gamma^2}{6} \eta^3 + g \eta^2 \right] dx \,.
\end{equation}
In the above formulas, $\partial_x^{-1} \eta$ is defined as
$\partial_x^{-1} \eta(x) = \int_{-\infty}^x \eta(s)ds.$
We assume that $\eta \to 0$ as $ x\to \pm \infty$. 
Furthermore, we assume that $\int_{-\infty}^{\infty}  \eta dx =0$ at $t=0$, which implies that it remains so for all $t$.
Indeed, integrating \eqref{e2} in $x$ over $\reals$, we get
$$\partial_t \int_\reals \eta dx - \int_\reals \Big( G(\eta) \xi - \frac{ \gamma }{2} \partial_x \eta^2 \Big) dx = 0 \,.$$
The  second term on the left-hand side rewrites as
$ \int_\reals G(\eta) \xi dx =  \int_{\partial\mathcal{S}} \partial_n \varphi d \sigma = \int_{\mathcal{S}} \Delta \varphi dx dy = 0$,
leading to $\int_\reals \eta dx = 0$ for all $t$.

Denoting  $\widehat f_k = \frac{1}{\sqrt{2 \pi}} \int_\reals e^{-ikx} f(x) dx$, the Fourier transform of a function $f(x)$,
we have in particular
\begin{equation*}
\partial_x^{-1} \eta(x) = - \frac{i}{\sqrt{2 \pi}} \int_\reals \frac{1}{k} \widehat \eta_k e^{ikx} dk \,, \quad \partial_x \zeta(x) = \frac{i}{\sqrt{2 \pi}} \int_\reals k \widehat \zeta_k e^{ikx} dk \,.
\end{equation*}
From the above assumption,
\begin{equation}
\label{zero-mass}
\widehat \eta_0 = 0 \,, \quad \widehat \eta_k = \mathcal{O}(k) \quad \text{for small} \, k \,.
\end{equation}
In the following, we will drop the hat notation. 
Since $\eta(x)$ and $\zeta(x)$ are real-valued functions, we have the relation $(\eta_{-k}, \zeta_{-k}) = (\overline \eta_k, \overline  \zeta_k)$
where the overbar stands for complex conjugation.


\subsection{Taylor expansion of the Hamiltonian near equilibrium}

It can be shown that the DNO is analytic in $\eta$ \citep{CM85} and admits a convergent Taylor series expansion
\begin{equation} \label{series}
G(\eta) = \sum_{m=0}^\infty G^{(m)}(\eta) \,,
\end{equation}
about the quiescent state $\eta=0$. For each $m$, the term $G^{(m)}(\eta)$ is homogeneous of degree $m$ in $\eta$ and can be calculated explicitly via  recursive relations \citep{CS93}.  Denoting  $D = -i \, \partial_x$, the first three terms are
\begin{equation}
\label{g-012-recursive}
\left\{ \begin{array}{l}
G^{(0)}(\eta)  = |D| \,, \\
G^{(1)}(\eta) = D \eta D - G^{(0)}\eta G^{(0)} \,, \\
G^{(2)}(\eta) = -\frac{1}{2} \left( |D|^2 \eta^2 G^{(0)}+ G^{(0)} \eta^2 |D|^2 - 2G^{(0)} \eta G^{(0)} \eta G^{(0)} \right) \,.
\end{array} \right.
\end{equation}
In Fourier variables, 
substituting the expansion for $G(\eta)$ into the Hamiltonian (\ref{hamiltonian}), we get
\begin{equation}
\label{HH}
\H = \H^{(2)} + \H^{(3)}+ \H^{(4)} + \dots \,,
\end{equation}
where each term $\H^{(m)}$ is homogeneous of degree $m$ in $(\eta,\zeta)$. In particular, we have
{\small
\begin{equation}
\label{homog-ham-initial}
\begin{aligned}
\H^{(2)} & = \frac{1}{2} \int_\reals \left[ |k| \Big( \zeta_{k} -  \frac{i \gamma}{2k} \eta_{k} \Big)
\Big( \overline \zeta_{k} + \frac{i \gamma}{2k} \overline \eta_{k} \Big) + g |\eta_{k}|^2 \right] dk \,, \\[3pt]
\H^{(3)} & = \frac{1}{2 \sqrt{2\pi}} \int_{\reals^3} \left[ (-k_1 k_3 - |k_1| |k_3|) \Big( \zeta_1 - \frac{i \gamma}{2k_1} \eta_1 \Big) \eta_2 \Big( \zeta_3 - \frac{i \gamma}{2k_3} \eta_3 \Big) - i \gamma k_1 \zeta_1 \eta_2 \eta_3 
 - \frac{\gamma^2}{6} \eta_1 \eta_2 \eta_3 \right] \delta_{123} dk_{123} \,, \\[3pt]
\H^{(4)} & = -\frac{1}{8 \pi} \int_{\reals^4} |k_1||k_4|(|k_1|+|k_4|-2|k_3+k_4|) \Big( \zeta_1 - \frac{i \gamma}{2k_1} \eta_1 \Big) \eta_2 \eta_3 \Big( \zeta_4 - \frac{i \gamma}{2k_4} \eta_4 \Big) \delta_{1234} dk_{1234} \,.
\end{aligned}
\end{equation}
}
In the above expressions, we use the compact notations
$(\eta_j, \zeta_j)= (\eta_{k_j}, \zeta_{k_j})$, $dk_{1 \dots n} = dk_1 \dots dk_n$, and $\delta_{1 \dots n} = \delta(k_1+ \dots + k_n)$ where $\delta(k) = \frac{1}{2\pi} \int_\reals e^{-ikx} dx$ is the Dirac distribution.
Hereafter, the domain of integration is omitted in integrals and is understood to be $\reals$ for each $x_j$ or $k_j$.

\subsection{Linearization near equilibrium}

The linearized water wave system about still water,  written in terms of $(\eta,\zeta)$, is
\begin{equation*}
\label{ww-hamiltonian-equation-lin}
\partial_t \begin{pmatrix}
\eta\\ \zeta
\end{pmatrix} =  
J \, \nabla \H^{(2)}(\eta,\zeta)  =
\begin{pmatrix} 
g- \frac{\gamma^2}{4} \partial_x^{-1} |D| \partial_x^{-1}  & -\frac{\gamma}{2}\partial_x^{-1}  |D| \\
\frac{\gamma}{2}\partial_x^{-1}  |D| & |D|
\end{pmatrix}
\begin{pmatrix}
\eta\\ \zeta
\end{pmatrix} \,.
\end{equation*}
We now introduce the symplectic complex coordinate
\begin{equation} \label{zeta-variable}
z := \frac{1}{\sqrt 2} \big( a(D)  \eta + i \, a(D) ^{-1} \zeta \big) \,,
\end{equation}
where $a^2(D)  := \omega(D)/|D|$ and $\omega(D)  := \sqrt{\gamma^2/4 + g|D|}$. The mapping $ (\eta,\zeta) \to (z,\bar z)$  is canonical, and in these variables, the water wave system becomes
\begin{equation}
\label{ww-hamiltonian-in-z}
\partial_t \begin{pmatrix}
z \\ \overline z
\end{pmatrix} = 
\begin{pmatrix}
0 & -i \\ i & 0
\end{pmatrix} \begin{pmatrix}
\partial_{z} \H \\ \partial_{\overline z} \H
\end{pmatrix} \, := J_1 \begin{pmatrix}
\partial_{z} \H \\ \partial_{\overline z} \H
\end{pmatrix} \,.
\end{equation}
Equivalently, in the Fourier space, 
\begin{equation}
\label{zeta-variable-fourier}
z_k := \frac{1}{\sqrt 2} \big( a_k \eta_k + i \, a_k^{-1} \zeta_k \big) \,,
\end{equation}
where $a_k^2 = \omega_k /|k|$ and $\omega_k= \sqrt{\gamma^2/4 + g|k|}$. 
Since the functions $\eta(x)$ and $\zeta(x)$ are real-valued in the physical space, we also have
\begin{equation}
\label{zeta-bar-variable-fourier}
\overline z_{-k} = \frac{1}{\sqrt 2} \big( a_k \eta_k - i \, a_k^{-1} \zeta_k \big) \,.
\end{equation}
We can express $\eta_k$ and $\zeta_k$ in terms of $z_k$ as
\begin{equation}
\label{eta-xi-via-zeta}
\eta_k = \frac{1}{\sqrt 2} a_k^{-1} (z_k + \overline z_{-k}), \quad \zeta_k = \frac{1}{i \sqrt{2}} a_k (z_k - \overline z_{-k}) ~.
\end{equation}
The quadratic term $\H^{(2)}$ given in (\ref{homog-ham-initial}) diagonalizes as  
\begin{equation}
\label{H2}
\H^{(2)} = \int
\Omega_k |z_k|^2 dk \,,
\end{equation}
where 
\begin{equation*}
\label{dispersive-relation}
\Omega_k = \Omega(k) = \frac{\gamma}{2} \sgn(k) + \omega_k \,,
\end{equation*}
is the linear dispersion relation for deep-water gravity waves with constant vorticity \citep{BFM21}.
The linearized system with $\H$ replaced by $\H^{(2)}$ in \eqref{ww-hamiltonian-in-z} 
reduces to the scalar equation $\partial_t z_k = -i \, \Omega_k z_k$.

We define the Poisson Bracket of  two functionals $K(\eta, \zeta)$ and $H(\eta, \zeta)$
of real-valued functions $\eta$ and $\zeta$ as
\begin{equation}
\label{poisson-bracket}
\{K, H\} = \int \big( \partial_\eta H \partial_\zeta K - \partial_\zeta H \partial_\eta K \big) dx \,.
\end{equation}
Assuming that $K$ and $H$ are real-valued, and using the Plancherel formula, it is written in terms of the Fourier variables as
$$
\begin{aligned}
&\{K, H\}  = \int \big( \partial_{\eta_k} H \overline{\partial_{\zeta_k} K} - \partial_{\zeta_k} H \overline{\partial_{\eta_k} K} \big) dk = \int \big( \partial_{\eta_k} H \partial_{\overline{\zeta}_k} K - \partial_{\zeta_k} H \partial_{\overline{\eta}_k} K \big) dk \\ 
& = \int \big( \partial_{\eta_k} H \delta_{\zeta_{-k}} K - \partial_{\zeta_k} H \partial_{\eta_{-k}} K \big) dk   
= \int \big( \partial_{\eta_{k_1}} H \partial_{\zeta_{k_2}} K - \partial_{\zeta_{k_1}} H \partial_{\eta_{k_2}} K \big) \delta_{12} dk_{12} \,.
\end{aligned}
$$
In terms of the complex symplectic coordinates, it becomes
\begin{equation}
\label{poisson-bracket-z}
\{K, H\} = \frac{1}{i} \int (\partial_{z_{k_1}} H \partial_{\overline z_{-k_2}} K - \partial_{\overline z_{- k_1}} H \partial_{z_{k_2}} K) \delta_{12} dk_{12} \,.
\end{equation}

\section{Birkhoff normal form transformations}

\subsection{Third-order term in the Hamiltonian}

The cubic term  $\H^{(3)}$ given in (\ref{homog-ham-initial})  can be further simplified due to symmetries.
We first state a simple  identity that will be useful throughout the paper.

\begin{lemma}
For any $(k_1, k_2, k_3) \in \mathbb{R}^3$ with $k_1+k_2+k_3 = 0$ and $k_j \ne 0$, we have
\begin{equation}
\label{sign-identity}
\sgn(k_1) \sgn(k_2) + \sgn(k_1) \sgn(k_3) + \sgn(k_2) \sgn(k_3) = -1 \,.
\end{equation}
\end{lemma}

\begin{proof}
Consider the different sectors of the $(k_1,k_2)$-plane and check that the equality is satisfied in each sector.
\end{proof}

\begin{lemma}
\label{lemma-H3-simplified}
The cubic term  $\H^{(3)}$   in  (\ref{homog-ham-initial}) can be written as 
\begin{equation}
\label{H3-in-eta-zeta}
\H^{(3)} = -\frac{1}{2\sqrt{2 \pi}} \int (1+\sgn(k_1) \sgn(k_3)) \Big( |k_1| |k_3| \zeta_1 \eta_2 \zeta_3 + i \frac{\gamma k_2}{2} \eta_1 \zeta_2 \eta_3 \Big) \delta_{123} dk_{123} \,.
\end{equation}
\end{lemma}

\begin{proof}
Expanding the  brackets in \eqref{homog-ham-initial}, we regroup the terms as follows 
\begin{equation}
\label{H3-decomposed}
\begin{aligned}
\H^{(3)} &= \frac{1}{2\sqrt{2\pi}} \int  \Big[ (-k_1k_3-|k_1||k_3|) \zeta_1 \eta_2 \zeta_3 
 + (k_1k_3+|k_1||k_3|)  \frac{i \gamma}{k_1} \eta_1 \eta_2 \zeta_3  - i \gamma k_1 \zeta_1 \eta_2 \eta_3 \\
& \quad + \Big( (k_1k_3+|k_1||k_3|) \frac{\gamma^2}{4k_1 k_3} - \frac{\gamma^2}{6} \Big) \eta_1 \eta_2 \eta_3 \Big] \delta_{123} dk_{123} \,.
\end{aligned}
\end{equation}
The term on  the second line of \eqref{H3-decomposed} corresponding to $\eta_1 \eta_2 \eta_3$ vanishes because of the symmetry in $\eta_1 \eta_2 \eta_3$ under index rearrangements, and  identity \eqref{sign-identity}. Indeed, we have
$$
\begin{aligned}
\int (k_1k_3+|k_1||k_3|) \frac{\gamma^2}{4k_1 k_3} \eta_1 \eta_2 \eta_3 \delta_{123} dk_{123} & = \frac{\gamma^2}{4} \int (1+\sgn(k_1) \sgn(k_3)) \eta_1 \eta_2 \eta_3 \delta_{123} dk_{123} \,, \\
& = \frac{\gamma^2}{6} \int \eta_1 \eta_2 \eta_3 \delta_{123} dk_{123} \,.
\end{aligned}
$$
As a result, $\H^{(3)} $ simplifies to
\begin{equation}
\label{H3-decomposed-2}
\H^{(3)} = - \frac{1}{2\sqrt{2\pi}} \int \Big[ (1+\sgn(k_1) \sgn(k_3)) |k_1| |k_3| \zeta_1 \eta_2 \zeta_3 - i \gamma |k_1| \sgn(k_3) \zeta_1 \eta_2 \eta_3 \Big] \delta_{123} dk_{123} \,.
\end{equation}
The first term in \eqref{H3-decomposed-2} identifies to the first term in  \eqref{H3-in-eta-zeta}, while its second term can be transformed  using index rearrangements and  identity \eqref{sign-identity} as follows
\begin{align*}
&\int |k_1| \sgn(k_3) \zeta_1 \eta_2 \eta_3 \delta_{123} dk_{123}  = 
\int |k_2| \sgn(k_3) \eta_1 \zeta_2 \eta_3 \delta_{123} dk_{123} \\
&=  \int \frac{|k_2|}{2}  (\sgn(k_1) + \sgn(k_3)) \eta_1 \zeta_2 \eta_3 \delta_{123} dk_{123}
 = - \int  \frac{k_2}{2}\left(1 + \sgn(k_1) \sgn(k_3) \right) \eta_1 \zeta_2 \eta_3 \delta_{123} dk_{123} \,.
\end{align*}
\end{proof}
In terms of the complex symplectic coordinates, $\H^{(3)}$ is 
a linear combination of third-order monomials 
\begin{align} \label{H3-fourier-z}
\H^{(3)} = & \; \frac{1}{8 \sqrt{\pi}} \int
\frac{1+\sgn(k_1) \sgn(k_3)}{a_1 a_2 a_3} \Big( \omega_1 \omega_3 - \frac{\gamma \omega_2}{2} \sgn(k_2) \Big) \times \\[3pt]
&
\big( z_1 z_2 z_3 + \overline z_1 \overline z_2 \overline z_3 - 2( \overline z_{-1} \overline z_{-2} z_3 + z_{-1} z_{-2} \overline z_3 ) + \overline z_{-1} z_2 \overline z_{-3} + z_{-1} \overline z_2 z_{-3} \big) \delta_{123} dk_{123} \,.
\nonumber
\end{align}

\subsection{Canonical transformations}

We are looking for a canonical transformation of the physical variables
\begin{equation*}
\label{transformation}
\tau: w = \begin{pmatrix}
\eta \\ \zeta
\end{pmatrix} \longmapsto w' \,,
\end{equation*}
defined in a neighborhood of the origin, such that
the transformed Hamiltonian $\H'$ satisfies
$$
\H'(w') = \H(\tau^{-1}(w')) \,, \quad 
\partial_t w' = J \, \nabla \H'(w') \,,
$$
and reduces to
\begin{equation}
\label{HH-new}
\H'(w') = \H^{(2)}(w') + Z^{(4)} + Z^{(5 )} + \dots \,,
\end{equation}
where 
each term $Z^{(m)}$ is of degree $m$, and all cubic terms are eliminated. 
We construct the transformation $\tau$  by the Lie transform method as a Hamiltonian flow $\phi$ from ``time'' $s=-1 $ to ``time'' $s=0$ governed by
$$
\partial_s \phi = J \, \nabla K(\phi) \,, \quad \phi(w')|_{s=0} = w' \,, \quad 
\phi(w')|_{s=-1} = w \,,
$$
and associated to an auxiliary Hamiltonian $K$.
Such a transformation is canonical and preserves the Hamiltonian structure of the system. The Hamiltonian $\H'$ satisfies 
$\H'(w') = \H(\phi(w'))|_{s=-1}$ and its Taylor expansion around $s=0$ is
$$
\H'(w') = \H(\phi(w'))|_{s=0} - \frac{d \H}{ds}(\phi(w'))|_{s=0} + \frac{1}{2} \frac{d^2 \H}{ds^2}(\phi(w'))|_{s=0} - \dots
$$
Abusing  notations, we further use $w = (\eta,\zeta)^\top $ 
to denote the new variable $w'$.  Terms in this expansion can be expressed using Poisson brackets as 
\begin{equation*}
\begin{aligned}
\H(\phi(w))|_{s=0} & = \H(w) \,, \\
\frac{d \H}{ds}(\phi(w))|_{s=0} & = \int \left( \partial_\eta \H \partial_s \eta + 
\partial_\zeta \H \partial_s \zeta \right) dx \,, \\    
& = \int \left( \partial_\eta \H \partial_\zeta K - \partial_\zeta \H \partial_\eta K \right) dx = 
\{K, \H\}(w) \,,
\end{aligned}
\end{equation*}
and similarly for the remaining terms. The Taylor expansion of $\H'$ around $s=0$ now has the form 
$$
\H'(w) = \H(w) - \{K, \H\}(w) + \frac{1}{2} \{K, \{K, \H\}\}(w) - \ldots
$$
Substituting this transformation into  the expansion (\ref{HH}) of $H$, we obtain
\begin{equation}
\label{new-ham-expansion-before}
\begin{aligned}
   \H'(w) & = \H^{(2)}(w) + \H^{(3)}(w) + \dots \\
 & \quad - \{K,\H^{(2)} \} (w) -\{K, \H^{(3)}\} (w)  -\{K, \H^{(4)}\} (w) - \dots \\
 & \quad + \frac{1}{2}\{K,\{K,\H^{(2)} \}\} (w) +\frac{1}{2} \{K,\{K,\H^{(3)}\}\} (w) + \dots
\end{aligned}
\end{equation}
If $K$ is homogeneous of degree $m$ and $\H^{(n)}$ is homogeneous of degree $n$, then $\{K,\H^{(n)} \}$ is of degree $m+n-2$. 
Thus, in order to  get the Hamiltonian as in (\ref{HH-new}), we need to
construct, if possible, an auxiliary Hamiltonian $K=K^{(3)}$ that is homogeneous of degree $3$ and satisfies the relation 
\begin{equation} \label{cohomological-relation}
   \H^{(3)}- \{K^{(3)},\H^{(2)} \} = 0 \,,
\end{equation}
which would eliminate all cubic terms from the transformed Hamiltonian $\H'$. In the following, we will show that it is possible to do so,
and that there are no resonant cubic terms in the Hamiltonian.

\subsection{Third-order Birkhoff normal form}

To find the auxiliary Hamiltonian $K^{(3)}$ from \eqref{cohomological-relation}, 
we use the following diagonal property of the coadjoint operator ${\rm coad}_{\H^{(2)}} := \{\cdot, \H^{(2)}\}$ when applied to monomial terms \citep{CS16}.
For example,
taking 
$
\mathcal{I} := \int z_1 z_2 \overline z_{-3} \delta_{123} dk_{123},
$
we have 
\begin{equation} \label{ad-H2}
\{ \mathcal{I}, \H^{(2)}\}  =  
i \int (\Omega_1 + \Omega_2 - \Omega_{-3})   z_1 z_2 \overline z_{-3}  \delta_{123} dk_{123} \,,
\end{equation}
where $\Omega_{\pm j} := \Omega_{\pm k_j}$. 

\begin{proposition} 
\label{proposition-on-K3}
The cohomological equation \eqref{cohomological-relation} has a unique solution $K^{(3)}$ which, in complex symplectic coordinates, is 
\begin{align} \label{K3-fourier-z}
K^{(3)} = & \; \frac{1}{8i \sqrt{\pi}} \int
\frac{1+\sgn(k_1) \sgn(k_3)}{a_1 a_2 a_3} \Big( \omega_1 \omega_3 - \frac{\gamma \omega_2}{2} \sgn(k_2) \Big) \delta_{123} \times \\
&
\left( \frac{z_1 z_2 z_3 - \overline z_1 \overline z_2 \overline z_3}{\Omega_1 + \Omega_2 + \Omega_3} + 2\frac{\overline z_{-1} \overline z_{-2} z_3 - z_{-1} z_{-2} \overline z_3 }{\Omega_{-1} + \Omega_{-2} - \Omega_3} - \frac{\overline z_{-1} z_2 \overline z_{-3} - z_{-1} \overline z_2 z_{-3}}{\Omega_{-1} - \Omega_2 + \Omega_{-3}} 
\right) dk_{123} \,.
\nonumber
\end{align}
Alternatively, in the $(\eta,\zeta)$ variables, $K^{(3)}$  has the form
{\small \begin{align}
\label{K3-fourier}
K^{(3)} =& \; \frac{1}{4 i \sqrt{2 \pi}} \int 
(1+\sgn(k_1) \sgn(k_3)) \delta_{123} \times \\
& \nonumber
\Big[ -\frac{\frac{\gamma^4}{8} + \frac{\gamma^2}{2} g|k_2| + g^2 |k_1| |k_3|}{g^2|k_1| |k_3|}
\Big( -\frac{\gamma}{2}\sgn(k_2) \eta_1 \eta_2 \eta_3 + i|k_2| \eta_1 \zeta_2 \eta_3 - 2i|k_3| \eta_1 \eta_2 \zeta_3 \Big) \\
& + \frac{\gamma \, \sgn(k_2) }{g^2|k_1| |k_3|} \Big( 2 \omega_3^2 |k_1| |k_2| \zeta_1 \zeta_2 \eta_3 - \omega_2^2 |k_1||k_3| \zeta_1 \eta_2 \zeta_3 + i \frac{\gamma \, \sgn(k_2) }{2}|k_1| |k_2| |k_3| \zeta_1 \zeta_2 \zeta_3 \Big) \Big] dk_{123} \,.
\nonumber
\end{align}}
\end{proposition}

\begin{proof}
Using (\ref{H3-fourier-z}) and diagonal properties of the coadjoint operator as in (\ref{ad-H2}), we  uniquely solve (\ref{cohomological-relation}) for $K^{(3)}$ and derive (\ref{K3-fourier-z}). 

To obtain the expression (\ref{K3-fourier}), we first need to rewrite (\ref{K3-fourier-z}) as a linear combination of third-order terms in $z_j$ and $z_{-j}$ only ($j = 1,2,3$). This can be done by applying the change of indices $(k_1, k_2, k_3) \to (-k_1, -k_2, -k_3)$ to monomials $\overline z_1 \overline z_2 \overline z_3$, $ z_{-1} z_{-2} \overline z_3$ and $z_{-1} \overline z_2 z_{-3}$. Then (\ref{K3-fourier}) is obtained 
in terms of $(\eta_k,\zeta_k)$ using the relations (\ref{zeta-variable-fourier})--(\ref{zeta-bar-variable-fourier}).
\end{proof}

\begin{remark}
The auxiliary Hamiltonian $K^{(3)}$ obtained in Proposition \ref{proposition-on-K3} is well-defined. There are no singularities due to the denominators 
$k_1$, $k_3$ 
since $\eta_k = \mathcal{O}(k)$ as $k \to 0$.
\end{remark}

Below we write an alternate representation of the auxiliary Hamiltonian $K^{(3)}$ which will be useful later when
we compute the Poisson bracket $ \{   \H^{(3)}, K^{(3)} \}$. Introduce the coefficient functions
\begin{equation}
\label{S123-A123}
S_{123} := \frac{1+\sgn (k_1) \sgn (k_3)}{a_1 a_2 a_3} \Big( k_1 k_3 a_1^2 a_3^2 - \frac{\gamma}{2} k_2 a_2^2 \Big) \,, \quad
A_{123} := \frac{1}{8\sqrt{\pi}} (S_{123} + S_{312} - S_{231}) \,.
\end{equation}
Note  that 
 $S_{123} = S_{321}$.
\begin{lemma}
We have 
\begin{equation}
\label{H3-K3-good-form}
\begin{aligned}
& \H^{(3)} = \int A_{123} (z_1 z_2 z_3 + \overline z_1 \overline z_2 \overline z_3 - z_{-1} z_{-2} \overline z_3 - \overline z_{-1} \overline z_{-2} z_3) \delta_{123} dk_{123} \,, \\[3pt]
& K^{(3)} = \frac{1}{i} \int \left[ \frac{A_{123}(z_1 z_2 z_3 - \overline z_{1} \overline z_{2} \overline z_{3})}{\Omega_1 + \Omega_2 + \Omega_3}
- \frac{A_{123}(z_{-1} z_{-2} \overline z_{3} - \overline z_{-1} \overline z_{-2} z_3)}{\Omega_{-1} + \Omega_{-2} - \Omega_3} \right] \delta_{123} dk_{123} \,.
\end{aligned}
\end{equation}
\end{lemma}
\begin{proof}
From \eqref{H3-fourier-z}, using that $k_1 k_3 = |k_1| |k_3|$ whenever $1+\sgn(k_1) \sgn(k_3) \neq 0$, we have 
\begin{equation}
\label{H3-fourier-z-intermediate}
\begin{aligned}
\H^{(3)} = \frac{1}{8 \sqrt{\pi}} \int S_{123} \Big( & (z_1 z_2 z_3 + \overline z_1 \overline z_2 \overline z_3) 
- 2 (z_{-1} z_{-2} \overline z_3 + \overline z_{-1} \overline z_{-2} z_3 ) \\
& + (\overline z_{-1} z_2 \overline z_{-3} + z_{-1} \overline z_2 z_{-3}) \Big) \delta_{123} dk_{123} \,.
\end{aligned}
\end{equation}
By symmetry, the first term on the right-hand side in \eqref{H3-fourier-z-intermediate} becomes
\begin{equation}
\label{zzz-symmetry-identity}
\frac{1}{8 \sqrt{\pi}} \int S_{123} (z_1 z_2 z_3 + \overline z_1 \overline z_2 \overline z_3) \delta_{123} dk_{123} = \int A_{123} (z_1 z_2 z_3 + \overline z_1 \overline z_2 \overline z_3) \delta_{123} dk_{123} \,.
\end{equation}
The second term in \eqref{H3-fourier-z-intermediate} has two copies of $(z_{-1} z_{-2} \overline z_3 + \overline z_{-1} \overline z_{-2} z_3 )$. We keep one copy as it is, and apply two index rearrangements to the second copy. First, $(k_1,k_2,k_3)\to (k_3,k_2,k_1)$ gives
$$ 
\int S_{123}
(z_{-1} z_{-2} \overline z_3 + \overline z_{-1} \overline z_{-2} z_3) \delta_{123}dk_{123} = \int S_{123}
(\overline{z}_{1} z_{-2} z_{-3} + z_{1} \overline{ z}_{-2} \overline{z}_{-3}) \delta_{123}dk_{123} \,,
$$
where we use that $S_{123} = S_{321}$. Second, $(1,2,3) \to (3,1,2)$ implies 
$$
 \int S_{123}
(\overline{z}_{1} z_{-2} z_{-3} + z_{1} \overline{ z}_{-2} \overline{z}_{-3}) \delta_{123}dk_{123} =  \int S_{312}
( z_{-1} z_{-2} \overline{z}_{3} + \overline{ z}_{-1} \overline{z}_{-2} z_3) \delta_{123}dk_{123} \,.
$$
For the third term in (\ref{H3-fourier-z-intermediate}), we apply $(k_1,k_2,k_3)\to (k_2,k_3,k_1)$ and leave details to the reader. Combining all the transformations above, we obtain 
$$
\begin{aligned}
\frac{1}{8 \sqrt{\pi}} & \int S_{123} \Big(- 2(z_{-1} z_{-2} \overline z_3 + \overline z_{-1} \overline z_{-2} z_3 ) + \overline z_{-1} z_2 \overline z_{-3} + z_{-1} \overline z_2 z_{-3} \Big) \delta_{123} dk_{123} \\
& = \int A_{123}(-z_{-1} z_{-2} \overline z_3 - \overline z_{-1} \overline z_{-2} z_3 ) \delta_{123} dk_{123} \,,
\end{aligned}
$$
which together with \eqref{zzz-symmetry-identity} implies the first relation of \eqref{H3-K3-good-form}. The second equation in \eqref{H3-K3-good-form} is derived from the relation \eqref{cohomological-relation} by using the diagonal properties \eqref{ad-H2}.
\end{proof}


The third-order normal form transformation defining the new coordinates is obtained as the solution map at $s=0$ of 
the Hamiltonian flow 
\begin{equation}
\label{normal-form-transformation-h}
\partial_s   
\begin{pmatrix}
\eta \\ \zeta
\end{pmatrix} = 
\begin{pmatrix}
0 & 1 \\
-1 & 0
\end{pmatrix}
\begin{pmatrix}
\partial_{\eta}K^{(3)} \\
\partial_{\zeta} K^{(3)}
\end{pmatrix} \,,
\end{equation}
with initial conditions at $s= -1$ being the original variables. 

To write $K^{(3)}$  in the physical space, it is convenient to introduce 
 $\widetilde \eta =  \Hil \eta$, 
 $\widetilde \zeta = \Hil \zeta$  where $ \Hil = - i \, \sgn(D)$ is the Hilbert transform.   In the following,  we will use the  identities $\Hil |D| = -\partial_x$, $\Hil |D|^{-1} = \partial_x^{-1}$ and similar ones.

 

\begin{proposition}
\label{proposition-K3-physical}
The auxiliary Hamiltonian $K^{(3)}$ in (\ref{K3-fourier}) can be written in the physical space as
\begin{equation}
\label{K3-physical-new}
\begin{aligned}
K^{(3)} = & \int \left[ \frac{1}{2} \widetilde{\eta}^2 \partial_x \widetilde\zeta 
 - \frac{\gamma}{4g} (g \eta^2 \widetilde \eta - \zeta^2 \partial_x \eta +2 \zeta 
\widetilde \eta \partial_x \widetilde\zeta) \right.
 \\
&
 -\frac{\gamma^2}{4g^2}  \Big(g (\partial_x^{-1} \eta) (\partial_x \eta) \zeta + g\eta  \widetilde \eta \widetilde\zeta -
g \widetilde\eta  (\partial_x \widetilde \zeta ) \partial_x^{-1} \eta   - \frac{1}{2}  \zeta^2 \partial_x \widetilde  \zeta \Big)\\
&  - \frac{\gamma^3}{16g^2}  \Big( \zeta^2 \widetilde \eta  +2  \zeta(\partial_x \widetilde\zeta)  \partial_x^{-1} \eta 
-2g (\partial_x^{-1} \eta) (\partial_x \widetilde\eta) (\partial_x^{-1}\widetilde\eta) \Big) \\
& \left. - \frac{\gamma^4}{16g^2}  (\eta \widetilde \zeta -\widetilde \eta \zeta) \partial_x^{-1} \eta   
- \frac{\gamma^5}{64g^2}   \widetilde \eta (\partial_x^{-1} \eta)^2 \right] dx \,.
\end{aligned}
\end{equation}

\end{proposition}

\begin{proof}
The main idea is to expand the brackets in the expression (\ref{K3-fourier}) for $K^{(3)}$, identify terms and combine them appropriately. 
We decompose $K^{(3)}$ as
$$
K^{(3)} = {\rm I} + {\rm II} + {\rm III} + {\rm IV} \,,
$$
where ${\rm I}$ is the part of $K^{(3)}$ associated with $\eta \eta \eta$-type terms, namely 
{\small 
\begin{equation*}
{\rm I} = \frac{1}{4 i \sqrt{2 \pi}} \int 
(1+\sgn(k_1) \sgn(k_3)) 
\left( \frac{\frac{\gamma^4}{8} + \frac{\gamma^2}{2} g|k_2| + g^2 |k_1| |k_3|}{g^2|k_1| |k_3|} \right)
\frac{\gamma}{2}\sgn(k_2) \eta_1 \eta_2 \eta_3 \delta_{123} dk_{123} \,.
\end{equation*}}
The part ${\rm II}$ is associated with $\eta \eta \zeta$-type terms, ${\rm III}$ is associated with $\eta \zeta \zeta$-type terms and ${\rm IV}$ is associated with $\zeta \zeta \zeta$-type terms.
Below we give the computations for ${\rm I}$. The remaining terms can be computed in a similar way. 
We further decompose ${\rm I} = {\rm I}_1 + {\rm I}_2 + {\rm I}_3$ into a sum of three terms based on the power of $\gamma$ involved, 
$$
\begin{aligned}
{\rm I}_1 & = \frac{\gamma^5}{64i g^2 \sqrt{2\pi}}\int 
(1+\sgn(k_1) \sgn(k_3)) 
\frac{1}{|k_1| |k_3|}
\sgn(k_2) \eta_1 \eta_2 \eta_3 \delta_{123} dk_{123} \,, \\[2pt]
{\rm I}_2 & = \frac{\gamma^3}{16 i g \sqrt{2\pi}}\int 
(1+\sgn(k_1) \sgn(k_3)) 
\frac{|k_2|}{|k_1| |k_3|}
\sgn(k_2) \eta_1 \eta_2 \eta_3 \delta_{123} dk_{123} \,, \\[2pt]
{\rm I}_3 & = \frac{\gamma}{8i \sqrt{2\pi}}\int 
(1+\sgn(k_1) \sgn(k_3)) 
\sgn(k_2) \eta_1 \eta_2 \eta_3 \delta_{123} dk_{123} \,.
\end{aligned}
$$
Writing $\sgn(k_2) = k_2/|k_2|$ in ${\rm I}_1$, we get
$$
\begin{aligned}
{\rm I}_1 
& = \frac{\gamma^5}{64i g^2 \sqrt{2\pi}}\int 
\frac{k_2}{|k_1||k_2| |k_3|}
\eta_1 \eta_2 \eta_3 \delta_{123} dk_{123} \\[2pt]
& \quad + \frac{\gamma^5}{64i g^2 \sqrt{2\pi}}\int 
\frac{\sgn(k_1)\sgn(k_2) \sgn(k_3)}{|k_1||k_3|}
\eta_1 \eta_2 \eta_3 \delta_{123} dk_{123} = {\rm I}_1^{(1)} + {\rm I}_1^{(2)} \,.
\end{aligned}
$$
Using the index rearrangements $(k_1,k_2,k_3) \to (k_2,k_1,k_3)$ and $(k_1,k_2,k_3)\to (k_1,k_3,k_2)$, we see that ${\rm I}_1^{(1)} = 0$
since the integration is over $k_1+k_2+k_3=0$. 
The part ${\rm I}_1$ thus reduces to
\begin{equation}
\label{I-1-term-K3-computation}
{\rm I}_1 = \frac{\gamma^5}{64i g^2 \sqrt{2\pi}} \int 
\frac{\sgn(k_2)}{k_1 k_3}
\eta_1 \eta_2 \eta_3 \delta_{123} dk_{123} = - \frac{\gamma^5}{64 g^2} \int (\partial_x^{-1} \eta)^2 \widetilde \eta \, dx \,.
\end{equation}
Similar steps imply 
\begin{equation}
\label{I-2-term-K3-computation}
{\rm I}_2 = \frac{\gamma^3}{8g} \int (\partial_x^{-1} \eta) (\partial_x \widetilde \eta) (\partial_x^{-1} \eta) \, dx \,, \quad
{\rm I}_3 = -\frac{\gamma}{4} \int \eta^2 \widetilde \eta \, dx \,.
\end{equation}
Combining (\ref{I-1-term-K3-computation})--(\ref{I-2-term-K3-computation}),
\begin{equation*}
\label{I-term-in-K3-final}
\begin{aligned}
{\rm I} & = -\frac{\gamma}{4} \int \eta^2 \widetilde \eta \, dx  + \frac{\gamma^3}{8g} \int (\partial_x^{-1} \eta) (\partial_x\widetilde \eta) (\partial_x^{-1} 
\widetilde\eta) \, dx - \frac{\gamma^5}{64 g^2} \int (\partial_x^{-1} \eta)^2 \widetilde\eta \, dx \,,
\end{aligned}
\end{equation*}
which are the second, eleventh and fourteenth terms in \eqref{K3-physical-new}. Terms ${\rm II}$, ${\rm III}$ and ${\rm IV}$ are calculated in a similar fashion and identify to the remaining terms in \eqref{K3-physical-new}.
\end{proof}
We now apply the variational derivatives $\partial_\eta$ and $\partial_\zeta$ on $K^{(3)}$ to obtain the third-order normal form transformation
 defining the Hamiltonian flow \eqref{normal-form-transformation-h}.
\begin{proposition} The Hamiltonian system that defines the  third-order normal form transformation has the form of a system of two partial differential equations
\begin{equation}
\label{normal-form-eta-evolution}
\begin{aligned}
\partial_s \eta & = \partial_{\zeta} K^{(3)}   
 =  \frac{1}{2} \Hil \partial_x \widetilde\eta^2 + \frac{\gamma}{2g} \Big( \zeta \partial_x \eta -  \widetilde \eta\partial_x \widetilde\zeta 
- |D| (\zeta \widetilde \eta) \Big) \\[2pt]
& \qquad +\frac{\gamma^2}{4g^2} \Big( \zeta\partial_x \widetilde \zeta + \frac{1}{2} |D|\zeta^2 \Big)
   +  \frac{\gamma^2}{4g} \Big(-(\partial_x^{-1} \eta) (\partial_x \eta)
+  \Hil( \eta \widetilde \eta )+  |D| (\widetilde \eta \partial_x^{-1} \eta)\Big) \\[2pt]
&\qquad 
 - \frac{\gamma^3}{8g^2} \Big( \zeta \widetilde\eta + (\partial_x \widetilde\zeta) (\partial_x^{-1} \eta) + |D| (\zeta \partial_x^{-1} \eta)\Big)
 + \frac{\gamma^4}{16g^2} \Big( \Hil ( \eta \partial_x^{-1} \eta) +  \widetilde\eta \partial_x^{-1} \eta \Big) \,,
\end{aligned}
\end{equation}
\begin{equation}
\label{normal-form-zeta-evolution}
\begin{aligned}
\partial_s \zeta & = -\partial_{\eta} K^{(3)}   
= 
 \Hil ( \widetilde\eta \partial_x \widetilde\zeta ) 
+ \frac{\gamma}{2} \Big( \eta \widetilde\eta - \frac{1}{2} \Hil (\eta^2) \Big)+ \frac{\gamma}{2g} \Big( \zeta \partial_x \zeta -  \Hil (\zeta \partial_x \widetilde
 \zeta) \Big)
\\[2pt]
& \quad   - \frac{\gamma^2}{4g}\Big( \partial_x (\zeta \partial_x^{-1} \eta) +  \partial_x^{-1} (\zeta \partial_x \eta) - \widetilde\zeta \widetilde\eta +
 \Hil( \eta \widetilde\zeta)  - \Hil \big( (\partial_x \widetilde \zeta) (\partial_x^{-1} \eta) \big) - \partial_x^{-1} \big( (\partial_x \widetilde \zeta) \widetilde \eta \big) \Big) \\[2pt]
&\quad + \frac{\gamma^3}{8g} \Big( \partial_x^{-1} \big( (\partial_x \widetilde \eta) ( \partial_x^{-1} \widetilde\eta) \big)
  - |D| \big( (\partial_x^{-1} \eta) (\partial_x^{-1} \widetilde\eta) \big) + |D|^{-1} \big( (\partial_x \widetilde \eta) (\partial_x^{-1} \eta) \big) \Big) \\[2pt]
& \quad - \frac{\gamma^3}{16g^2} \Big(\Hil (\zeta^2) + 2 \partial_x^{-1} (\zeta \partial_x \widetilde \zeta)\Big)
+ \frac{\gamma^4}{16g^2} \Big( \widetilde \zeta \partial_x^{-1} \eta -  \partial_x^{-1} (\eta\widetilde \zeta) + \Hil (\zeta \partial_x^{-1} \eta) + \partial_x^{-1} (\zeta \widetilde \eta) \Big)\\[2pt]
&\quad
-\frac{\gamma^5}{64g^2} \Big( \Hil \big( (\partial_x^{-1} \eta)^2 \big) + 2\partial_x^{-1} \big( \widetilde\eta  \partial_x^{-1} \eta \big) \Big) \,.
\end{aligned}
\end{equation}

\end{proposition}

In the absence of vorticity ($\gamma = 0$), the equation for $\eta$ simplifies to 
$\partial_s \eta  = \frac{1}{2} \Hil \partial_x ( \widetilde \eta)^2$,
or equivalently, to the inviscid Burgers equation for $\widetilde \eta$
$$
\partial_s \widetilde\eta  + \widetilde \eta  \partial_x \widetilde \eta=0 \,,
$$
as obtained by \cite{CS16}, while  $\widetilde \zeta$ satisfies 
$$
\partial_s \widetilde\zeta  + \widetilde \eta  \partial_x \widetilde \zeta=0 \,,
$$
which is its linearization along the Burgers flow.
These equations were tested in \cite{CGS21a,GKSX21,GKS22} in the context of irrotational gravity waves on deep water.

Due to the complexity of the formulas for $\partial_{\eta} K^{(3)}$ and $\partial_{\zeta} K^{(3)}$, we have checked a posteriori that 
the cohomology equation   \eqref{cohomological-relation} is indeed satisfied.  

\section{Reduced Hamiltonian}

In this section, we analyze the new Hamiltonian $\H'$ obtained after applying the third-order normal form transformation given by the flow of the auxiliary Hamiltonian system \eqref{normal-form-transformation-h}. By construction, such a transformation removes all cubic homogeneous terms based on the relation \eqref{cohomological-relation}. For simplicity, we now drop the primes from all new quantities. From \eqref{new-ham-expansion-before}, the new Hamiltonian becomes 
\begin{equation*}
\label{reduced-hamiltonian-0}
\begin{aligned}
\H(w) & = \H^{(2)}(w) + \H^{(4)}(w) - \{K^{(3)}, \H^{(3)}\}(w) + 
\frac{1}{2} \{K^{(3)}, \{K^{(3)}, \H^{(2)}\}\}(w) + R^{(5)} \,, \\
& = \H^{(2)}(w) + \H_+^{(4)}(w) + R^{(5)} \,,
\end{aligned}
\end{equation*}
where $R^{(5)}$ denotes all terms of order $5$ and higher, and $\H^{(4)}_+$ is the new fourth-order term
\begin{equation}
\label{new-H4-formula}
\H^{(4)}_+ = \H^{(4)} - \frac{1}{2} \{K^{(3)}, \H^{(3)}\} \,.
\end{equation}
Our approximation is based on the quadratic and quartic homogeneous terms of $\H$. The quadratic term  is given by \eqref{H2} in terms of new complex symplectic coordinates. 
The quartic term  $\H_+^{(4)}$ is more complicated and requires careful  computations. 
 Since the latter is  homogeneous of degree $4$ in $\eta$ and $\zeta$, every monomial appearing in $\H_+^{(4)}$ is of type $A_1 A_2 A_3 A_4$ with $A_j$ being  $\eta$ or $\zeta$. 
 Equivalently, 
 in terms of the complex symplectic coordinates \eqref{zeta-variable-fourier}, 
 it  has the form of  a sum of integrals 
with all possible combinations of fourth-order monomials in $z_k$ and $\overline z_{-k}$, 
\begin{equation}
\label{new-H4-general-form-1}
\begin{aligned}
\calH_+^{(4)} = &
\int  \Big [T^{+} z_1 z_2 z_3 z_4  + 
 T^\pm z_1 z_2 z_3 \overline z_{-4}   + 
 T^{+}_{-} z_1 z_2 \overline z_{-3} \overline z_{-4}  + 
T^\mp z_1 \overline z_{-2} \overline z_{-3} \overline z_{-4} \\
& \qquad+ T^{-} \overline z_{-1} \overline z_{-2} \overline z_{-3} \overline z_{-4}   \Big]\delta_{1234} d{k}_{1234} \,,
\end{aligned}
\end{equation}
where $T^+$, $T^{\pm}$, $T_-^+$, $T^{\mp}$ and $T^-$ are coefficients depending on $k_1, k_2, k_3$ and $k_4$.
In view of the subsequent modulational Ansatz and homogenization process,  it is not necessary to calculate
explicitly all the coefficients above. As shown in \cite{GKS22}, under the modulational Ansatz, 
only the third term in \eqref{new-H4-general-form-1} is relevant as the other terms are negligible due to scale separation. Therefore, we 
will only calculate the term 
\begin{equation*}
\label{H-plus-R-0}
\H_{+R}^{(4)} = \int T^{+}_{-} z_1 z_2 \overline z_{-3} \overline z_{-4} \delta_{1234} dk_{1234} \,.
\end{equation*}
which is,  after index rearrangement $(k_1, k_2, k_3, k_4) \to (k_1, k_2, -k_3, -k_4)$, 
\begin{equation}
\label{fourth-order-H-R}
\calH_{+R}^{(4)} =  \int T z_1 z_2 \overline z_{3} \overline z_{4} \delta_{1+2-3-4} d{k}_{1234} \,.
\end{equation}
Denoting
\begin{align} \label{decompH4}
&\calH_R^{(4)} =\int T_1 z_1 z_2 \overline z_{3} \overline z_{4} \delta_{1+2-3-4} d{ k}_{1234} \,, \quad
\{K^{(3)}, \calH^{(3)}\}_R = \int T_2 z_1 z_2 \overline z_{3} \overline z_{4} \delta_{1+2-3-4} d{k}_{1234} \,,
\end{align}
the contributions from $zz\overline{z} \overline{z}$-type monomials to $\calH^{(4)}$  and $\{K^{(3)}, \calH^{(3)}\}$ respectively,  we have  
\begin{equation}
\label{T-formula}
T = T_1 - \frac{1}{2} T_2 \,.
\end{equation}


The precise formulas for  the coefficients $T_1$ and $T_2$ are given in the  next two propositions.

\begin{proposition}
\label{lemma-T1-coeff}
We have $T_1 = T_1^{(1)} + T_1^{(2)} + T_1^{(3)}$ where
\begin{equation*}
\begin{aligned}
& T_1^{(1)} = -D^{(1)}_{12(-3)(-4)} -D^{(1)}_{(-4)(-3)21} -D^{(1)}_{1(-3)2(-4)} -D^{(1)}_{(-4)2(-3)1} +D^{(1)}_{1(-4)(-3)2} +D^{(1)}_{(-4)21(-3)} \,, \\[2pt]
& T_1^{(2)} = D^{(2)}_{12(-3)(-4)} -D^{(2)}_{(-4)(-3)21} +D^{(2)}_{1(-3)2(-4)} -D^{(2)}_{(-4)2(-3)1} +D^{(2)}_{1(-4)(-3)2} -D^{(2)}_{(-4)21(-3)} \,, \\[2pt]
& T_1^{(3)} = D^{(3)}_{12(-3)(-4)} +D^{(3)}_{(-4)(-3)21} +D^{(3)}_{1(-3)2(-4)} +D^{(3)}_{(-4)2(-3)1} +D^{(3)}_{1(-4)(-3)2} +D^{(3)}_{(-4)21(-3)} \,,
\end{aligned}
\end{equation*}
with
$$
\begin{aligned}
& D_{1234}^{(1)} = \frac{a_1 a_4}{32 \pi a_2 a_3} |k_1| |k_4| (|k_1|+|k_4| - 2|k_3+k_4|) \,,\\[2pt]
& D_{1234}^{(2)} = \frac{\gamma a_1}{32 \pi a_2 a_3 a_4} |k_1| \sgn(k_4) (|k_1|+|k_4| - |k_3+k_4| - |k_3+k_1|) \,,\\[2pt]
& D_{1234}^{(3)} = \frac{\gamma^2}{128 \pi a_1 a_2 a_3 a_4} \sgn(k_1) \sgn(k_4) (|k_1|+|k_4| - 2|k_3+k_4|) \,.
\end{aligned}
$$
\end{proposition}

\begin{proof}
Expanding the brackets in the expression (\ref{homog-ham-initial}) of $\H^{(4)}$, we have 
$$
\begin{aligned}
\H^{(4)} = - \frac{1}{8 \pi} \int & |k_1||k_4|(|k_1| + |k_4| - 2 |k_3+k_4|) \times \\[2pt]
& \Big( \zeta_1 \eta_2 \eta_3 \zeta_4 - \frac{i \gamma}{2k_1} \eta_1 \eta_2 \eta_3 \zeta_4 - \frac{i \gamma}{2k_4} \zeta_1 \eta_2 \eta_3 \eta_4 - \frac{ \gamma^2}{4k_1k_4} \eta_1 \eta_2 \eta_3 \eta_4 \Big) \delta_{1234} dk_{1234} \,.
\end{aligned}
$$
The terms associated with $\eta_1 \eta_2 \eta_3 \zeta_4$ and $\zeta_1 \eta_2 \eta_3 \eta_4$ can be combined by using the index rearrangement $(k_1,k_2,k_3,k_4) \to (k_4,k_2,k_3,k_1)$, and we obtain 
\begin{equation}
\label{H4-intermediate-T1}
\begin{aligned}
\H^{(4)} &=  - \frac{1}{8 \pi} \int  |k_1||k_4|(|k_1| + |k_4| - 2 |k_3+k_4|)
\zeta_1 \eta_2 \eta_3 \zeta_4  \delta_{1234} dk_{1234}\\[2pt]
& \quad + \frac{1}{8 \pi} \int |k_1||k_4|(|k_1| + |k_4| - |k_3+k_4| - |k_3+k_1|) \frac{i \gamma}{k_4} \zeta_1 \eta_2 \eta_3 \eta_4 \delta_{1234} dk_{1234}\\[2pt]
& \quad + \frac{1}{8 \pi} \int  |k_1||k_4|(|k_1| + |k_4| - 2 |k_3+k_4|)
\frac{ \gamma^2}{4k_1k_4} \eta_1 \eta_2 \eta_3 \eta_4 \delta_{1234} dk_{1234} \,, \\[2pt]
& := {\rm I} + {\rm II} + {\rm III} \,.
\end{aligned}
\end{equation}
We then write $(\eta,\zeta)$ in terms of $(z, \bar z)$, and retain only
monomials of type $zz\overline{z} \overline{z}$. Denoting ${\rm I}_R$, ${\rm II}_R$, ${\rm III}_R$ the corresponding terms in \eqref{H4-intermediate-T1},
 we find
\begin{equation*}
\begin{aligned}
{\rm I}_R  = \int D_{1234}^{(1)} & (-z_1 z_2 \overline{z}_{-3} \overline{z}_{-4} - z_1 \overline{z}_{-2} z_3 \overline{z}_{-4} + z_1 \overline{z}_{-2} \overline{z}_{-3} z_4 \\
&+ \overline{z}_{-1} z_2 z_3 \overline{z}_{-4} - \overline{z}_{-1} z_2 \overline{z}_{-3} z_4 - \overline{z}_{-1} \overline{z}_{-2}z_3 z_4) \delta_{1234}dk_{1234} \,.
\end{aligned}    
\end{equation*}
We turn all monomials in the above expression into $z_1 z_2 \overline{z}_3 \overline{z}_4$ by transforming indices in an appropriate way,
leading to 
$$
{\rm I}_R = \int T_1^{(1)} z_1 z_2 \overline{z}_{3} \overline{z}_{4} \delta_{1+2-3-4} dk_{1234} \,,
$$
as well as 
\begin{equation*}
\begin{aligned}
{\rm II}_R  = \int T_1^{(2)} z_1 z_2 \overline{z}_{3} \overline{z}_{4} \delta_{1+2-3-4} dk_{1234} \,, \quad
{\rm III}_R = \int T_1^{(3)} z_1 z_2 \overline{z}_{3} \overline{z}_{4} \delta_{1+2-3-4} dk_{1234} \,.
\end{aligned}
\end{equation*}
\end{proof}
To find the explicit expression for the coefficient $T_2$ of $\{K^{(3)}, \H^{(3)}\}_R$, we follow the steps in Appendix B of our recent paper \citep{GKS22}. The main idea is to use  \eqref{H3-K3-good-form} to expand the Poisson bracket $\{K^{(3)}, \H^{(3)}\}$ according to formula \eqref{poisson-bracket-z}, and extract terms of $zz \overline z \overline z $-type. 

\begin{proposition}
\label{lemma-T2-coeff}
We have 
$ T_2 = T_2^{(1)} + T_2^{(2)} + T_2^{(3)}$ with

\begin{align}
\label{T-2-1-term}
T_2^{(1)}  & =  \frac{1}{64 \pi}  \big(S_{(-1-2)12} + S_{2(-1-2)1} + S_{12(-1-2)}\big)
 \big(S_{(-3-4)34} + S_{4(-3-4)3} + S_{34(-3-4)}\big) \\
\nonumber
 \phantom{t} &
\quad \times 
\Big( \frac{1}{\Omega_{1}+ \Omega_{2}+ \Omega_{-1-2}} + \frac{1}{\Omega_{3}+ \Omega_{4}+ \Omega_{-3-4}} \Big) \,,
\end{align}


\begin{align}
\label{T-2-2-term}
&T_2^{(2)}  =  - A_{(-1)(-2)(1+2)} A_{(-3)(-4)(3+4)}
 \Big( \frac{1}{\Omega_{1}+ \Omega_{2}- \Omega_{1+2}} + \frac{1}{\Omega_{3}+ \Omega_{4}- \Omega_{3+4}}
\Big) \,,
\end{align}


\begin{align}
\label{T-2-3-term}
T_2^{(3)} =  4 A_{(1-3)(-1)3} 
A_{(4-2)(-4)2}  
\Big( \frac{1}{\Omega_{3- 1}+ \Omega_{1}- \Omega_{3}} + \frac{1}{\Omega_{2-4}+ \Omega_{4}- \Omega_{2}}
\Big) \,.
\end{align}
\end{proposition}

\begin{proof}
The Poisson bracket of the cubic Hamiltonian \eqref{H3-K3-good-form} contains $zz\overline{z}\overline{z}$-type terms given by
\begin{equation}
\label{K3-H3-poisson}
\begin{aligned}
i \, \{K^{(3)}, \H^{(3)}\}_R = & \left\{ \int \frac{A_{123} ~z_1 z_2 z_3 }{\Omega_1+\Omega_2+\Omega_3} \delta_{123} dk_{123} \,, \int A_{456} \overline z_{4} \overline z_{5} \overline z_{6} \delta_{456} d{ k}_{456} \right\}\\[5pt]
& - \left\{ \int \frac{A_{123} ~\overline z_{1} \overline z_{2} \overline z_{3} }{\Omega_1+\Omega_2+\Omega_3} \delta_{123} d{ k}_{123} \,, \int A_{456}  z_4 z_5 z_6 \delta_{456} d{ k}_{456} \right\}\\[5pt]
& + \left\{ \int \frac{A_{123} ~z_{-1} z_{-2} \overline z_{3} }{\Omega_{-1}+\Omega_{-2}-\Omega_{3}}  \delta_{123} d{ k}_{123} \,, \int A_{456} \overline z_{-4} \overline z_{-5} z_6 \delta_{456} d{ k}_{456} \right\}\\[5pt]
& - \left\{ \int \frac{A_{123} ~\overline z_{-1} \overline z_{-2} z_{3} }{\Omega_{-1}+\Omega_{-2}-\Omega_3} \delta_{123} d{ k}_{123} \,, \int A_{456} z_{-4} z_{-5} \overline z_{6} \delta_{456} d{ k}_{456} \right\} \,, \\[5pt]
&:= i \, (R_1+R_2+R_3+R_4) \,,
\end{aligned}
\end{equation}
where we denote each line of (\ref{K3-H3-poisson}) by $R_1$, $R_2$, $R_3$, $R_4$ respectively.
We obtain 
\begin{equation*}
\begin{aligned}
& R_1 + R_2 = \int T_2^{(1)} z_1 z_2 \overline z_3 \overline z_4 \delta_{1+2-3-4} dk_{1234} \,, \\
& R_3 + R_4 = \int (T_2^{(2)} + T_2^{(3)}) z_1 z_2 \overline z_3 \overline z_4 \delta_{1+2-3-4} dk_{1234} \,.
\end{aligned}
\end{equation*}
We refer to \cite{GKS22} for more details on such computations.
\end{proof}

\section{Modulational Ansatz}

We restrict our interest to solutions in the form of near-monochromatic waves with carrier wavenumber $k_0 > 0$. In the Fourier space, this corresponds to a narrowband approximation with $\eta_k$ and $\zeta_k$
localized near $k_0$. 
Equivalently,  $z_{k}$ and $\overline z_{k}$ are also localized around $k_0$. 
As it was pointed out earlier, such assumptions allow us to simplify the analysis of the quartic part $\H_+^{(4)}$ in \eqref{new-H4-general-form-1}, as
several of its terms become negligible. 
This is indeed a problem of homogenization which is treated via a scale separation lemma (see Lemma 4.4 of \cite{GKS22}).
This homogenization naturally selects the term \eqref{fourth-order-H-R} in $\H_+^{(4)}$, namely the 4-wave resonances, among all the possible quartic interactions as this term involves fast oscillations that exactly cancel out.
 Its coefficient $T$ can be found according to formula \eqref{T-formula}. The full expression of this coefficient is given by the combination of results in Propositions \ref{lemma-T1-coeff} and \ref{lemma-T2-coeff}. These expressions however simplify in the  modulational regime. 

We introduce the modulational Ansatz 
\begin{equation}
\label{modulation}
k = k_0 + \varepsilon \lambda \,, \quad \text{where} \quad \frac{\lambda}{k_0} = \calO(1) \,, \quad \varepsilon \ll 1 \,,
\end{equation}
which captures the slow modulation of small-amplitude near-monochromatic waves with carrier wavenumber $k_0 > 0$.
The small dimensionless parameter $\varepsilon$ is a measure of the wave spectrum's narrowness around $k = k_0$.
Accordingly, we define the function $U$ as
\begin{equation}
\label{U-in-fourier}
U(\lambda) = z(k_0 + \varepsilon \lambda) \,, \quad \overline U(\lambda) = \overline z(k_0 + \varepsilon \lambda) \,,
\end{equation}
in the Fourier space, where the time dependence is omitted. In the physical space,
\begin{equation}
\label{z-u-relation-physical}
\begin{aligned}
z(x) & = \frac{1}{\sqrt{2\pi}} \int z(k) e^{ikx} dk =  \frac{\varepsilon}{\sqrt{2\pi}} \int U(\lambda) e^{i k_0 x} e^{i \lambda \varepsilon x} d\lambda = \varepsilon \, u(X) e^{ik_0 x} \,, 
\end{aligned}
\end{equation}
where $u$, as a function of  the long spatial scale $X = \varepsilon \, x$, is the inverse Fourier transform of $U$.
Equation \eqref{z-u-relation-physical} indicates that the dimensionless parameter $\varepsilon$ 
may also be related to some measure of the wave steepness.



To calculate the quartic interactions in the modulational regime, we approximate the coefficients in Propositions \ref{lemma-T1-coeff}
and \ref{lemma-T2-coeff} under the modulational Ansatz \eqref{modulation}. 
First, we need a few simple expansions that are summarized in the lemma below.

\begin{lemma}
Under the modulational Ansatz \eqref{modulation}, 
we have the following expansions
\begin{equation}
\label{expansion-identities}
\begin{aligned}
& |k| = k_0 + \varepsilon \lambda + \calO (\varepsilon^4) \,, \quad \sgn(k) = 1 + \calO (\varepsilon^5) \,, \quad \omega_k = \omega_0 \left( 1+ \varepsilon \frac{g}{2\omega_0^2}\lambda \right) + \calO (\varepsilon^2) \,, \\
& a_k = \sqrt{\frac{\omega_0}{k_0}} \left(1 + \frac{\varepsilon}{4} \Big( \frac{g}{\omega_0^2} - \frac{2}{k_0} \Big) \lambda \right) + \calO (\varepsilon^2) \,,\\
& a_{1-3} = a(k_1 - k_3) = \frac{\sqrt{|\gamma|}}{ \sqrt{2} \varepsilon^{1/2} |\lambda_1 - \lambda_3|^{1/2}}\left( 1+\frac{\varepsilon g}{\gamma^2} |\lambda_1-\lambda_3|
\right) + \calO (\varepsilon^{3/2}) \,, \\
& \omega_{1-3} = \frac{|\gamma|}{2} \left( 1 + \frac{2g\varepsilon}{\gamma^2} |\lambda_1-\lambda_3| \right) + \calO (\varepsilon^2) \,.
\end{aligned}
\end{equation}
\end{lemma}

\begin{lemma}
\label{lemma-T1-approximation}
Under the modulational Ansatz (\ref{modulation}), we have 
\begin{equation*}
\int T_1 z_1 z_2 \overline z_3 \overline z_4 \delta_{1+2-3-4} dk_{1234} = \varepsilon^3 \int \Big(c_0^l 
 + \varepsilon c_0^r  (\lambda_2 + \lambda_3) \Big)
U_1 U_2 \overline U_3 \overline U_4 \delta_{1+2-3-4} d\lambda_{1234}  + \calO(\varepsilon^5) \,,
\end{equation*}
where $U_j := U(\lambda_j)$, $T_1$ is given in Proposition \ref{lemma-T1-coeff} and the coefficients are 
\begin{equation}
\label{c-0-coeff}
\begin{aligned}
& c_0^l = \frac{k_0^3 \Omega_{0}^2}{8\pi \omega_{0}^2} \,, \quad
c_0^r = \frac{3 k_0^2 \Omega_{0}^2}{16\pi \omega_{0}^2} - \frac{\gamma g k_0^3 \Omega_{0}}{32 \pi \omega_{0}^4} \,.
\end{aligned}
\end{equation}
\end{lemma}

\begin{proof}
 $T_1$  given in Proposition \ref{lemma-T1-coeff} is the sum of several terms, all involving coefficients similar to  $D_{1234}^{(1)}$.
The latter includes only factors of type $|k_j|$ and $a_j$, and expansion of these factors is given in the lemma above.
For $D_{12(-3)(-4)}^{(1)}$, we write
$$
D_{12(-3)(-4)}^{(1)} = -\frac{k_0^3}{16\pi} \Big( 1 + \frac{\varepsilon}{2k_0}(\lambda_2 + 3\lambda_3 +2\lambda_4) +  \frac{\varepsilon g}{4 \omega_0^2} (\lambda_1 - \lambda_2 - \lambda_3 + \lambda_4) \Big) + \calO (\varepsilon^2) \,.
$$
The brackets above simplify in the integration due to the symmetry of $z_1 z_2 \overline z_3 \overline z_4$ under index rearrangements and the presence of the delta function $\delta_{1+2-3-4}$, yielding 
\begin{align*}
\int - D_{12(-3)(-4)}^{(1)} z_1 z_2 \overline z_3 \overline z_4 \delta_{1+2-3-4} dk_{1234}  =  \int 
\Big( \frac{k_0^3}{16\pi}+  \frac{3 \varepsilon k_0^2}{32\pi}(\lambda_2 + \lambda_3) \Big)
z_1 z_2 \overline z_3 \overline z_4 \delta_{1+2-3-4} dk_{1234} \,. \\
\end{align*}
The remaining computations  are similar and we have 
\begin{align*}
\int T_1^{(1)} z_1 z_2 \overline z_3 \overline z_4 \delta_{1+2-3-4} dk_{1234}  = \int
\Big( \frac{k_0^3}{8\pi} +  \frac{3 \varepsilon k_0^2}{16\pi}  (\lambda_2 + \lambda_3) \Big)
  z_1 z_2 \overline z_3 \overline z_4 \delta_{1+2-3-4} dk_{1234} \,,
\end{align*}
\begin{equation*}
\begin{aligned}
\int T_1^{(2)} z_1 z_2 \overline z_3 \overline z_4 \delta_{1+2-3-4} dk_{1234} & = \int
\Big(\frac{\gamma k_0^3}{8 \pi \omega_0} + \frac{\varepsilon \gamma k_0^3}{8 \pi \omega_0} ( \frac{3}{2k_0} - \frac{g}{4 \omega_0^2} ) (\lambda_2 + \lambda_3) \Big) \\
& \qquad \quad \times z_1 z_2 \overline z_3 \overline z_4 \delta_{1+2-3-4} dk_{1234} \,, 
\end{aligned}
\end{equation*}
\begin{align*}
\int T_1^{(3)} z_1 z_2 \overline z_3 \overline z_4 \delta_{1+2-3-4} dk_{1234} & = \int
\Big(\frac{\gamma^2 k_0^3}{32\pi \omega_0^2}  +   \frac{\varepsilon \gamma^2 k_0^3}{32\pi \omega_0^2}  \big(\frac{3}{2k_0} - \frac{g}{2\omega_0^2} \big)(\lambda_2 + \lambda_3) \Big) \\
& \qquad \quad \times z_1 z_2 \overline z_3 \overline z_4 \delta_{1+2-3-4} dk_{1234} \,. 
\end{align*}
Using that  $T_1 = T_1^{(1)}+ T_1^{(2)} +T_1^{(3)}$, we get  
\begin{align*}
\label{T1-coefficient-proof}
\int T_1 z_1 z_2 \overline z_3 \overline z_4 \delta_{1+2-3-4} dk_{1234} 
=  \int \Big(c_0^l +\varepsilon c_0^r 
(\lambda_2 + \lambda_3) \Big)
z_1 z_2 \overline z_3 \overline z_4 \delta_{1+2-3-4} dk_{1234}+ \calO(\varepsilon^2) \,.
%
\end{align*}
Writing the right-hand side in terms of $U$ given by \eqref{U-in-fourier} and using that 
$
\delta (k_1+k_2-k_3-k_4) 
= \varepsilon^{-1} \delta (\lambda_1 + \lambda_2 - \lambda_3 -\lambda_4)
$, we obtain the desired result.
\end{proof}
\begin{lemma}
\label{lemma-T2-approximation}
Under the modulational Ansatz (\ref{modulation}), we have  
\begin{equation}
\label{T-2-1-approximation}
\begin{aligned}
\int T_2^{(1)} z_1 z_2 \overline z_3 \overline z_4 \delta_{1+2-3-4} dk_{1234} &=  \varepsilon^3 \int 
\Big(c_1^l + \varepsilon c_1^r  (\lambda_2 + \lambda_3) \Big) U_1 U_2 \overline U_3 \overline U_4 \delta_{1+2-3-4} d\lambda_{1234} 
 + \calO(\varepsilon^5) \,,
\end{aligned}
\end{equation}


\begin{equation*}
\begin{aligned}
\int T_2^{(2)} z_1 z_2 \overline z_3 \overline z_4 \delta_{1+2-3-4} dk_{1234} = \varepsilon^3 \int 
\Big(c_2^l + \varepsilon c_2^r (\lambda_2 + \lambda_3) \Big) U_1 U_2 \overline U_3 \overline U_4 \delta_{1+2-3-4} d\lambda_{1234} 
 + \calO(\varepsilon^5) \,,
\end{aligned}
\end{equation*}
\begin{equation}
\label{T-2-3-approximation}
\begin{aligned}
&\int T_2^{(3)} z_1 z_2 \overline z_3 \overline z_4 \delta_{1+2-3-4} dk_{1234}  \\
& \qquad\qquad = \varepsilon^3 \int 
\Big( c_3^l + \varepsilon c_3^{r,1} (\lambda_2 + \lambda_3)  + \varepsilon^4 c_3^{r,2}  |\lambda_1 - \lambda_3| \Big) 
 U_1 U_2 \overline U_3 \overline U_4 \delta_{1+2-3-4} d\lambda_{1234} 
 + \calO(\varepsilon^5) \,,
\end{aligned}
\end{equation}
where
{\small
\begin{equation} \label{cj-coeff}
\begin{aligned}
& c_1^l = \frac{k_0^3 (2\sq^2  +\gamma \sp)^2}{16\pi \sq^2 \sp 
(2\Omega_{0} + \Omega_{-2k_0})} \,, \quad
\Omega_{\pm 2k_0} = \frac{\gamma}{2} \sgn(\pm 2k_0) + \omega_{2k_0} \,, \\[3pt]
& c_1^r = g c_1^l \left( \frac{2\Omega_{2k_0}
}{\sp (2\sq^2 + \gamma \sp)} - \frac{1}{2 \sq^2} - \frac{1}{2\sp^2} + \frac{3}{2gk_0} - \frac{\sp + \sq}{2\sp \sq
(2\Omega_{0} + \Omega_{-2k_0})} \right) \,, \\[3pt]
& c_2^l = -\frac{k_0^3 (2 \sq^2-\gamma \sp)^2}{16\pi \sq^2 \sp 
(2\Omega_{0} - \Omega_{2k_0})} \,, \\[3pt]
& c_2^r = g c_2^l \left( \frac{2\Omega_{-2k_0}
}{\sp (2\sq^2 - \gamma \sp)} - \frac{1}{2 \sq^2} - \frac{1}{2\sp^2} + \frac{3}{2gk_0} - \frac{\sp - \sq}{2\sp \sq 
(2\Omega_{0} - \Omega_{2k_0})
} \right) \,, \\[3pt]
& c_3^l = 
\frac{\gamma^2 k_0^2  \omega_{0}}{2 \pi g \Omega_{0}} \,,
\quad
 c_3^{r,1} = c_3^l \left( \frac{1}{k_0} + \frac{g \gamma }{8\Omega_{0} \omega_{0}^2} \right) \,,
\quad
c_3^{r,2} = \frac{k_0^2 \omega_0^2}{2\pi \Omega_0^2} \,. 
\end{aligned}
\end{equation}}
\end{lemma}

\begin{proof}
The proof is given in Appendix \ref{appendix-proof-expansion-lemma}.
\end{proof}

The leading coefficients in the expansions of Lemmas \ref{lemma-T1-approximation} and \ref{lemma-T2-approximation} combine together as 
\begin{equation*}
c_0^l - \frac{1}{2} \left(c_1^l +c_2^l + c_3^l \right) = \frac{k_0^3 (\omega_0 - \gamma)(\gamma^2 + 4 \omega_0^2)}{8\pi \omega_0 \Omega_0 (2 \omega_0 - \gamma)} \,.
\end{equation*}
Denoting 
\begin{equation} \label{coefA}
\beta := 8\pi \Big[ c_0^r - \frac{1}{2} \big( c_1^r+ c_2^r + c_3^{r,1} \big) \Big] \,,
\end{equation}
where the  $c_j^r$ $(j=0,\dots,2)$ and $c_3^{r,1}$ are given in \eqref{c-0-coeff} and \eqref{cj-coeff},
 the reduced Hamiltonian $\H^{(4)}_+$ takes the form
\begin{equation}
\label{H4-reduced}
\begin{aligned}
\H^{(4)}_+  & = \, \varepsilon^3
\frac{k_0^3 (\omega_0 - \gamma)(\gamma^2 + 4 \omega_0^2)}{8\pi \omega_0 \Omega_0 (2 \omega_0 - \gamma)} \int
U_1 U_2 \overline{U}_3 \overline{U}_4 \delta_{1+2-3-4} d\lambda_{1234}\\ 
& \quad + \varepsilon^4 \int 
\Big(  
\frac{\beta}{8\pi} (\lambda_2 + \lambda_3)  -  \frac{k_0^2 \omega_0^2}{4\pi \Omega_0^2}  |\lambda_1 - \lambda_3| \Big)
U_1 U_2 \overline{U}_3 \overline{U}_4 \delta_{1+2-3-4} d\lambda_{1234}   
 + \mathcal{O}(\varepsilon^5) \,.
\end{aligned} 
\end{equation}

\section{Hamiltonian Dysthe equation}

The third-order normal form transformation eliminates all cubic terms from the Hamiltonian $\H$. In the modulational 
regime (\ref{modulation}), the reduced Hamiltonian truncated at fourth order  is  
\begin{equation}
\label{reduced-H-fourier}
\H  = \H^{(2)} + \H_+^{(4)} \,. 
\end{equation}
The goal now is to derive an associated  Hamiltonian Dysthe equation for deep-water gravity waves with constant vorticity. 

\subsection{Hamiltonian in the physical variables}

\begin{lemma}
\label{lemma-reduced-H-in-physical}
In the physical variables $(u, \overline u)$, the Hamiltonian $\H$ in \eqref{reduced-H-fourier} reads
\begin{equation}
\label{reduced-H-for-dysthe}
\begin{aligned}
\H = & \, \varepsilon \int \overline u \, \Omega (k_0 + \varepsilon D_X) u \, dX + \varepsilon^3 \frac{k_0^3 (\omega_0 - \gamma)(\gamma^2 + 4 \omega_0^2)}{4 \omega_0 \Omega_0 (2 \omega_0 - \gamma)}
\int
|u|^4 dX\\[3pt]
& + \varepsilon^4
\frac{\beta}{2} \int |u|^2 {\rm Im} (\overline{u} \partial_X u) dX - \varepsilon^4 \frac{k_0^2 \omega_0^2}{2 \Omega_0^2} \int |u|^2 |D_X| |u|^2 dX + \calO (\varepsilon^5) \,,
\end{aligned}
\end{equation}
where $D_X = -i \partial_X$ in the slow spatial variable $X$.
\end{lemma}

\begin{proof}
The first term in \eqref{reduced-H-for-dysthe} comes out by applying the change of variables \eqref{U-in-fourier} to the quadratic Hamiltonian 
\begin{equation*}
\begin{aligned}
\H^{(2)} & = \int \Omega(k_0+\varepsilon \lambda) |z(k_0 +\varepsilon \lambda)|^2 dk 
 = \varepsilon \int \overline{u} \, \Omega(k_0+\varepsilon D_X) u \, dX \,,
\end{aligned}
\end{equation*}
where we use that $u(X) = \frac{1}{\sqrt{2\pi}} \int e^{i\lambda X} U(\lambda) d\lambda$ is the inverse Fourier transform of $U(\lambda)$. 
Furthermore, 
\begin{align*} \label{eq5-11}
 \int U_1 U_2 \overline U_3 \overline U_4 \delta_{1+2-3-4} d\lambda_{1234} 
& = 2 \pi \int |u|^4 d X \,.
\end{align*}
Similarly, 
\begin{equation*} \label{eq5-12}
\int (\lambda_2 + \lambda_3) U_1 U_2 \overline U_3 \overline U_4 \delta_{1+2-3-4} d\lambda_{1234} = 4 \pi \int |u|^2 {\rm Im}(\overline u \partial_X u) \, dX \,.
\end{equation*}
while the remaining term of $\H_+^{(4)}$ identifies to
\begin{equation*} \label{eq5-13}
\int |\lambda_1-\lambda_3| U_1 U_2 \overline U_3 \overline U_4 \delta_{1+2-3-4} d\lambda_{1234} = 2 \pi \int |u|^2 |D_X| |u|^2 \, dX \,.
\end{equation*}
\end{proof}
Expanding the linear dispersion relation in \eqref{reduced-H-for-dysthe} around $k_0$ as 
\begin{equation*}
\label{omega-taylor-expansion}
\Omega (k_0+ \varepsilon D_X) = \Omega_0 + \frac{g D_X}{2 \omega_0} \varepsilon - \frac{g^2 D_X^2}{8 \omega_0^3} \varepsilon^2 + \frac{g^3 D_X^3}{16 \omega_0^5} \varepsilon^3 + \calO(\varepsilon^4) \,,
\end{equation*}
gives an alternate form of the Hamiltonian $\H$ in the physical variables,
\begin{equation}
\label{reduced-H-for-dysthe-new}
\begin{aligned}
\H = & \, \int \Big[ \varepsilon\, \Omega_0 \, |u|^2 + \varepsilon^2 \frac{g}{2\omega_0} {\rm Im}(\overline u \partial_X u) \\
& \qquad - \varepsilon^3 \frac{g^2}{8 \omega_0^3} |\partial_X u|^2 + \varepsilon^3 \frac{k_0^3 (\omega_0 - \gamma)(\gamma^2 + 4 \omega_0^2)}{4 \omega_0 \Omega_0 (2 \omega_0 - \gamma)} |u|^4 - \varepsilon^4 \frac{g^3}{16 \omega_0^5} {\rm Im}(\overline u \partial_X^3 u) \\
& \qquad + \varepsilon^4
\frac{\beta}{2} |u|^2 {\rm Im} (\overline{u} \partial_X u) - \varepsilon^4 \frac{k_0^2 \omega_0^2}{2 \Omega_0^2} |u|^2 |D_X| |u|^2 \Big] dX +  \calO(\varepsilon^5) \,.
\end{aligned}
\end{equation}

\subsection{Derivation of the Dysthe equation}

Using the relation \eqref{z-u-relation-physical} rewritten as
\begin{equation*}
\label{new-variables-u}
\begin{pmatrix}
u \\ \overline u
\end{pmatrix} = P_2 \begin{pmatrix}
z \\ \overline z  
\end{pmatrix} = \varepsilon^{-1} \begin{pmatrix}
e^{-i k_0 x} & 0 \\
0 & e^{i k_0 x}
\end{pmatrix} \begin{pmatrix}
z \\ \overline z 
\end{pmatrix} \,,
\end{equation*}
the Hamiltonian system \eqref{ww-hamiltonian-in-z} takes the form
\begin{equation}
\label{ww-hamiltonian-in-u}
\partial_t \begin{pmatrix}
u \\ \overline u
\end{pmatrix} = J_2 \begin{pmatrix}
\partial_u \H \\ \partial_{\overline u} \H
\end{pmatrix} = \varepsilon^{-1} \begin{pmatrix}
0 & -i \\ i & 0
\end{pmatrix} \begin{pmatrix}
\partial_u \H \\ \partial_{\overline u} \H
\end{pmatrix} \,,
\end{equation}
where $J_2 = \varepsilon P_2 J_1 P_2^*$ \citep{CGS10}. 
The additional factor $\varepsilon$ in  the definition of
$J_2$ reflects the change in symplectic structure associated with the spatial rescaling $X = \varepsilon \, x$.

Substituting the reduced Hamiltonian \eqref{reduced-H-for-dysthe-new} into \eqref{ww-hamiltonian-in-u},  we get
\begin{equation}  \label{Dysthe}
\begin{aligned}
i \, \partial_t u & = \varepsilon^{-1} \partial_{\overline u} \H \,, \\
& = \Omega_0 u -i \varepsilon \frac{g}{2 \omega_0} \partial_X u + \varepsilon^2 \frac{g^2}{8 \omega_0^3} \partial_X^2 u + \varepsilon^2 \frac{k_0^3 (\omega_0 - \gamma)(\gamma^2 + 4 \omega_0^2)}{2 \omega_0 \Omega_0 (2 \omega_0 - \gamma)} |u|^2 u \\
& \quad + i\varepsilon^3 \frac{g^3}{16 \omega_0^5} \partial_X^3 u -
i \varepsilon^3 \beta |u|^2 \partial_X u - \varepsilon^3 \frac{k_0^2 \omega_0^2}{\Omega_0^2} u |D_X| |u|^2 \,,
\end{aligned}
\end{equation}
which is a Hamiltonian Dysthe equation for two-dimensional gravity waves on deep water with constant vorticity. 
It describes modulated waves moving in the positive $x$-direction at group velocity
$\Omega_0' = \partial_{k} \Omega(k_0) = g/(2 \omega_0)$ as indicated by the advection term. 
The nonlocal term $u |D_X| |u|^2$ is a signature of the Dysthe equation, which reflects the presence of the wave-induced mean flow. 
The coefficient $\beta$ is given by \eqref{coefA}. 
The coefficient of the nonlinear term $|u|^2 u$ above agrees with that in the NLS equation of \cite{TKM12} 
and in the Dysthe equation of \cite{CCK18} (see \eqref{curtis} below) up to scaling  factors consistent with 
the difference in definition for the wave envelope described in the various models.

The first two terms on the right-hand side of \eqref{Dysthe} can be eliminated via phase invariance and reduction to a moving reference frame.
The latter is equivalent, in the framework of canonical transformations, to subtraction from $\calH$ of a multiple of the momentum \eqref{moment}
which reduces to
$$
I  = \int \eta \partial_x\zeta dx = \varepsilon \int \Big[ k_0 |u|^2 
+ \varepsilon \operatorname{Im}(\overline u \partial_{X} u) \Big] dX \,,
$$
while the former is equivalent to subtraction from $\calH$ of a multiple of the wave action 
\begin{equation} \label{action}
M = \varepsilon  \int |u|^2 \, dX \,,
\end{equation}
which is conserved due to the phase-invariance property of the Dysthe equation. 
Because $I$ and $M$ Poisson commute with $\calH$, this transformation preserves the symplectic structure $J_2$ \citep{CGS21b}. 
The resulting Hamiltonian is given by
\begin{align*}
\widehat \calH  =  \calH - \Omega_0' I 
- \big( \Omega_0 - {k}_0 \Omega_0' \big) M \,,
\end{align*}
which, after introducing a new long-time scale $\tau = \varepsilon^2 t$, leads to the following version of the Hamiltonian Dysthe equation 
\begin{equation*} \label{Dy-longtime}
\begin{aligned}
i \, \partial_\tau u 
& =
\frac{g^2}{8 \omega_0^3} \partial_X^2 u +  \frac{k_0^3 (\omega_0 - \gamma)(\gamma^2 + 4 \omega_0^2)}{2 \omega_0 \Omega_0 (2 \omega_0 - \gamma)} |u|^2 u + i\varepsilon \frac{g^3}{16 \omega_0^5} \partial_X^3 u \\
& \quad - i \varepsilon \beta |u|^2 \partial_X u - \varepsilon \frac{k_0^2 \omega_0^2}{\Omega_0^2} u |D_X| |u|^2 \,.
\end{aligned}
\end{equation*}
It governs the long-time evolution of the envelope of modulated waves in a reference frame moving in the positive horizontal direction at group velocity $\Omega_0'$. The nonlocal operator $|D_X|$ is the Fourier multiplier with symbol $|\lambda|$.
The associated Hamiltonian reads more explicitly (after multiplying by $\varepsilon^{-3}$ and dropping the hat)
 \begin{equation*} 
\begin{aligned}
\H = & \, \int \Big[
 - \frac{g^2}{8 \omega_0^3} |\partial_X u|^2 +  \frac{k_0^3 (\omega_0 - \gamma)(\gamma^2 + 4 \omega_0^2)}{4 \omega_0 \Omega_0 (2 \omega_0 - \gamma)}
|u|^4 - \varepsilon \frac{g^3}{16 \omega_0^5} {\rm Im} (\overline u \partial_X^3 u) \\
& \qquad + \varepsilon
\frac{\beta}{2} |u|^2 {\rm Im} (\overline{u} \partial_X u) - \varepsilon \frac{k_0^2 \omega_0^2}{2 \Omega_0^2} |u|^2 |D_X| |u|^2 \Big] dX \,.
\end{aligned}
\end{equation*}
As suggested by \cite{TKDV00}, retaining the exact linear dispersion relation, rather than expanding it in powers of $\varepsilon$,
may provide an overall better approximation of the wave envelope.
On a related note, \cite{OS15} proposed a full-dispersion Davey--Stewartson system and compared its analytical properties
to those of the classical version.
In the present context, the Dysthe equation with full linear dispersion takes the form
\begin{equation*} \label{Dysthe3}
\begin{aligned}
i \, \partial_t u &= \Omega(k_0+\varepsilon D_X) u
+ \varepsilon^2 \frac{k_0^3 (\omega_0 - \gamma)(\gamma^2 + 4 \omega_0^2)}{2 \omega_0 \Omega_0 (2 \omega_0 - \gamma)} |u|^2 u 
- i \varepsilon^3 \beta |u|^2 \partial_X u - \varepsilon^3 \frac{k_0^2 \omega_0^2}{\Omega_0^2} u |D_X| |u|^2 \,,
\end{aligned}
\end{equation*}
and the corresponding Hamiltonian is given by \eqref{reduced-H-for-dysthe}.

\section{Numerical results} 

We now present numerical simulations to illustrate the performance of our Hamiltonian Dysthe equation.
We consider the problem of modulational stability of Stokes waves and examine the influence of vorticity.
We compare these results to predictions by another related envelope model
and to direct simulations of the full nonlinear equations.
We also test the capability of our reconstruction procedure against a simpler approach.

\subsection{Stability of Stokes waves}

We first give the theoretical prediction for modulational or Benjamin--Feir (BF) instability of Stokes waves.
These are represented by the exact uniform solution 
\begin{equation} \label{steady}
u_0(t)  = B_0 e^{-i (\Omega_0 + \varepsilon^2 \beta_0 B_0^2) t} \,,
\end{equation}
for \eqref{Dysthe}, where $B_0$ is a positive real constant and
\[
\beta_0 = \frac{k_0^3 (\omega_0 - \gamma)(\gamma^2 + 4 \omega_0^2)}{2 \omega_0 \Omega_0 (2 \omega_0 - \gamma)} \,.
\]
In the irrotational case ($\gamma = 0$), such a solution is known to be linearly unstable 
with respect to sideband (i.e. long-wave) perturbations.

The formal calculation consists in linearizing \eqref{Dysthe} about $u_0$ by inserting a perturbation of the form
\[
u(X,t) = u_0(t) \big[ 1 + B(X,t) \big] \,,
\]
where
\[
B(X,t) = B_1 e^{\sigma t + i \lambda X} + B_2 e^{\overline \sigma t - i \lambda X} \,,
\]
and $B_1, B_2$ are complex coefficients. We find that the condition 
$\operatorname{Re}(\sigma) \neq 0$ for instability implies
\begin{equation} \label{BFCond}
\alpha = \frac{g^2}{8 \omega_0^3} \lambda^2 \Gamma > 0 \,,
\end{equation}
with
\[
\Gamma = 2 B_0^2 \big( \beta_0 - \varepsilon \beta_3 |\lambda| \big) -  \frac{g^2}{8 \omega_0^3} \lambda^2 \,, \quad
\beta_3 = \frac{k_0^2 \omega_0^2}{\Omega_0^2} \,.
\]
This is a tedious but straightforward calculation for which we skip the details.
Similar calculations can be found in \cite{CCK18,D79,GT11}.

Figure \ref{BF_inst} depicts the normalized growth rate 
\[
\frac{|\textrm{Re}(\sigma)|}{\omega_0} = \frac{\sqrt{\alpha}}{\omega_0} \,,
\]
delimiting the instability region as predicted by condition \eqref{BFCond}
for $(B_0,k_0) = (0.002,10)$ and various values of $\gamma$.
The growth rate (and instability region) for $\gamma = 0$ is also included as a reference.
Hereafter, all the variables are rescaled to absorb $\varepsilon$ back into their definition,
and all the equations are non-dimensionalized by using $1/k_0$ and $1/\sqrt{g k_0}$ 
as characteristic length and time scales respectively, so that $g = 1$.
For convenience, we retain the same notations for all the dimensionless quantities.
We set $\varepsilon = k_0 A_0$ (surface wave steepness), 
noting that the envelope amplitude $B_0$ and the surface amplitude $A_0$ are related via
\begin{equation} \label{amplitude}
B_0 = A_0 \sqrt{\frac{\omega_0}{2k_0}} \,,
\end{equation}
according to \eqref{zeta-variable} and \eqref{z-u-relation-physical}.
The graphs in Fig. \ref{BF_inst} correspond to a wave steepness of about $\varepsilon = 0.05$.
Clearly, the vorticity $\gamma$ (both its magnitude and sign) has an influence on \eqref{BFCond}.
We see that $\gamma < 0$ tends to enhance the instability by amplifying the growth rate
and enlarging the instability region to higher sideband wavenumbers.
On the other hand, $\gamma > 0$ tends to diminish it.
Figure \ref{BF_inst} even suggests that, for sufficiently large $\gamma > 0$, instability no longer occurs.
This is confirmed by Fig. \ref{factor_inst} which shows that the factor $\Gamma$ in \eqref{BFCond}
is no longer positive at any wavenumber $\lambda$ when $\gamma > 3.5$ for $(B_0,k_0) = (0.002,10)$.
A positive vorticity (co-propagating current) has therefore a stabilizing effect on the dynamics of Stokes waves.

\subsection{Reconstruction of the original variables}

At any instant $t$, the surface elevation and velocity potential can be reconstructed from the wave envelope
by inverting the normal form transformation. 
This is accomplished by solving the auxiliary system 
\eqref{normal-form-eta-evolution}--\eqref{normal-form-zeta-evolution} backward from $s = 0$ to $s = -1$,
with ``initial'' conditions given by the transformed variables
\begin{eqnarray} \label{init_eta}
\eta(x,t) \big|_{s=0} & = & \frac{1}{\sqrt{2}} a^{-1}(D) \Big[ u(x,t) e^{i k_0 x} + \overline{u}(x,t) e^{-i k_0 x} \Big] \,, \\
\zeta(x,t) \big|_{s=0} & = & \frac{1}{i \sqrt{2}} a(D) \Big[ u(x,t) e^{i k_0 x} - \overline{u}(x,t) e^{-i k_0 x} \Big] \,,
\label{init_zeta}
\end{eqnarray}
according to \eqref{eta-xi-via-zeta} and \eqref{z-u-relation-physical}.
In these expressions, $u$ obeys \eqref{Dysthe} and $a^{-1}(D) = \sqrt{|D|/\omega(D)}$.
The final solution at $s = -1$ represents the original variables $(\eta,\zeta)$.
Starting from the first harmonics (with carrier wavenumber $k_0$) in the initial conditions \eqref{init_eta}--\eqref{init_zeta},
the evolutionary process in $s$ will automatically generate the next-order contributions
from lower and higher harmonics via nonlinear interactions according to \eqref{normal-form-eta-evolution}--\eqref{normal-form-zeta-evolution}.
Recall also that the non-canonical velocity potential $\xi$ can be recovered from the canonical one $\zeta$ 
via the direct relation \eqref{xi-zeta-relation}.

\subsection{Simulations and comparisons}

For the comparison, the full nonlinear system \eqref{ww-nearly-hamiltonian-equation} is solved numerically 
following a high-order spectral approach \citep{CS93}.
The corresponding equations read more explicitly
\begin{eqnarray} \label{HamMotionEta}
\partial_t \eta & = & G(\eta) \xi + \gamma \eta \partial_x \eta \,, \\
\partial_t \xi & = & -g \eta - \frac{1}{2} (\partial_x \xi)^2
+ \frac{1}{2} \frac{\big[ G(\eta) \xi + (\partial_x \eta) (\partial_x \xi) \big]^2}{1 + (\partial_x \eta)^2} 
+ \gamma \eta \partial_x \xi + \gamma \partial_x^{-1} G(\eta) \xi \,.
\label{HamMotionXi}
\end{eqnarray}
These are discretized in space by a pseudo-spectral method based on the fast Fourier transform (FFT).
The computational domain is taken to be $0 \le x \le 2\pi$ with periodic boundary conditions
and is divided into a regular mesh of $N$ collocation points.
The DNO is computed via its series expansion \eqref{series} for which a small number $m$ of terms 
is sufficient to achieve highly accurate results by virtue of its analyticity properties. 
The value $m = 6$ is selected based on previous extensive tests \citep{XG09}.
Time integration of \eqref{HamMotionEta} and \eqref{HamMotionXi} is carried out in the Fourier space 
so that linear terms can be solved exactly by the integrating factor technique.
The nonlinear terms are integrated in time by using a 4th-order Runge--Kutta scheme with constant step $\Delta t$.
More details can be found in \cite{G17,G18}.

The same numerical methods are applied to the envelope equation \eqref{Dysthe},
as well as to the reconstruction procedure, with the same resolutions in space and time.
In particular, the auxiliary system \eqref{normal-form-eta-evolution}--\eqref{normal-form-zeta-evolution} 
is integrated in $s$ by using the same step size $\Delta s = \Delta t$.
While this system of equations may look complicated, its numerical treatment is straightforward
and efficient via the FFT.
Moreover, because this computation is not performed at each instant $t$
(only when data on $\eta$ are required) and because it is performed over a short interval $-1 \le s \le 0$,
the associated cost is insignificant.
Note that, by virtue of the zero-mass assumption \eqref{zero-mass}, indetermination at $k = 0$ in the evaluation
of any quantity involving a Fourier multiplier such as $\partial_x^{-1}$ or $|D|^{-1}$ may be lifted
by simply setting its zeroth-mode component to zero.

To examine the stability of Stokes waves in the presence of a shear current, initial conditions of the form
\begin{equation} \label{init_u}
u(x,0) = B_0 \big[ 1 + 0.1 \cos(\lambda x) \big] \,,
\end{equation}
are specified for \eqref{Dysthe}, where $\lambda$ denotes the wavenumber of some long-wave perturbation.
For the purpose of comparing with the full system \eqref{HamMotionEta}--\eqref{HamMotionXi}, 
initial conditions $\eta(x,0)$ and $\xi(x,0)$ are reconstructed by solving 
\eqref{normal-form-eta-evolution}--\eqref{normal-form-zeta-evolution} from transformed initial data 
\eqref{init_eta}--\eqref{init_zeta} given in terms of \eqref{init_u}.

The following tests focus on the case $(B_0,k_0,\lambda) = (0.002,10,1)$ as considered in the previous stability analysis.
The spatial and temporal resolutions are set to $\Delta x = 0.012$ ($N = 512$) and $\Delta t = 0.005$.
Figure \ref{L2err_wmp1_2} shows the time evolution of the relative $L^2$ error
\begin{equation} \label{errors}
\frac{\| \eta_f - \eta_w \|_2}{\| \eta_f \|_2} \,,
\end{equation}
on $\eta$ between the fully ($\eta_f$) and weakly ($\eta_w$) nonlinear solutions, for various values of $\gamma \lessgtr 0$.
We see that the errors remain under unity (i.e. under 100\%) over the time interval $0 \le t \le 1000$,
noting that the validity of the Dysthe equation deteriorates faster as $\gamma$ is decreased.
This is expected in view of the stability analysis because the solution tends to become more unstable
(and thus more nonlinear) with decreasing $\gamma$.
Development of the BF instability is especially apparent for $\gamma = -1$ and $-2$ 
as indicated by a hump in their error plots.

Comparison of surface elevations $\eta$ predicted from the weakly nonlinear equation \eqref{Dysthe} 
and the full nonlinear system \eqref{HamMotionEta}--\eqref{HamMotionXi}
is presented in Fig. \ref{wave_u0002_k10} for the same set of values of $\gamma$
at their respective times of maximum wave growth.
The perturbed Stokes wave at $t = 0$ (which is the same initial condition for all cases considered)
is depicted in Fig. \ref{wave_u0002_k10}(a).
These results are consistent with our previous observations from Figs. \ref{BF_inst} and \ref{L2err_wmp1_2}.
Excitation and growth of the most unstable sideband mode $\lambda = 1$ (according to the stability analysis)
are clearly revealed in these plots.
A more negative $\gamma$ promotes the BF instability (by making it happen sooner with a stronger wave amplification),
while a more positive $\gamma$ tends to reduce and even offset it.
In all these cases, the Dysthe model is found to provide a very good approximation up to at least $t = 1000$.
As expected, for $\gamma = -2$, discrepancies are more pronounced due to the higher nonlinearity reached:
a slight phase lag and drop in wave amplitude can be discerned around the main peak at $t = 500$.

It is suitable to compare our Hamiltonian Dysthe equation \eqref{Dysthe} 
with another related model that has recently been derived by \cite{CCK18,CM20} in the same physical setting.
Note that these authors additionally considered surface tension but we will only examine the gravity-wave version of their model.
Moreover, because they expressed their model in a form that contains a first derivative in time as well as a mixed derivative 
in space and time (see Eq. 2.37 in \cite{CCK18}), we find it more appropriate to rewrite it in a more standard form
with a single time derivative (as it is typically so for the Dysthe equation \citep{D79}) to allow for a fairer comparison.
We also take into account the fact that vorticity in the mathematical formulation used by \cite{CCK18}
is defined as the opposite of ours.
The resulting model for the first-harmonic envelope $\eta_1$ is given by
\begin{eqnarray} \label{curtis}
\partial_t \eta_1 & = & -c_g \partial_x \eta_1 - i \frac{c_g^2}{d_0} \partial_x^2 \eta_1
+ 2 \frac{c_g^3}{d_0^2} \partial_x^3 \eta_1
- i \frac{\widetilde \alpha_0}{d_0} |\eta_1|^2 \eta_1 \nonumber \\
& & - \frac{\widetilde \alpha_1}{d_0} |\eta_1|^2 \partial_x \eta_1
- \frac{\widetilde \alpha_2}{d_0} \eta_1^2 \partial_x \overline \eta_1 
+ i \frac{\widetilde \alpha_3}{d_0} \eta_1 \Hil \partial_x |\eta_1|^2 \,,
\end{eqnarray}
where
\[
d_0 = 2 \Omega_0 - \gamma \,, \quad c_g = \frac{1}{d_0} \,, \quad d_1 = \frac{d_0}{1 + \gamma c_g} \,, \quad
\widetilde \alpha_0 = \alpha_0 - \gamma^2 k_0 d_0 d_1 \,,
\]
\[
\widetilde \alpha_1 = \alpha_1 - \gamma^2 k_0 d_1 (c_g + \alpha_d \gamma d_1)
- \gamma^2 d_1 (2 d_0 + 3 k_0 c_g) + 4 c_g (\gamma^2 k_0 d_1 - \alpha_0 c_g) \,,
\]
\[
\widetilde \alpha_2 = \alpha_2 + \gamma^2 k_0 d_1 (c_g + \alpha_d \gamma d_1)
- \gamma^2 d_1 (d_0 + k_0 c_g) + 2 c_g (\gamma^2 k_0 d_1 - \alpha_0 c_g) \,,
\]
\[
\widetilde \alpha_3 = \alpha_3 - \gamma k_0 c_g d_1 (d_0 +  d_1) \,.
\]
The reader is directed to \cite{CCK18} where the expressions of $\alpha_d$ and $\alpha_j$ ($j=0,\dots,3$) can be found.
Note that $\Hil \partial_x = |D|$ for the nonlocal term in \eqref{curtis}.
For the purpose of comparing with the full system \eqref{HamMotionEta}--\eqref{HamMotionXi},
we have also re-expressed Curtis et al.'s model in a fixed reference frame,
hence the additional advection term in \eqref{curtis} as compared to Eq. (2.37) in \cite{CCK18}.
In the following discussion, we will refer to \eqref{curtis} as the ``classical'' Dysthe equation for this problem,
because it is not Hamiltonian and has the same typical form as in the irrotational case.
Furthermore, its derivation is based on a perturbative Stokes-type expansion for the dependent variables
$\eta$ and $\xi$, which is similar to the classical derivation by the method of multiple scales \citep{D79}.
Indeed, following \cite{CCK18}, the surface elevation and velocity potential at any instant $t$ 
can be reconstructed perturbatively from $\eta_1$ as
\begin{eqnarray} \label{stokes}
\eta(x,t) & = & \gamma d_1 |\eta_1|^2 + 2 \, {\rm Re} \big( \eta_1 e^{i \theta} + \ell_0 \eta_1^2 e^{2 i \theta} + \dots \big) \,,
\nonumber \\
\xi(x,t) & = & \Hil \left( \frac{\Omega_0}{k_0} \eta - \frac{\Omega_0}{k_0} \eta \Hil \partial_x \eta
- \frac{\gamma}{2} \eta^2 + \frac{\Omega_0}{2 k_0} \Hil \partial_x \big( \eta^2 \big) + \dots \right) \,,
\end{eqnarray}
for which the expression of $\ell_0$ can be found in \cite{CCK18}, and only contributions from up to the second harmonics
are included here because \cite{CCK18} did not provide expressions for contributions from higher harmonics.
The phase function is given by $\theta = k_0 x - \Omega_0 t$.
The ``classical'' reconstruction procedure based on \eqref{stokes} clearly differs from the present approach.
It is more explicit and thus computationally more efficient but is perturbative.
Contributions at each order up to the desired one need to be derived and their expressions become increasingly complicated.
On the other hand, our Hamiltonian procedure requires solving an auxiliary system of PDEs
to reconstruct $\eta$ and $\zeta$ (or $\xi$) from $u$ but it is non-perturbative.
Indeed, Eqs. \eqref{normal-form-eta-evolution} and \eqref{normal-form-zeta-evolution} 
constitute an exact representation of the Birkhoff normal form transformation that eliminates non-resonant triads in this problem.

As an illustration, Fig. \ref{L2comp_u0002_k10} compares the $L^2$ errors \eqref{errors} on $\eta$ from 
the classical and Hamiltonian Dysthe equations in the large-vorticity cases $\gamma = \pm 2$.
For each of these models, the error is calculated relative to the fully nonlinear solution with respective
initial conditions $\eta(x,0)$ and $\xi(x,0)$.
These are provided by \eqref{stokes} with
\[
\eta_1(x,0) = \frac{A_0}{2} \big[ 1 + 0.1 \cos(\lambda x) \big] \,,
\]
when \eqref{curtis} is tested against \eqref{HamMotionEta}--\eqref{HamMotionXi}.
Recall that $A_0$ and $B_0$ are related through \eqref{amplitude}.
We use the same numerical methods as described earlier 
(and specify the same resolutions in space and time) to solve \eqref{curtis} and evaluate \eqref{stokes}.
For $\gamma = -2$, both Dysthe solutions are found to perform similarly,
with the error from the Hamiltonian model being slightly lower than that from the classical model.
The relatively quick loss of accuracy in this case, which is common to both models
(with errors reaching near 50\% at $t \simeq 500$) should be attributed to deterioration of the Dysthe approximation
during development of the BF instability, rather than to the reconstruction procedure.
By contrast, for $\gamma = +2$, the errors remain small and do not vary much over the time interval $0 \le t \le 1000$,
which is expected considering that the solution is more stable in this case.
We see however that the present approach outperforms the classical one by about an order of magnitude.
In all these error plots, the seemingly sharp values near $t = 0$ are already an indication of the level of approximation
associated with the different equations,
as they represent adjustment of the full system \eqref{HamMotionEta}--\eqref{HamMotionXi} to the imposed initial conditions
during early stages of the simulation.

The corresponding surface profiles are depicted in Fig. \ref{wavec_u0002_k10_wm2} for the unstable case $\gamma = -2$,
with predictions from each Dysthe model being compared to the fully nonlinear solution.
Snapshots of $\eta$ at $t = 390$ (early stage of BF instability), $t = 500$ (around the time of maximum wave growth)
and $t = 1000$ (near the end of the quasi-recurrent cycle of modulation-demodulation) are presented.
The satisfactory performance of both Dysthe solutions as indicated in this figure
is consistent with the error plots in Fig. \ref{L2comp_u0002_k10}.
A noticeable discrepancy between the weakly and fully nonlinear predictions 
is a phase lag that tends to develop over time.
Otherwise, salient features of the wave dynamics (including the shape of the steep wave at $t = 500$)
seem to be well captured, even in this highly focusing regime.
Regarding the comparison of surface profiles for $\gamma = +2$, these look indistinguishable from 
Fig. \ref{wave_u0002_k10}(f) at the graphical scale and thus are not displayed for convenience.

We point out in passing that the main purpose of these tests is not to show whether one modulational approach is better than the other.
In particular, regarding the reconstruction procedure for the classical Dysthe equation,
we understand that adding contributions from higher harmonics to formulas \eqref{stokes}
would likely improve their accuracy and lead to closer agreement with the full system.
Rather, a goal here is to validate our new Hamiltonian approach against other existing formulations.
As a byproduct of this comparison, given the overall positive assessment based on Figs. \ref{L2comp_u0002_k10}
and \ref{wavec_u0002_k10_wm2}, we in turn provide an independent validation of Curtis et al.'s model.
Such a validation was not conducted in their earlier study \citep{CCK18,CM20}.

It is comforting to see that the solution of \eqref{normal-form-eta-evolution}--\eqref{normal-form-zeta-evolution} 
helps achieve an accurate computation of the free surface in our Hamiltonian framework,
which was not obvious considering the rather lengthy expressions of \eqref{normal-form-eta-evolution}--\eqref{normal-form-zeta-evolution}.
The good agreement found also confirms the validity of the zero-mass assumption \eqref{zero-mass}
since it is used to evaluate nonlocal terms in \eqref{normal-form-eta-evolution}--\eqref{normal-form-zeta-evolution}.
To further demonstrate the effectiveness of this reconstruction scheme 
(which we will refer to as full reconstruction by solving \eqref{normal-form-eta-evolution}--\eqref{normal-form-zeta-evolution}),
we now test the Hamiltonian Dysthe equation \eqref{Dysthe} by simply using \eqref{init_eta}--\eqref{init_zeta}
to recover $\eta$ and $\zeta$ from $u$ at any instant $t$ 
(which we will refer to as partial reconstruction).
This simplified procedure is equivalent to retaining only contributions from the first harmonics
in the representation of $\eta$ and $\zeta$.

The $L^2$ errors \eqref{errors} associated with these two versions of our Hamiltonian approach are illustrated 
in Fig. \ref{L2comp1st_u0002_k10} for $\gamma = \pm 2$.
We have made sure again that suitable initial conditions are specified for the full system 
\eqref{HamMotionEta}--\eqref{HamMotionXi} when comparing it to each version.
These results confirm that the decline in performance (for partial vs. full reconstruction of $\eta$) can be considerable.
The difference is found to be by about an order of magnitude for $\gamma = -2$
and by more than two orders of magnitude for $\gamma = +2$.
In both cases, the errors quickly grow to exceed 100\% at some point during the time interval $0 \le t \le 1000$.

Examination of the surface profiles obtained from partial reconstruction as compared to the fully nonlinear solution 
is provided in Fig. \ref{wave1st_u0002_k10_wm2} for $\gamma = \pm 2$.
Consistent with the error plots in Fig. \ref{L2comp1st_u0002_k10}, we see that the discrepancies in wave amplitude and phase
tend to develop faster. The phase lag is clearly noticeable and affects the entire wave train, even in the stabilizing case $\gamma = +2$.
It is so severe for $\gamma = -2$ that the weakly nonlinear solution appears completely out of phase at $t = 1000$
during the near-recurrent stage.

Finally, the time evolution of the relative error
\[
\frac{\Delta \calH}{\calH_0} = \frac{|\calH - \calH_0|}{\calH_0} \,,
\]
on energy \eqref{reduced-H-for-dysthe-new} associated with the Hamiltonian Dysthe equation \eqref{Dysthe} 
is shown in Fig. \ref{ener_u0002_k10_wmp1_2} for various values of $\gamma$.
Integrals in \eqref{reduced-H-for-dysthe-new} and in the $L^2$ norm \eqref{errors} are computed via the trapezoidal rule 
over the periodic cell $[0,2\pi]$.
The reference value $\calH_0$ denotes the initial value of \eqref{reduced-H-for-dysthe-new} at $t = 0$.
Overall, $\calH$ is very well conserved in all these cases.
The gradual loss of accuracy over time, which becomes more pronounced as $\gamma$ is decreased,
is likely due to amplification of numerical errors triggered by the BF instability.

\section{Conclusion}

Starting from the basic Hamiltonian formulation of the water wave problem with constant vorticity as proposed by \cite{W07,CIP08},   
we derive a Hamiltonian version of the Dysthe equation (a higher-order NLS equation) 
for the nonlinear modulation of two-dimensional gravity waves on deep water, in the presence of a background uniform shear flow. 
The resulting model exhibits a well-defined symplectic structure and conserves an energy (i.e. the reduced Hamiltonian) over time.
Our methodology, introduced recently for two- and three-dimensional irrotational gravity waves \citep{CGS21a, GKSX21,GKS22},
consists in performing a sequence of canonical transformations that involve a reduction to normal form
(devoid of non-resonant triads) and use of a modulational Ansatz together with a scale separation lemma.
A novelty of our approach is a direct reconstruction of the surface variables from the wave envelope 
through inversion of the third-order normal form transformation. 
This reconstruction requires solving an auxiliary Hamiltonian system of PDEs, for which we provide an explicit derivation.
Such a procedure differs from the classical one where physical quantities like the surface elevation 
are reconstructed perturbatively in terms of a Stokes expansion.  
As a consequence, both steps (solving for the wave envelope and recovering the surface elevation) 
consistently fit within a Hamiltonian framework.

To validate our approximation, we perform numerical simulations of this Hamiltonian Dysthe equation and compare them to computations 
based on the full water wave system and another related Dysthe equation recently derived by \cite{CCK18} in the same setting.  
For a range of values of the vorticity, we examine the long-time dynamics of perturbed Stokes waves
and find very good agreement, thus providing a verification for both Dysthe models. 
In particular, the performance of our Hamiltonian model is found to be quite satisfactory over the entire range considered.
We observe that the presence of vorticity clearly has an effect on the BF instability of Stokes waves on deep water.  
Consistent with results from previous studies, a counter-propagating shear flow (negative vorticity) tends to enhance this instability 
as it amplifies the growth rate and enlarges the instability region to higher sideband wavenumbers, 
while a co-propagating current (positive vorticity) tends to stabilize it.  
We hope this Hamiltonian Dysthe equation may serve as an efficient tool to study wave-current interactions in future applications.
As subsequent work, it would be of interest to extend the present method to the situation of constant finite depth with possibly surface tension.
For this problem, the reduction to normal form is expected to be significantly more complicated.

\appendix

\section{Proof of Lemma \ref{lemma-T2-approximation}}\label{appA}
\label{appendix-proof-expansion-lemma}

We provide here the main steps in the proof of Lemma \ref{lemma-T2-approximation}. 

\subsection{Computation of $T_2^{(1)}$} 

First, we notice that the terms $S_{(-1-2)12}$, $S_{12(-1-2)}$, $S_{(-3-4)34}$ and $S_{34(-3-4)}$ in  \eqref{T-2-1-term} are of order $\calO (\varepsilon^2)$. 
 Indeed, under the modulational Ansatz \eqref{modulation}, we have 
\begin{equation*}
\begin{aligned}
S_{(-1-2)12} = \frac{1-\sgn (k_1+k_2) \sgn (k_2)}{a_1 a_2 a_{1+2}} \Big( -(k_1+k_2)k_2 a_{1+2}^2 a_2^2 - \frac{\gamma}{2} k_1 a_1^2 \Big) = \calO (\varepsilon^2) \,,
\end{aligned}
\end{equation*}
where, from \eqref{expansion-identities}, we have
$1-\sgn (k_1+k_2) \sgn (k_2) = \calO (\varepsilon^2)$ 
and
$a_1, a_2 , a_{1+2} = \calO (1)$.
The computations of $S_{12(-1-2)}$, $S_{(-3-4)34}$ and $S_{34(-3-4)}$ are similar.  We thus skip such terms as we approximate $T_2^{(1)}$ up to  order $\calO (\varepsilon)$ only.
By contrast, the terms $S_{2(-1-2)1}$ and $S_{4(-3-4)3}$ are of order $\calO (1)$, and they both contribute to the $\calO (\varepsilon)$-expansion of  $T_2^{(1)}$.
For the expansion of $S_{2(-1-2)1}$, we use \eqref{expansion-identities} and obtain
\begin{equation*}
\begin{aligned}
S_{2(-1-2)1} &= \frac{(2k_0)^{3/2} (2\omega_0^2 + \gamma \omega_{2k_0})}{2\omega_0 \sqrt{\omega_{2k_0}}}\\
&\quad \times \left[ 1 - \frac{\varepsilon g}{4} \Big( \frac{1}{\sq^2} + \frac{1}{\sp^2} - \frac{3}{gk_0} \Big) (\lambda_1+\lambda_2) + \frac{\varepsilon g \Omega_{2k_0}}{\sp (2\sq^2 + \gamma \sp)} (\lambda_1+\lambda_2) \right] \,,
\end{aligned}
\end{equation*}
with a similar expression for $S_{4(-3-4)3}$ where $(\lambda_1+\lambda_2)$ is replaced by $(\lambda_3+\lambda_4)$. It remains to get an expansion for the bracket on the second line of \eqref{T-2-1-term}. Using \eqref{expansion-identities}, we have 
\begin{equation*}
\begin{aligned}
\frac{1}{\Omega_{1}+ \Omega_{2}+ \Omega_{-1-2}} +& \frac{1}{\Omega_{3}+ \Omega_{4}+ \Omega_{-3-4}} \\[2pt] 
&=\frac{2}{2\Omega_0 + \Omega_{-2k_0}}
\left(1 - \frac{\varepsilon g (\sq+\sp)}{4 \sq \sp (2\Omega_0 + \Omega_{-2k_0})}(\lambda_1+\lambda_2+\lambda_3+\lambda_4) \right) \,.
\end{aligned}
\end{equation*}
We substitute the above estimates into the expression \eqref{T-2-1-term} for $T_2^{(1)}$  and use that
\begin{equation*}
\int (\lambda_1+\lambda_2+\lambda_3+\lambda_4) z_1 z_2 \overline z_3 \overline z_4 \delta_{1+2-3-4} dk_{1234} =  2\int (\lambda_2+\lambda_3) z_1 z_2 \overline z_3 \overline z_4 \delta_{1+2-3-4} dk_{1234} \,,
\end{equation*}
to get 
\begin{equation*}
\begin{aligned}
\int T_2^{(1)} z_1 z_2 \overline z_3 \overline z_4 \delta_{1+2-3-4} dk_{1234} &=  \int \Big( 
 c_1^l   + \varepsilon c_1^r  (\lambda_2 + \lambda_3) \Big) z_1 z_2 \overline z_3 \overline z_4 \delta_{1+2-3-4} dk_{1234}
  + \calO(\varepsilon^2) \,,
\end{aligned}
\end{equation*}
which identifies to 
 \eqref{T-2-1-approximation} in terms of the variable $U$. 
A similar calculation is performed for $T_2^{(2)}$.

\subsection{Computation of $T_2^{(3)}$}

We estimate each term in \eqref{T-2-3-term} under the modulational Ansatz \eqref{modulation}. Due to dependence on $(k_1 - k_3)$, these terms are of different orders compared to the above computations for $T_2^{(1)}$. Indeed, we show that  $A_{(1-3)(-1)3}$ and $A_{(4-2)(-4)2}$ are of order $\calO (\varepsilon^{1/2})$, and the bracket in \eqref{T-2-3-term} is of order $\calO (\varepsilon^{-1})$. 

We start with $A_{(1-3)(-1)3}$. Using the relation  \eqref{S123-A123}, we need to compute $S_{(1-3)(-1)3}$, $S_{3(1-3)(-1)}$ and $S_{(-1)3(1-3)}$. 
We immediately rule out the contribution from  $S_{3(1-3)(-1)}$ as it is of order $\calO (\varepsilon^{5/2})$. The remaining terms are combined using \eqref{S123-A123} as follows
\begin{equation*}
\begin{aligned}
S_{(1-3)(-1)3} - S_{(-1)3(1-3)} = \frac{1}{a_1 a_3 a_{1-3}} \Big(& \omega_{1-3} (\omega_3 - \omega_1) + \frac{\gamma}{2} \sgn (k_1-k_3) (\omega_1 - \omega_3) \\
& + \frac{\gamma}{2} (\omega_1 + \omega_3) + \sgn (k_1-k_3) \omega_{1-3} (\omega_1 + \omega_3) \Big) \,.
\end{aligned}
\end{equation*}
 Using expansions \eqref{expansion-identities}, 
 we obtain 
\begin{align*}
&S_{(1-3)(-1)3} - S_{(-1)3(1-3)} = \frac{\sqrt{2}k_0 \varepsilon^{1/2} |\lambda_1 - \lambda_3|^{1/2}}{\sqrt{ |\gamma|} \omega_0}
\Big[ 1 - \frac{\varepsilon}{4} \Big( \frac{g}{\omega_0^2} - \frac{2}{k_0} \Big) (\lambda_1+\lambda_3) - \frac{g}{\gamma^2} \varepsilon |\lambda_1-\lambda_3| \Big] \\
&\qquad \times \Big[ |\gamma| \omega_0 (\sgn(\gamma) + \sgn (k_1-k_3))  
+ \frac{\varepsilon g |\gamma|}{4 \omega_0} (\sgn(\gamma) + \sgn (k_1-k_3)) (\lambda_1+\lambda_3)
\\
& \qquad \quad
+ \frac{\varepsilon g \gamma}{4\omega_0} (\sgn (k_1-k_3) - \sgn(\gamma)) (\lambda_1 - \lambda_3) + \frac{2\varepsilon g \omega_0}{|\gamma|} (\lambda_1-\lambda_3) \Big] \,,
\end{align*}
and a similar expression for $A_{(4-2)(-4)2}$ with $(k_1, k_3)$ replaced by $(k_4, k_2)$.
Furthermore, using that 
$(\sgn(k_1-k_3)+ \sgn(\gamma)) (\sgn(k_1-k_3)- \sgn(\gamma)) = 0$,
several  terms in the product $A_{(1-3)(-1)3} A_{(4-2)(-4)2}$ vanish, and 
{\small{
\begin{align*}
A_{(1-3)(-1)3} A_{(4-2)(-4)2} = \frac{\varepsilon k_0^2 \gamma (\lambda_1 - \lambda_3)}{16 \pi} (1+\sgn (\gamma) \sgn(k_1-k_3))
  \Big( 1 + \frac{\varepsilon}{2k_0} (\lambda_1 + \lambda_2+ \lambda_3 + \lambda_4) \Big) +\calO (\varepsilon^2) \,.
\end{align*}
}}
In addition, for \eqref{T-2-3-term}, we have 
\begin{equation*}
\begin{aligned}
\frac{1}{\Omega_{3- 1}+ \Omega_{1}- \Omega_{3}} + \frac{1}{\Omega_{2-4}+ \Omega_{4}- \Omega_{2}} &= \frac{2 \gamma \omega_0 }{\varepsilon (\lambda_1 - \lambda_3) g \Omega_0}\\
& \times \left( 1 + \frac{\varepsilon g \gamma }{16 \omega_0^2 \Omega_0} (\lambda_1 + \lambda_2+ \lambda_3 + \lambda_4) + \frac{\varepsilon g \omega_0}{\gamma^2 \Omega_0}|\lambda_1-\lambda_3| \right) \,.
\end{aligned}
\end{equation*}
We combine these estimates according to \eqref{T-2-3-term} and Eq. \eqref{T-2-1-approximation} follows.

\section*{Acknowledgements}
A. K. thanks McMaster University for its
support.
C. S. is partially supported by the NSERC (grant
number 2018-04536) and a Killam Research Fellowship from the Canada Council for the Arts. 






\begin{figure}
\centering
\subfloat{\includegraphics[width=.5\linewidth]{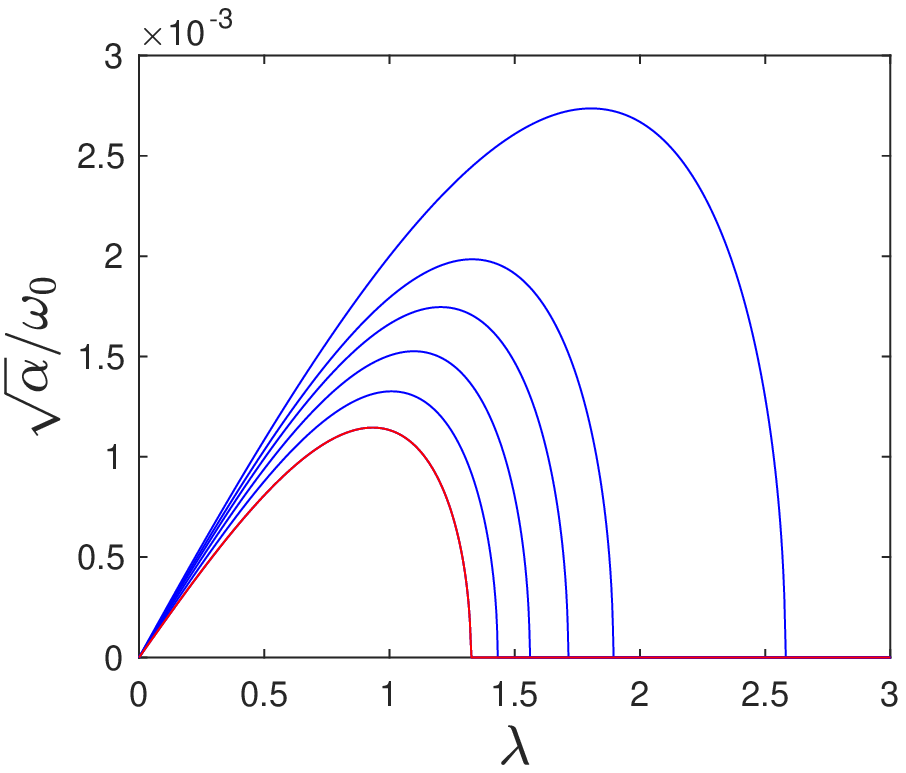}}
\hfill
\subfloat{\includegraphics[width=.5\linewidth]{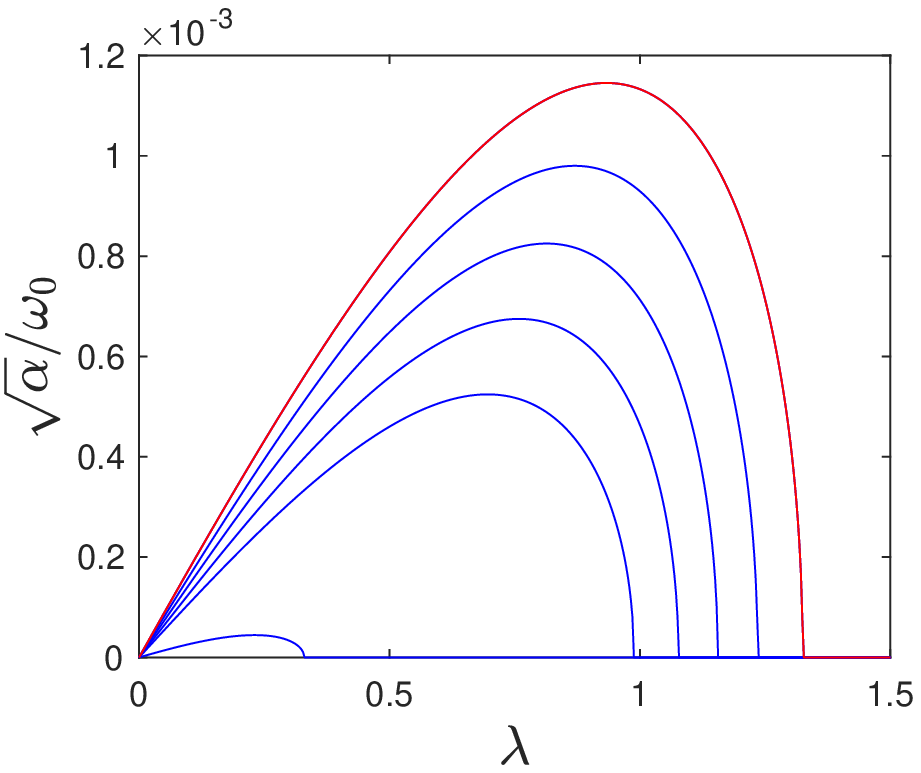}}
\caption{Regions of modulational instability according to \eqref{BFCond} for $(B_0,k_0) = (0.002,10)$.
The blue curves correspond to $\gamma \neq 0$ while the red curve corresponds to $\gamma = 0$.
Left panel: $\gamma = \{ -0.5, -1, -1.5, -2, -3.5 \}$ (expanding curves with decreasing $\gamma$). 
Right panel: $\gamma = \{ +0.5, +1, +1.5, +2, +3.5 \}$ (shrinking curves with increasing $\gamma$).}
\label{BF_inst}
\end{figure}

\begin{figure}
\centering
\includegraphics[width=.5\linewidth]{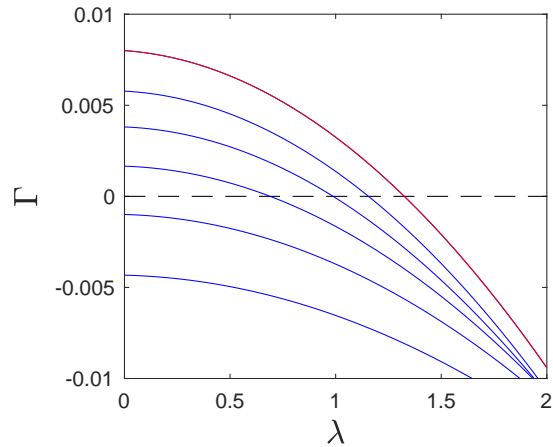}
\caption{Plots of $\Gamma$ versus $\lambda$ for $(B_0,k_0) = (0.002,10)$
and $\gamma = \{ 0, +1, +2, +3, +4, +5 \}$ (falling curves with increasing $\gamma$).
The red curve corresponds to $\gamma = 0$.}
\label{factor_inst}
\end{figure}

\begin{figure}
\centering
\includegraphics[width=.5\linewidth]{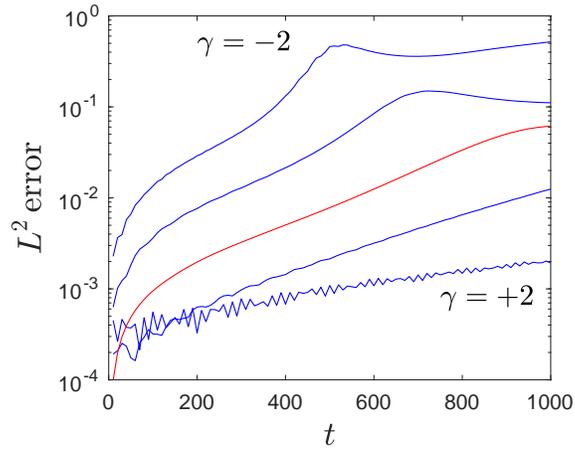}
\caption{Relative $L^2$ errors on $\eta$ between fully and weakly nonlinear solutions
for $(B_0,k_0,\lambda) = (0.002,10,1)$ and $\gamma = \{ -2, -1, 0, +1, +2 \}$.
The red curve corresponds to $\gamma = 0$.}
\label{L2err_wmp1_2}
\end{figure}

\begin{figure}
\centering
\subfloat[$\gamma = 0$, $t = 0$]{\includegraphics[width=.5\linewidth]{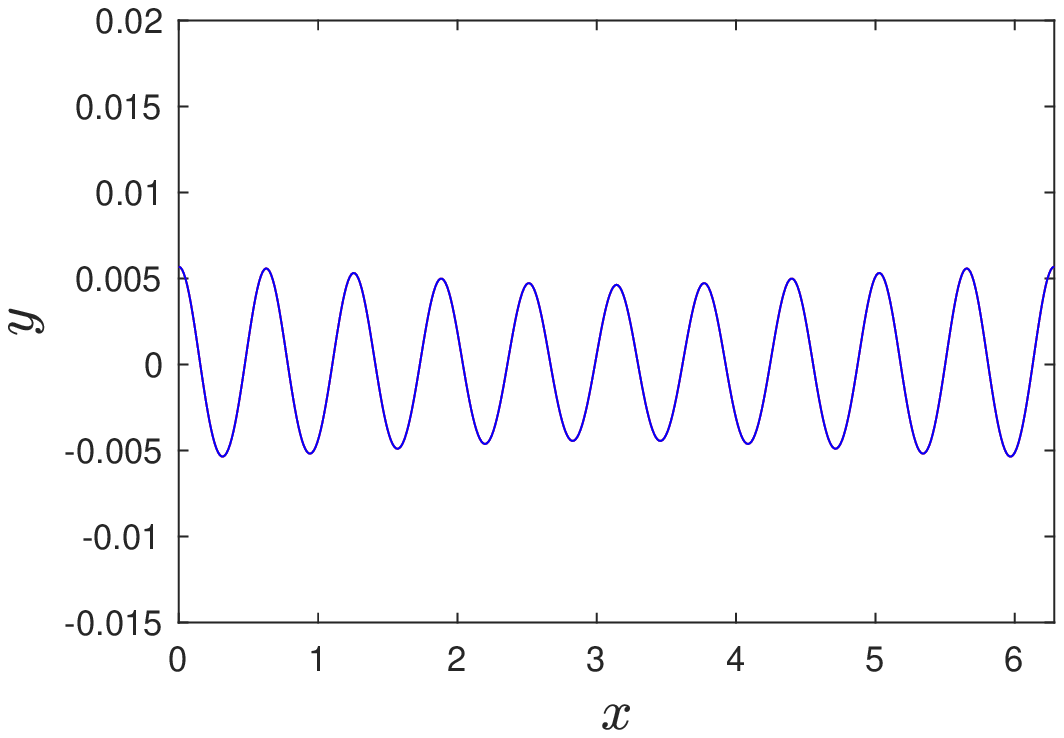}}
\hfill
\subfloat[$\gamma = 0$, $t = 940$]{\includegraphics[width=.5\linewidth]{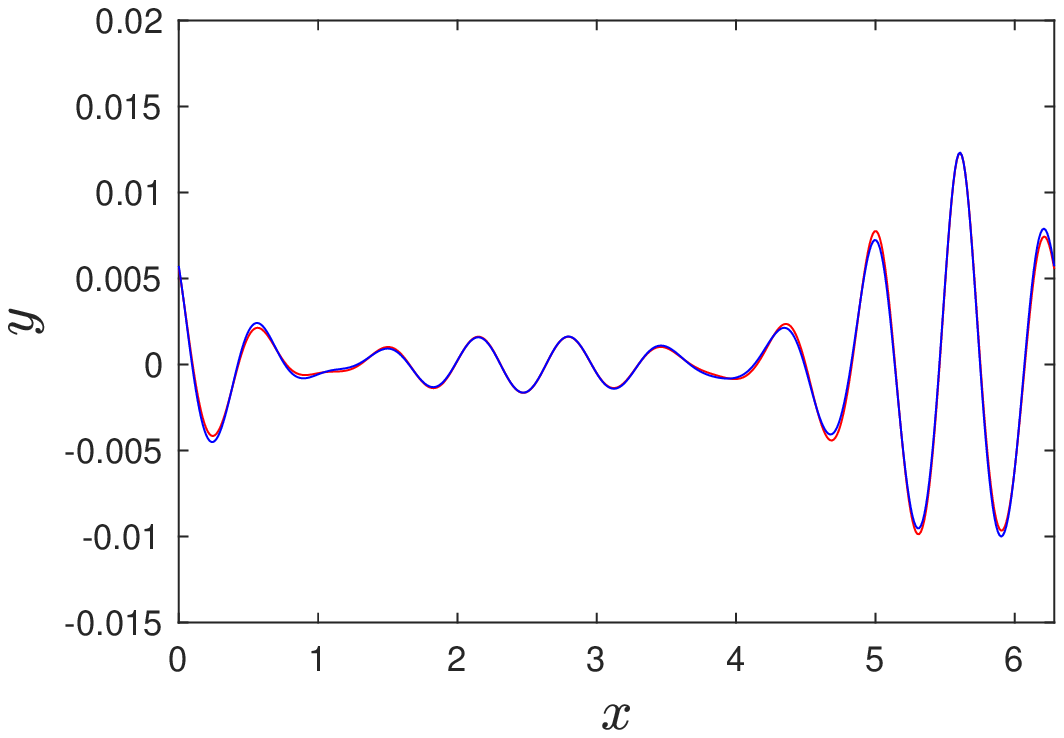}}
\hfill
\subfloat[$\gamma = -1$, $t = 680$]{\includegraphics[width=.5\linewidth]{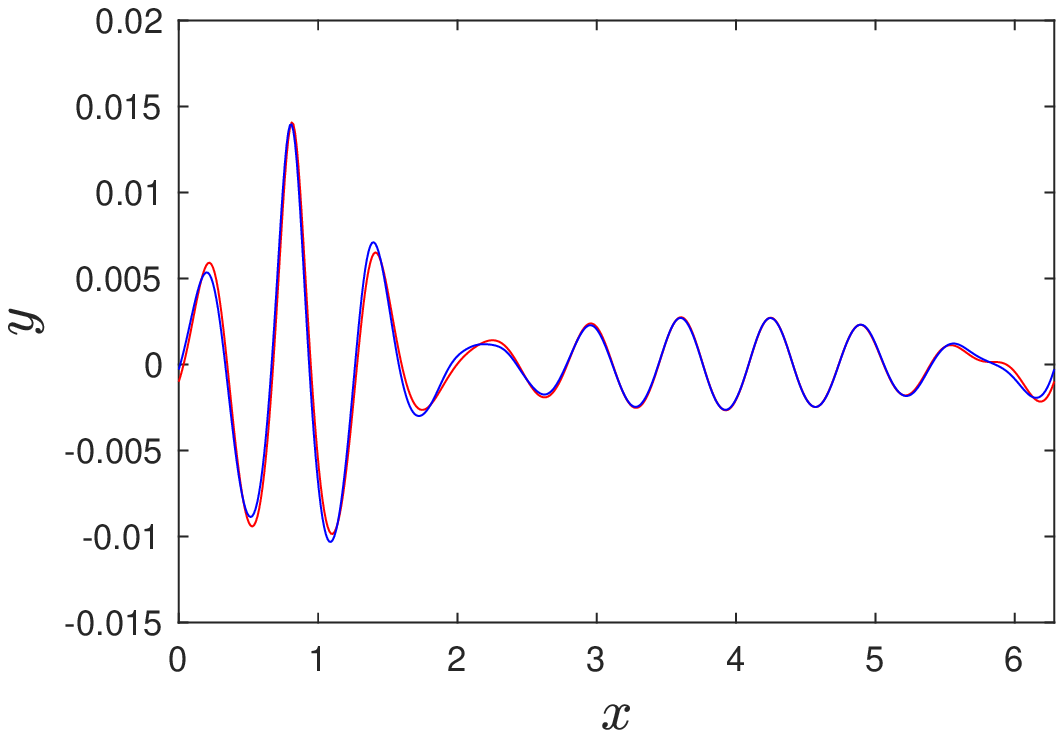}}
\hfill
\subfloat[$\gamma = -2$, $t = 500$]{\includegraphics[width=.5\linewidth]{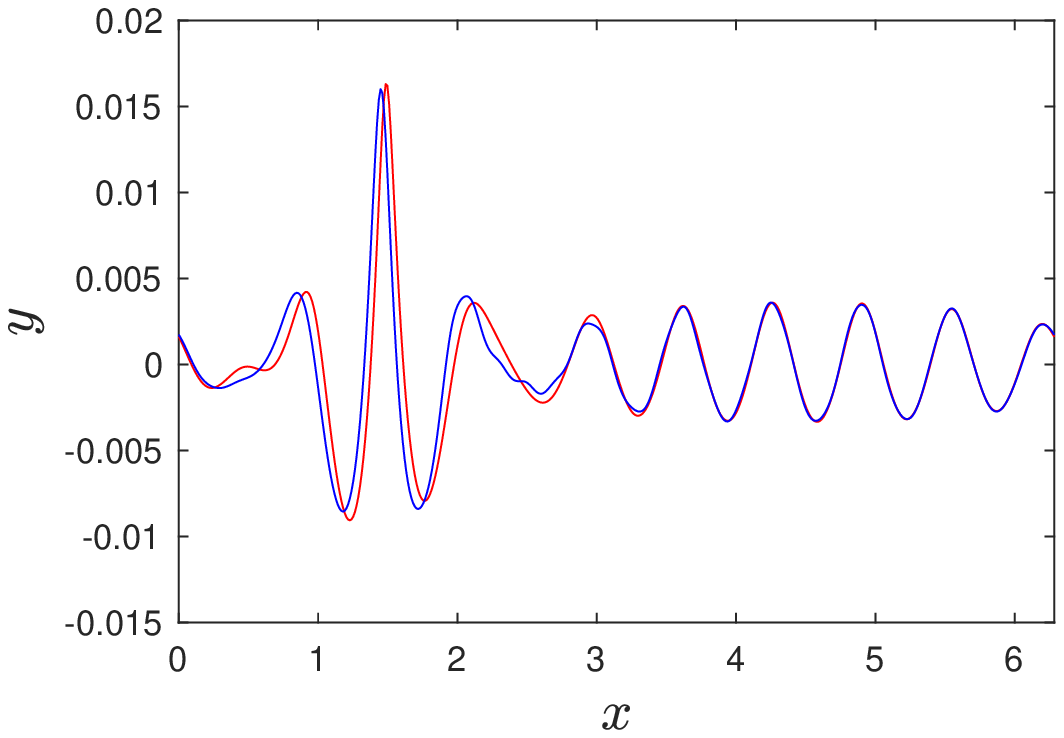}}
\hfill
\subfloat[$\gamma = +1$, $t = 1000$]{\includegraphics[width=.5\linewidth]{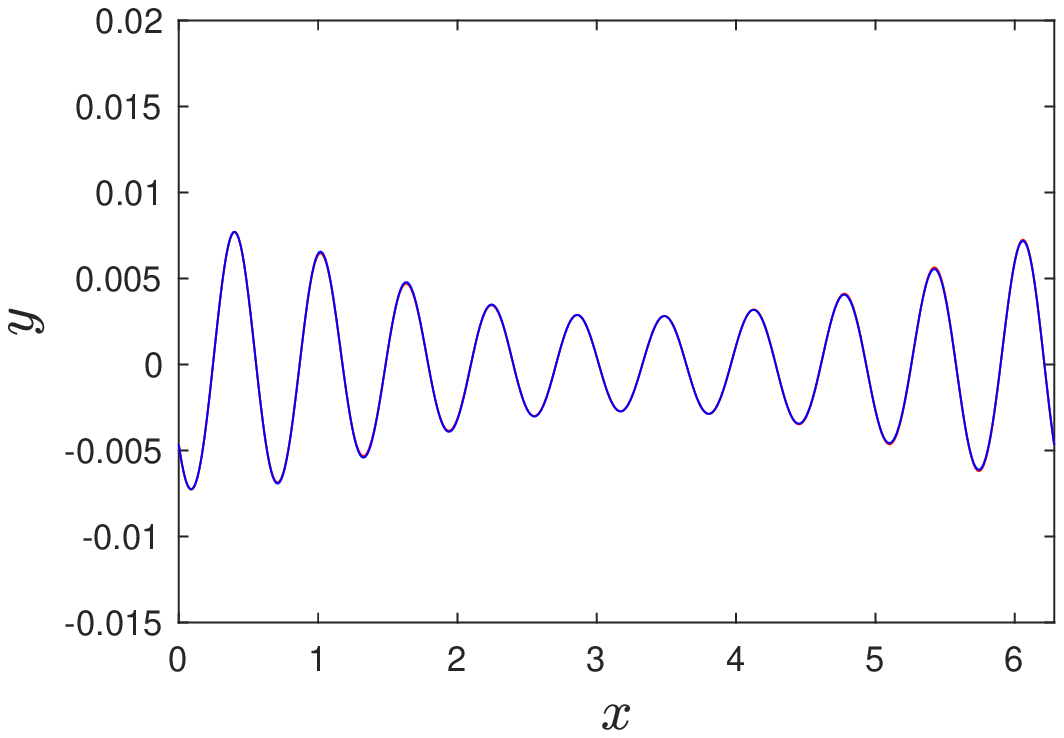}}
\hfill
\subfloat[$\gamma = +2$, $t = 1000$]{\includegraphics[width=.5\linewidth]{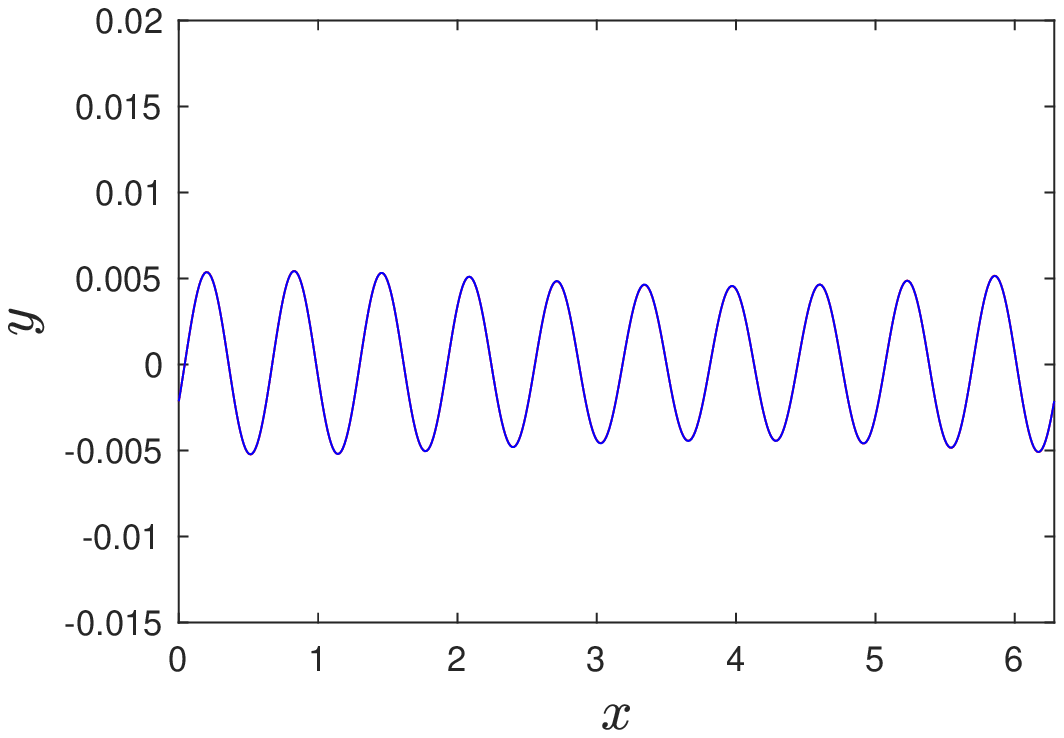}}
\caption{Comparison of surface elevations $\eta$ between fully and weakly nonlinear solutions
for $(B_0,k_0,\lambda) = (0.002,10,1)$ with
(a) $\gamma = 0$ ($t = 0$), (b) $\gamma = 0$ ($t = 940$), (c) $\gamma = -1$ ($t = 680$),
(d) $\gamma = -2$ ($t = 500$), (e) $\gamma = +1$ ($t = 1000$), (f) $\gamma = +2$ ($t = 1000$).
The blue curve represents the Hamiltonian Dysthe equation while the red curve represents the full nonlinear system.}
\label{wave_u0002_k10}
\end{figure}

\begin{figure}
\centering
\subfloat{\includegraphics[width=.5\linewidth]{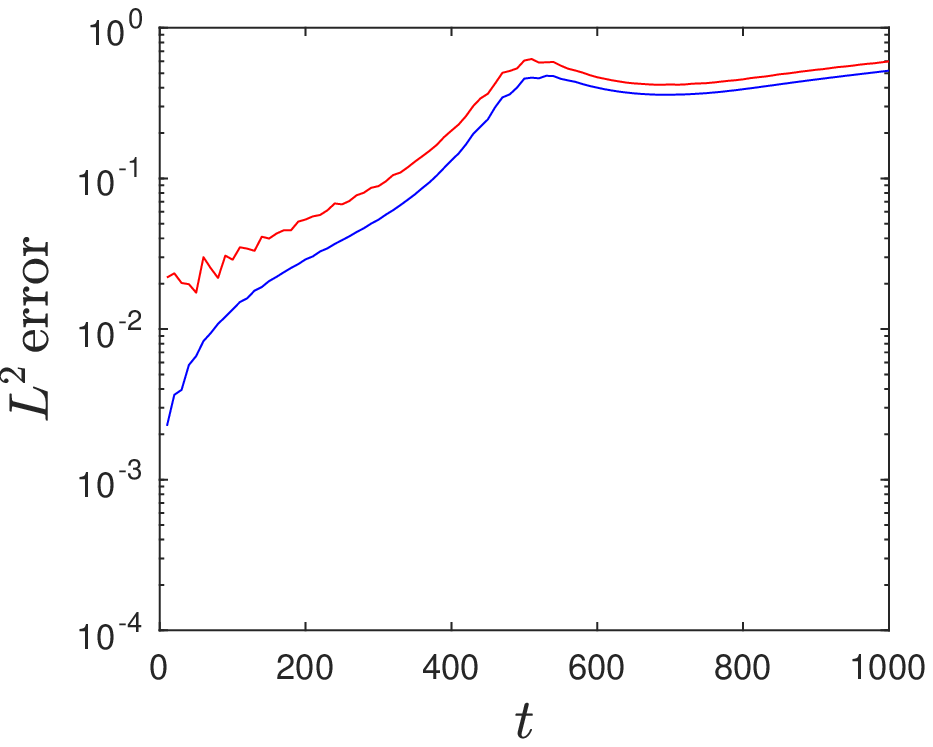}}
\hfill
\subfloat{\includegraphics[width=.5\linewidth]{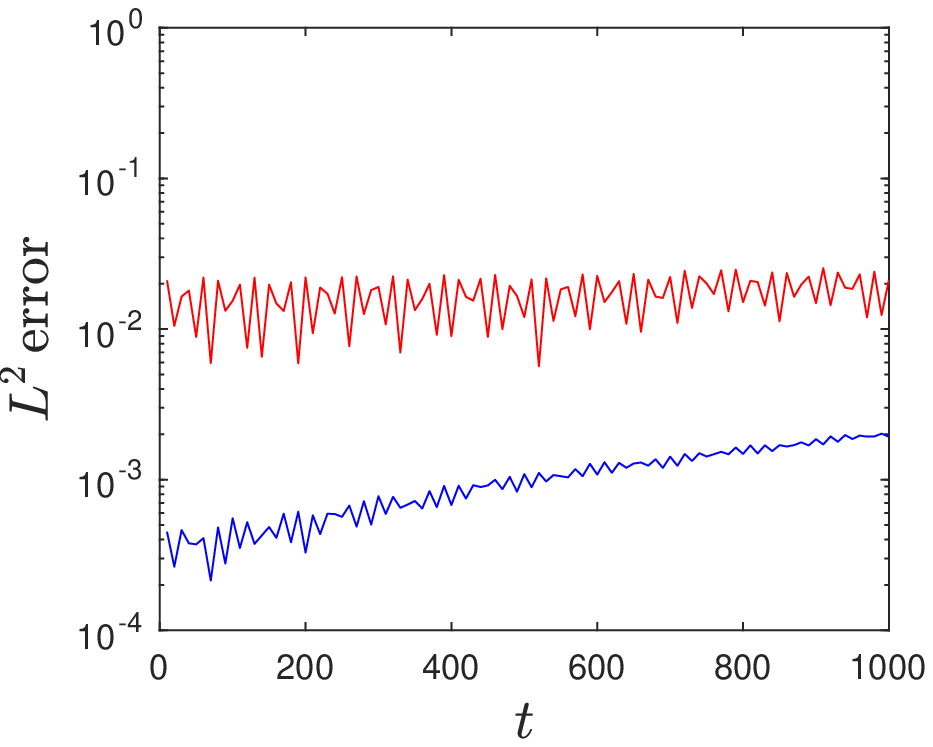}}
\caption{Relative $L^2$ errors on $\eta$ between fully and weakly nonlinear solutions
for $(B_0,k_0,\lambda) = (0.002,10,1)$.
The blue curve represents the Hamiltonian Dysthe equation while the red curve represents the classical Dysthe equation.
Left panel: $\gamma = -2$. Right panel: $\gamma = +2$.}
\label{L2comp_u0002_k10}
\end{figure}

\begin{figure}
\centering
\subfloat{\includegraphics[width=.33\linewidth]{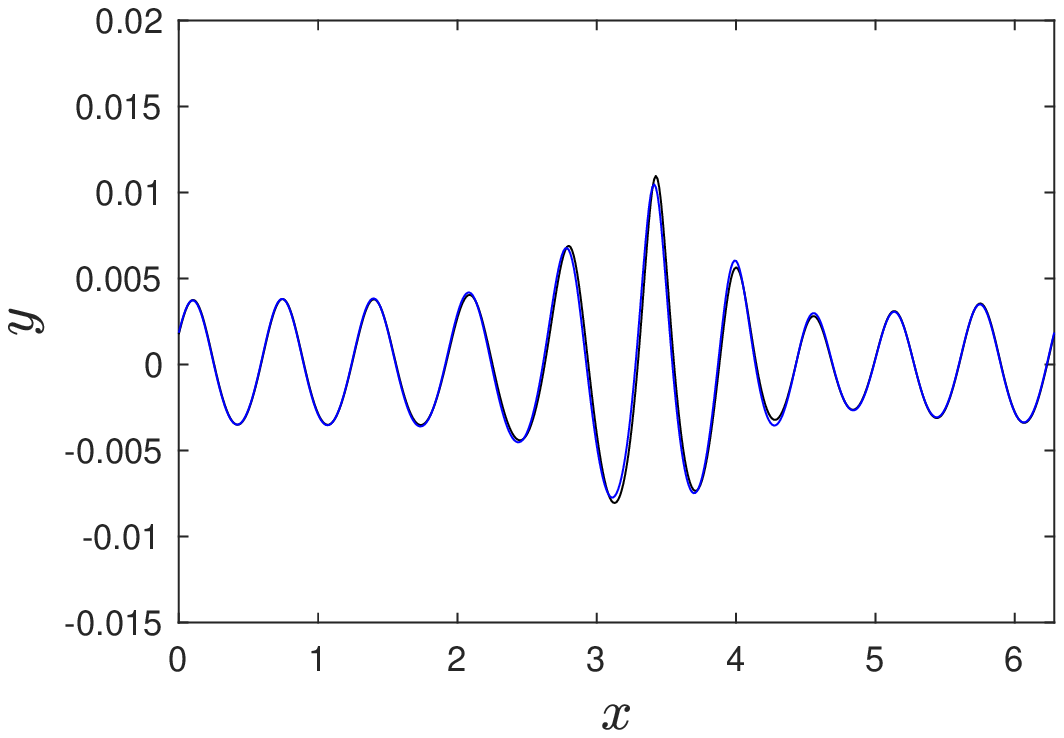}}
\hfill
\subfloat{\includegraphics[width=.33\linewidth]{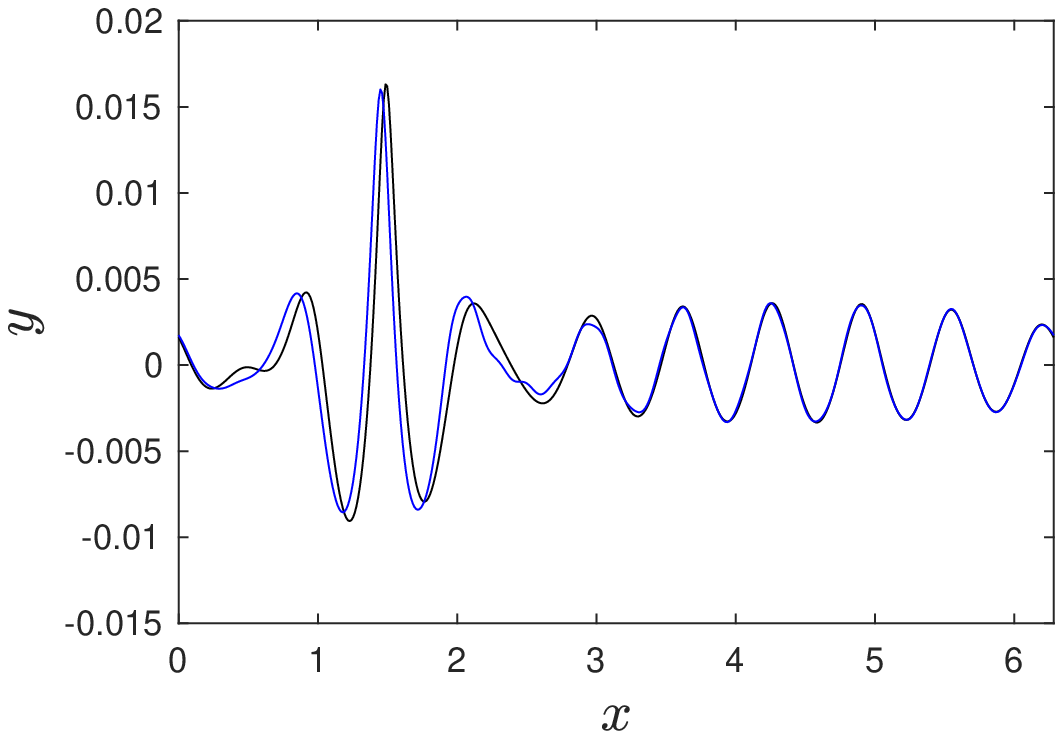}}
\hfill
\subfloat{\includegraphics[width=.33\linewidth]{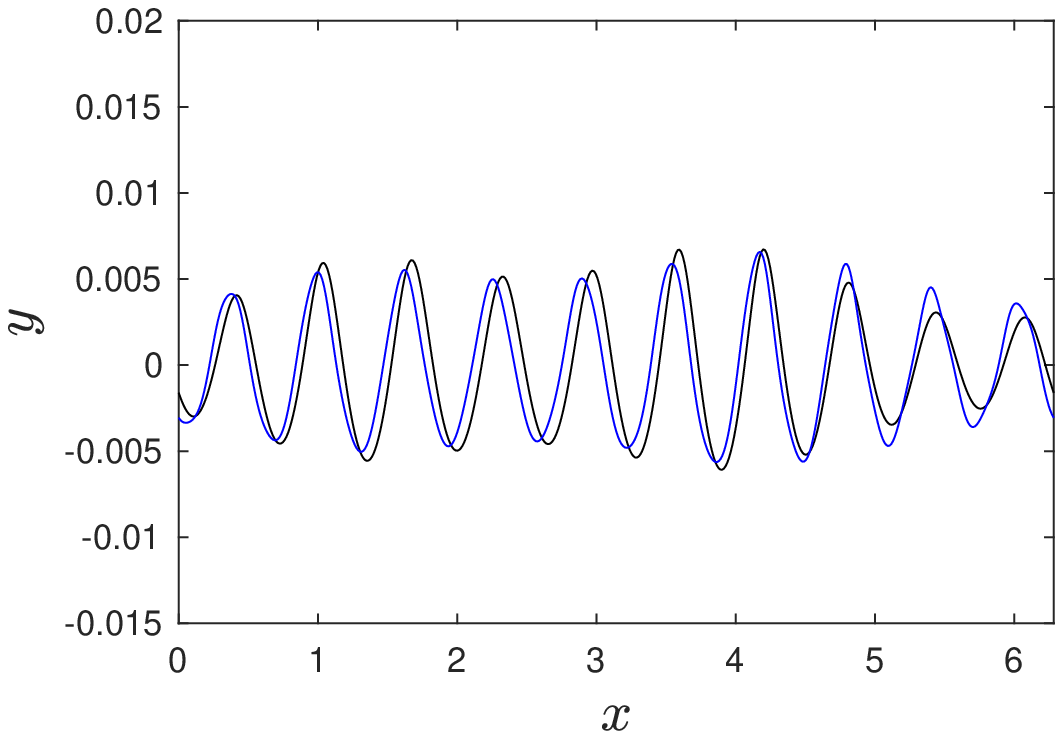}}
\hfill
\subfloat{\includegraphics[width=.33\linewidth]{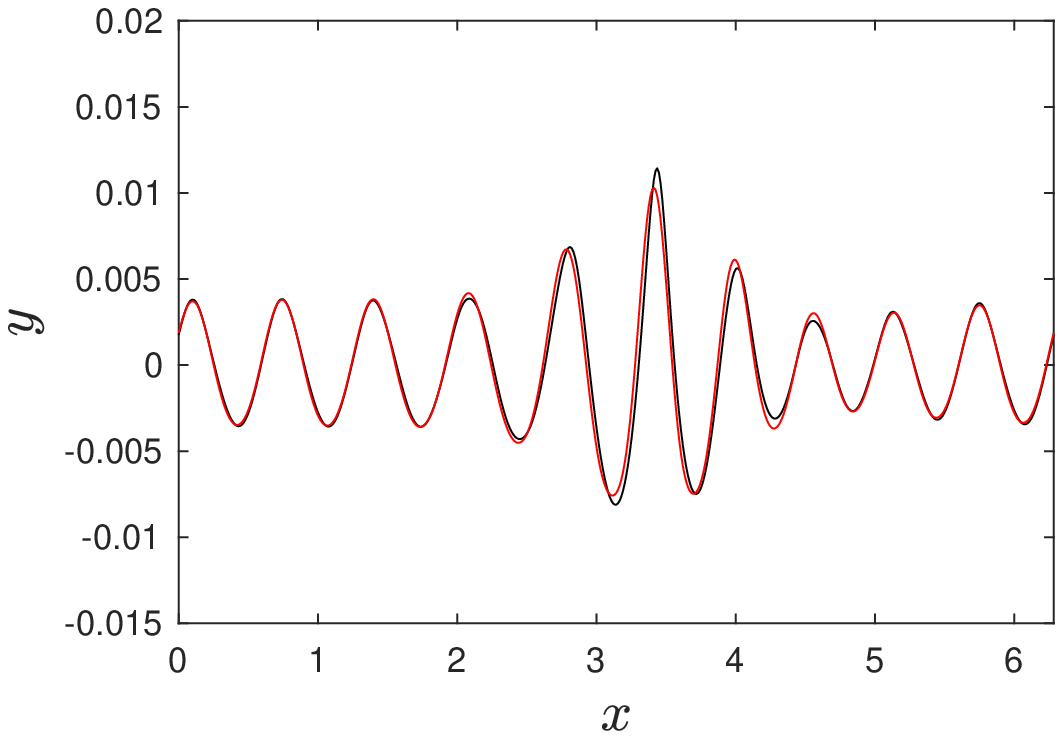}}
\hfill
\subfloat{\includegraphics[width=.33\linewidth]{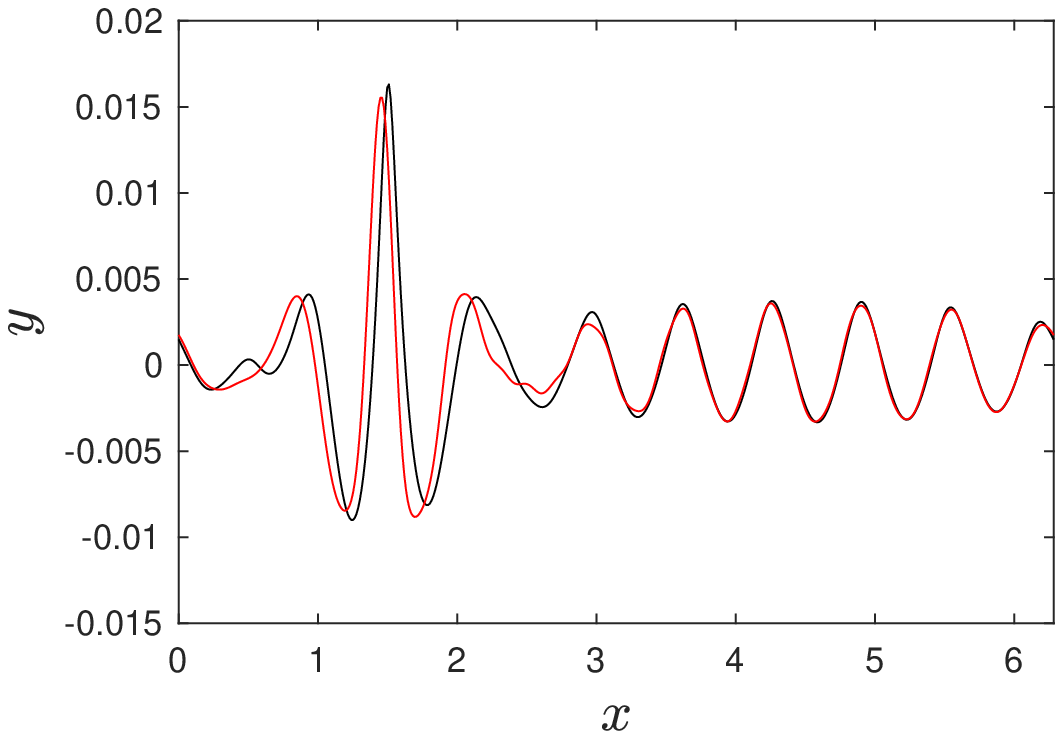}}
\hfill
\subfloat{\includegraphics[width=.33\linewidth]{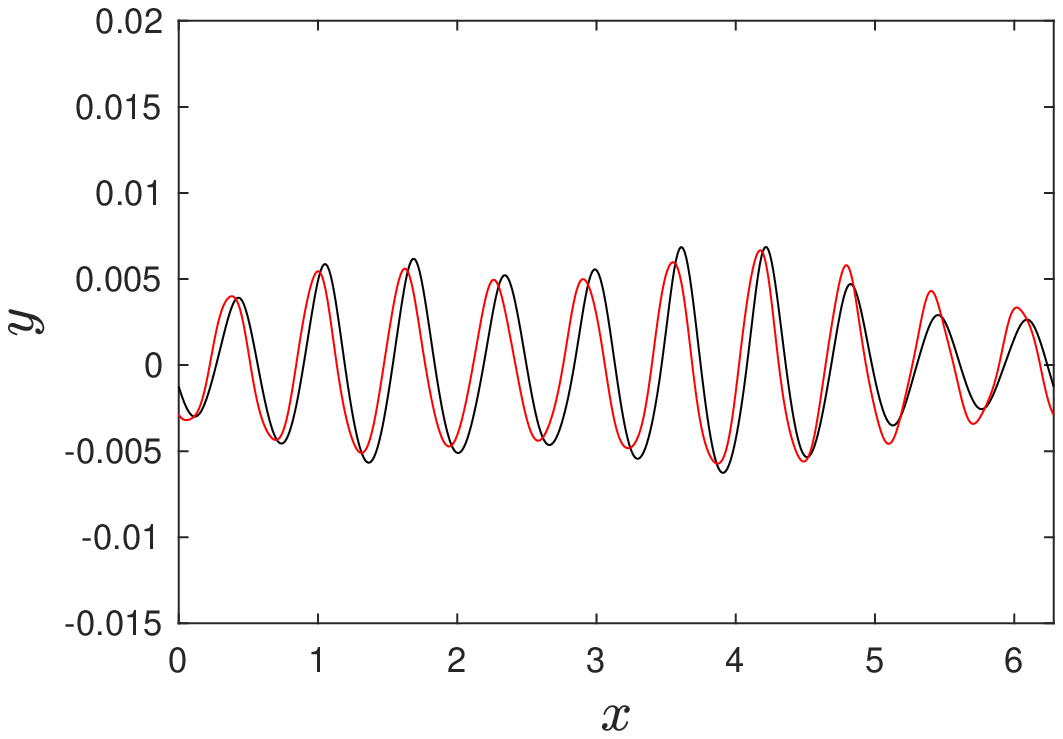}}
\caption{Comparison on $\eta$ between fully and weakly nonlinear solutions
for $(B_0,k_0,\lambda) = (0.002,10,1)$ and $\gamma = -2$ at $t = 390$, $500$, $1000$ (from left to right).
Upper panels: Hamiltonian Dysthe equation in blue.
Lower panels: classical Dysthe equation in red.
The black curve represents the full nonlinear system.}
\label{wavec_u0002_k10_wm2}
\end{figure}

\begin{figure}
\centering
\subfloat{\includegraphics[width=.5\linewidth]{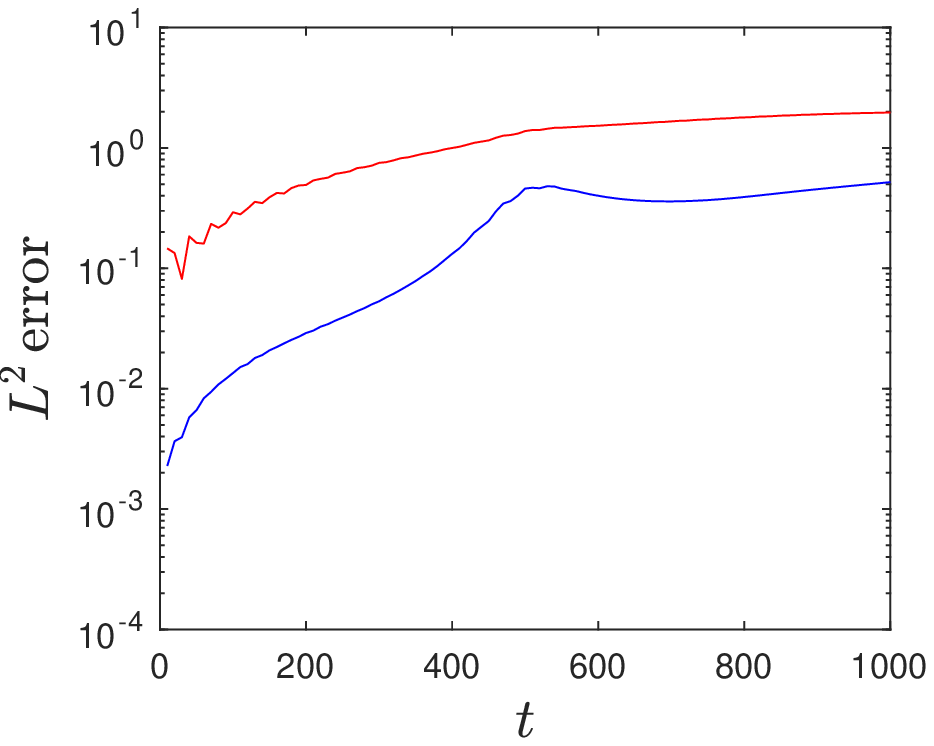}}
\hfill
\subfloat{\includegraphics[width=.5\linewidth]{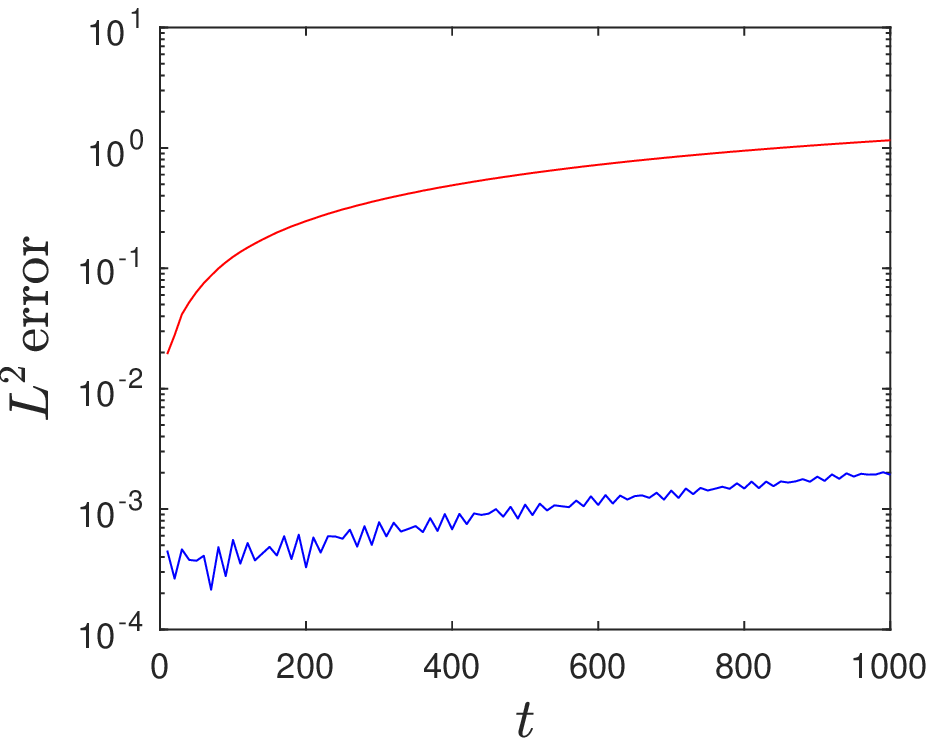}}
\caption{Relative $L^2$ errors on $\eta$ between fully and weakly nonlinear solutions
for $(B_0,k_0,\lambda) = (0.002,10,1)$.
The blue curve represents the Hamiltonian Dysthe equation with full reconstruction,
while the red curve represents the Hamiltonian Dysthe equation with partial reconstruction.
Left panel: $\gamma = -2$. Right panel: $\gamma = +2$.}
\label{L2comp1st_u0002_k10}
\end{figure}

\begin{figure}
\centering
\subfloat{\includegraphics[width=.33\linewidth]{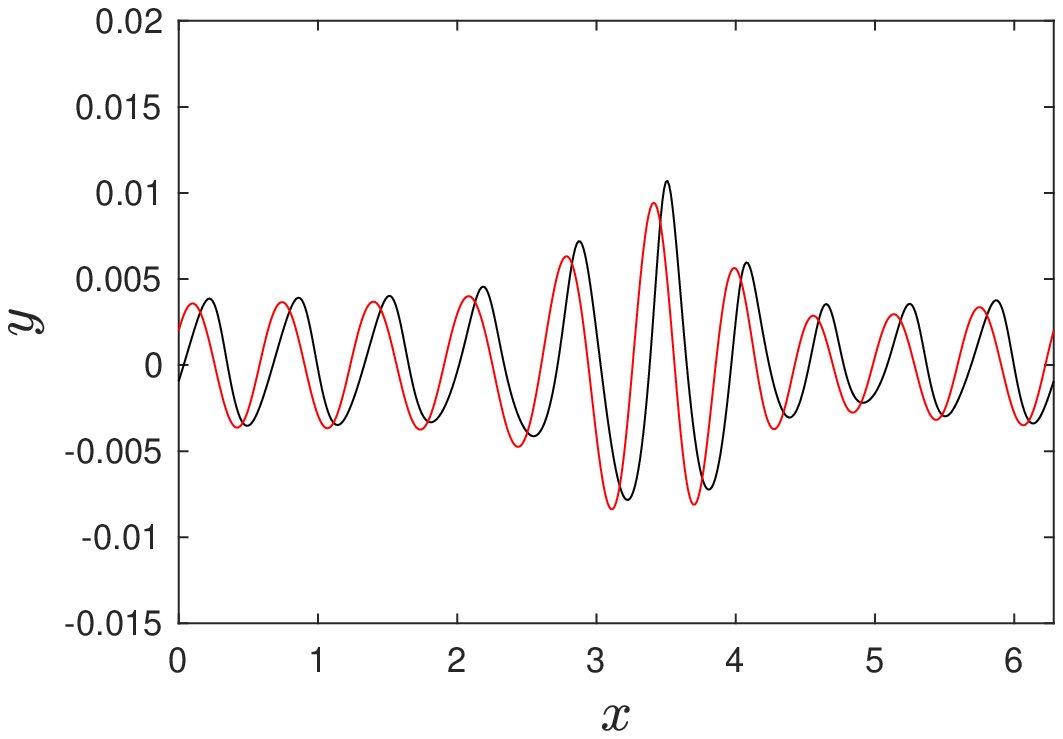}}
\hfill
\subfloat{\includegraphics[width=.33\linewidth]{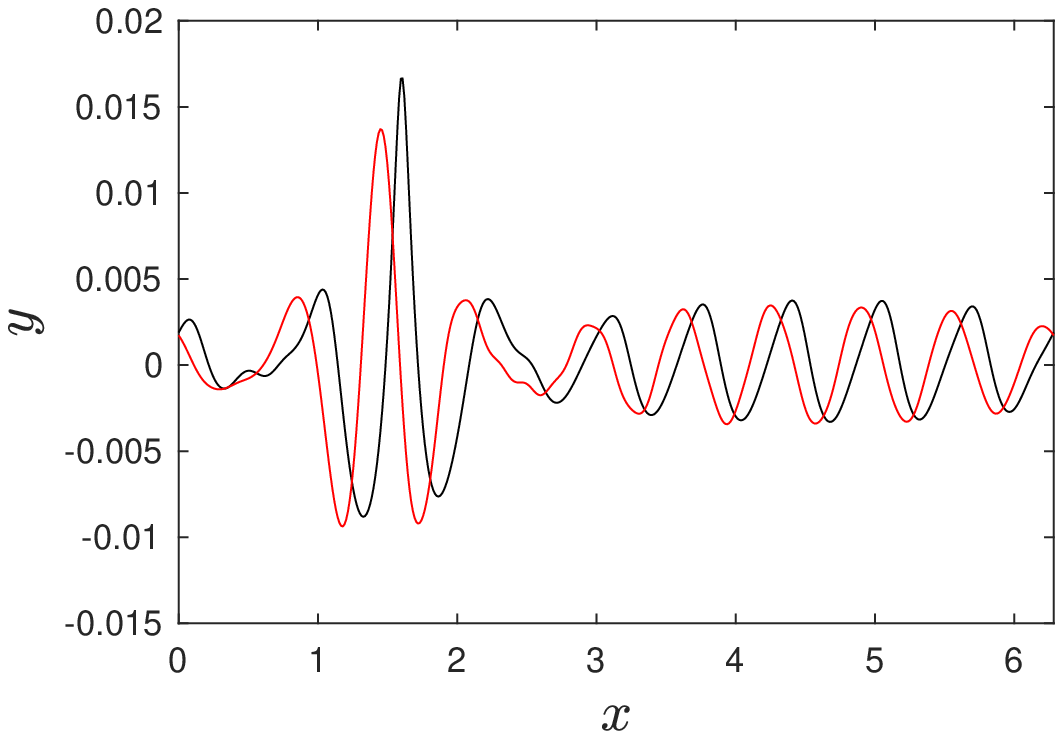}}
\hfill
\subfloat{\includegraphics[width=.33\linewidth]{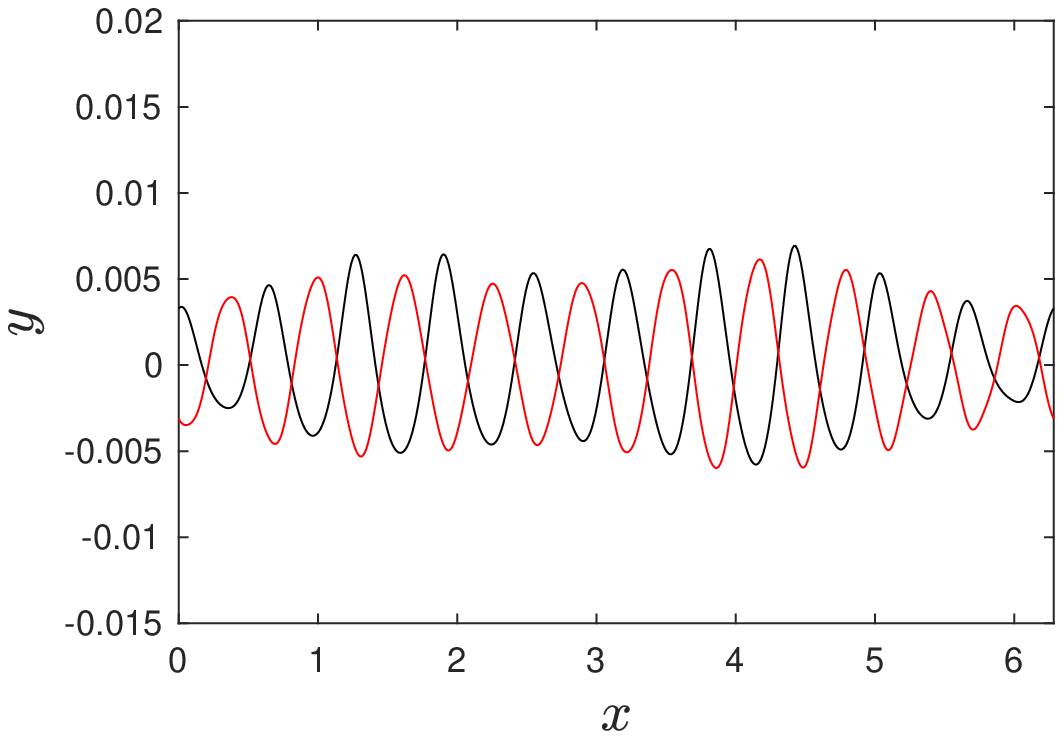}}
\hfill
\subfloat{\includegraphics[width=.33\linewidth]{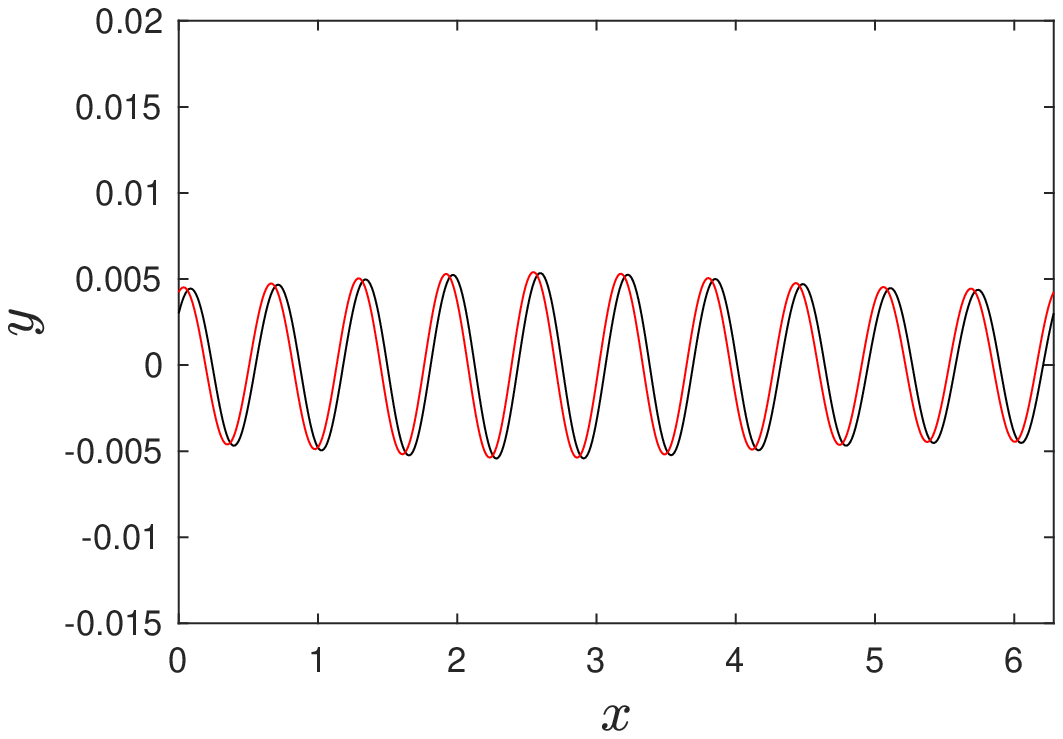}}
\hfill
\subfloat{\includegraphics[width=.33\linewidth]{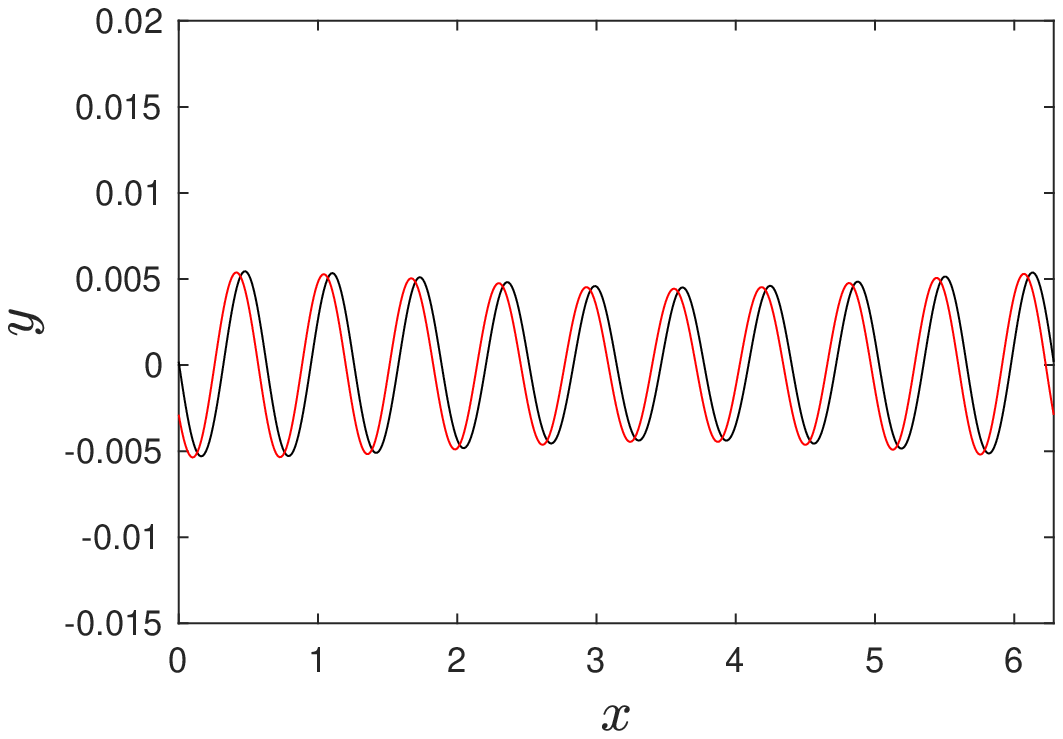}}
\hfill
\subfloat{\includegraphics[width=.33\linewidth]{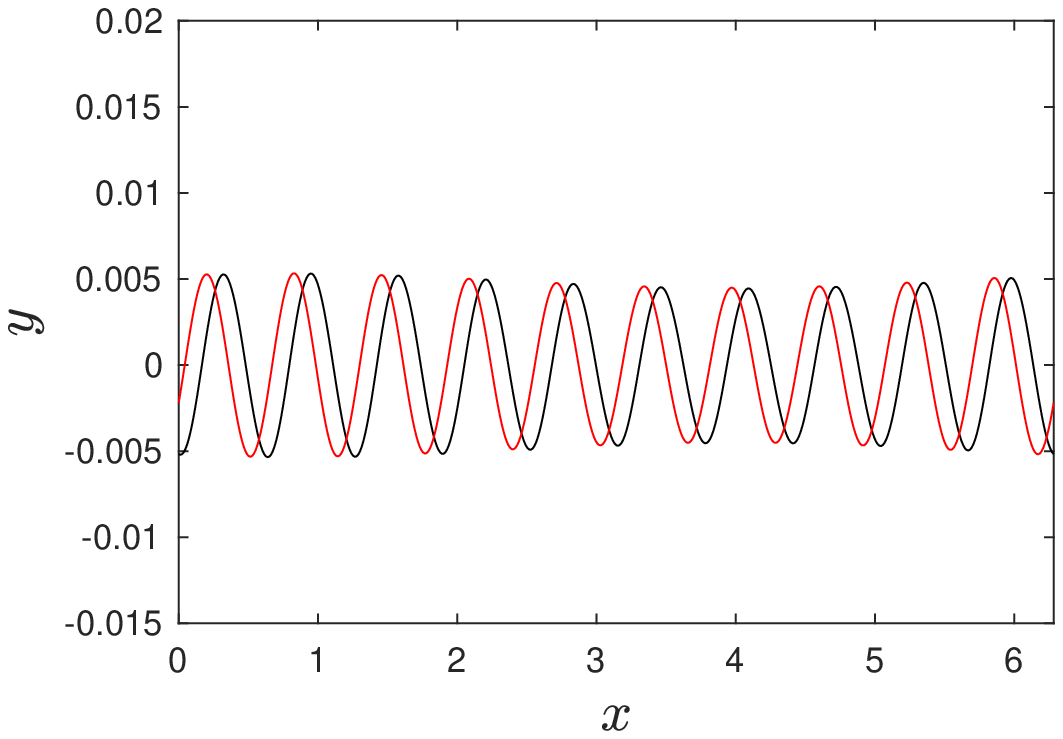}}
\caption{Comparison on $\eta$ between fully and weakly nonlinear solutions
for $(B_0,k_0,\lambda) = (0.002,10,1)$ at $t = 390$, $500$, $1000$ (from left to right).
The red curve represents the Hamiltonian Dysthe equation with partial reconstruction,
while the black curve represents the full nonlinear system.
Upper panels: $\gamma = -2$. Lower panels: $\gamma = +2$.}
\label{wave1st_u0002_k10_wm2}
\end{figure}

\begin{figure}
\centering
\subfloat{\includegraphics[width=.5\linewidth]{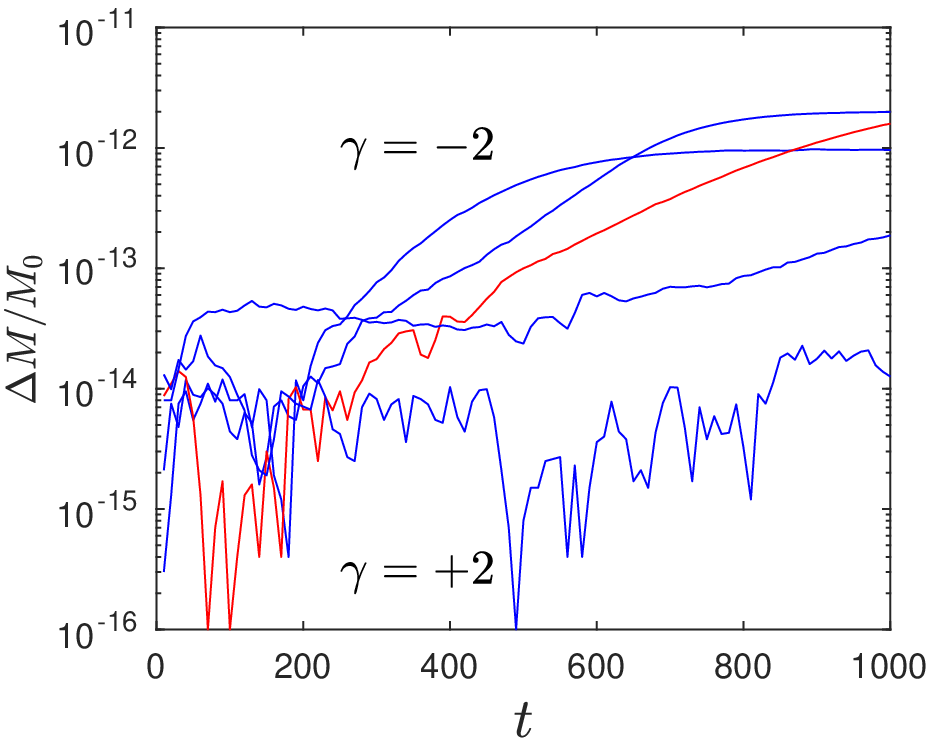}}
\hfill
\subfloat{\includegraphics[width=.5\linewidth]{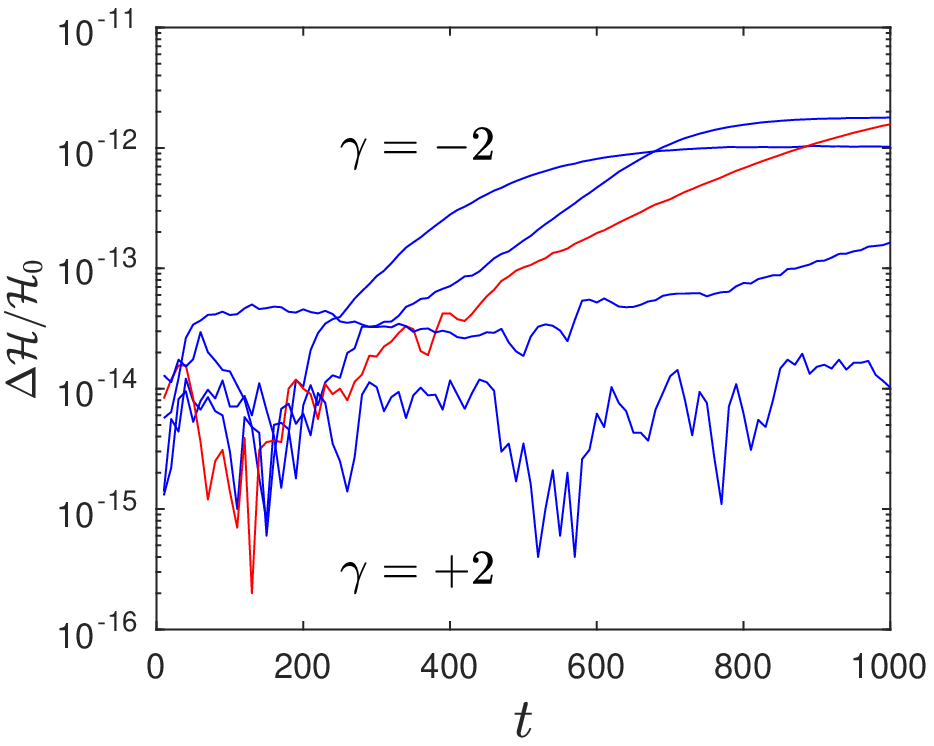}}
\caption{Relative errors on $M$ and $\calH$ for the Hamiltonian Dysthe equation 
with $(B_0,k_0,\lambda) = (0.002,10,1)$ and $\gamma = \{ -2, -1, 0, +1, +2 \}$.
The red curve corresponds to $\gamma = 0$.}
\label{ener_u0002_k10_wmp1_2}
\end{figure}

\end{document}